\documentclass{amsart}

\usepackage[letterpaper, margin=1.4in]{geometry}

\usepackage[backend=biber,style=alphabetic,sorting=nyt,minalphanames=10,maxalphanames=10,maxnames=15, doi=false,isbn=false,url=false,giveninits=true]{biblatex}
\renewbibmacro{in:}{} 
\usepackage{amssymb}
\usepackage{amssymb, latexsym, MnSymbol}
\usepackage{color, xcolor}
\usepackage{amscd, graphicx, psfrag, tikz}
\usetikzlibrary{calc,patterns}
\usepackage[mathscr]{euscript}
\usepackage[all]{xy}
\usepackage{mathtools}

\usepackage{enumitem}
\usepackage{stmaryrd}
\usepackage{url,hyperref}

\hypersetup{
    hidelinks,
    colorlinks=false,
    linkcolor=black,
    citecolor=black,
    urlcolor=black
}

\usepackage{comment}

\newcommand{\Ga}{\Gamma}
\newcommand{\Si}{ {\Sigma} }

\newcommand{\si}{ {\sigma} }

\newcommand{\bC}{ {\mathbb{C}} }

\newcommand{\bK}{\mathbb{K}}
\newcommand{\bL}{\mathbb{L}}

\newcommand{\bP}{\mathbb{P}}
\newcommand{\bQ}{\mathbb{Q}}
\newcommand{\bR}{\mathbb{R}}

\newcommand{\bT}{\mathbb{T}}

\newcommand{\bZ}{\mathbb{Z}}

\newcommand{\cA}{\mathcal{A}}
\newcommand{\cB}{\mathcal{B}}

\newcommand{\cE}{\mathcal{E}}

\newcommand{\cI}{\mathcal{I}}

\newcommand{\cL}{\mathcal{L}}
\newcommand{\cM}{\mathcal{M}}
\newcommand{\cO}{\mathcal{O}}

\newcommand{\cS}{\mathcal{S}}
\newcommand{\cT}{\mathcal{T}}
\newcommand{\cX}{\mathcal{X}}
\newcommand{\cV}{\mathcal{V}}

\newcommand{\age}{\mathrm{age}}
\newcommand{\Aut}{\mathrm{Aut}}
\newcommand{\CR}{ {\mathrm{CR}} }

\newcommand{\Hom}{\mathrm{Hom}}
\newcommand{\End}{\mathrm{End}}

\newcommand{\Ker}{\mathrm{Ker}}
\newcommand{\Res}{\mathrm{Res}}
\newcommand{\Spec}{\mathrm{Spec}}

\renewcommand{\Re}{\mathrm{Re}}

\newcommand{\can}{ {\mathrm{can}} }
\newcommand{\eff}{ {\mathrm{eff}} }
\newcommand{\ext}{ {\mathrm{ext}} }
\newcommand{\ev}{\mathrm{ev}}
\newcommand{\val}{ {\mathrm{val}} }
\newcommand{\vir}{ {\mathrm{vir}} }

\newcommand{\Int}{\mathrm{Int}}
\newcommand{\pt}{\mathrm{pt}}

\newcommand{\Nef}{{\mathrm{Nef}}}
\newcommand{\NE}{{\mathrm{NE}}}

\renewcommand{\Box}{\mathrm{Box}}

\newcommand{\Hol}{{\mathrm{Hol}}}
\newcommand{\sgn}{ {\mathrm{sgn}}}

\newcommand{\bff}{\mathbf{f}}

\newcommand{\bu}{\mathbf{u}}

\newcommand{\one}{\mathbf{1}}
\newcommand{\bmu}{\boldsymbol{\mu}}
\newcommand{\btau}{\boldsymbol{\tau}}
\newcommand{\bsi}{{\boldsymbol{\si}}}
\newcommand{\brho}{{\boldsymbol{\rho}}}

\newcommand{\bfeta}{{\boldsymbol{\eta}}}
\newcommand{\bSi}{\mathbf{\Si}}

\newcommand{\fg}{\mathfrak{g}}
\newcommand{\fl}{\mathfrak{l}}
\newcommand{\fm}{\mathfrak{m}}

\newcommand{\fn}{\mathfrak{n}}
\newcommand{\fp}{\mathfrak{p}}
\newcommand{\fr}{\mathfrak{r}}
\newcommand{\fs}{\mathfrak{s}}

\newcommand{\w}{\mathsf{w}}
\newcommand{\su}{\mathsf{u}}
\newcommand{\sv}{\mathsf{v}}
\newcommand{\sw}{\mathsf{w}}

\newcommand{\sfa}{\mathsf{a}}
\newcommand{\sfb}{\mathsf{b}}
\newcommand{\sff}{\mathsf{f}}
\newcommand{\sP}{\mathsf{P}}
\newcommand{\Crit}{\mathsf{Crit}}

\newcommand{\hX}{\hat{X}}

\newcommand{\hH}{\hat{H}}

\newcommand{\hx}{\hat{x}}
\newcommand{\hy}{\hat{y}}
\newcommand{\hxi}{\hat{\xi}}
\newcommand{\htheta}{\hat{\theta}}

\newcommand{\txi}{\widetilde{\xi}}

\newcommand{\tbeta}{\widetilde{\beta}}

\newcommand{\tL}{\widetilde{L}}
\newcommand{\tM}{\widetilde{M}}
\newcommand{\tN}{\widetilde{N}}
\newcommand{\tQ}{\widetilde{Q}}

\newcommand{\tS}{\widetilde{S}}

\newcommand{\tX}{\widetilde{X}}

\newcommand{\tb}{\widetilde{b}}

\newcommand{\tmu}{\widetilde{\mu}}

\newcommand{\tbT}{\widetilde{\bT}}

\newcommand{\tNef}{\widetilde{\Nef}}
\newcommand{\tNE}{\widetilde{\NE}}

\newcommand{\vGa}{\vec{\Gamma}}
\newcommand{\bGa}{\mathbf{\Gamma}}
\newcommand{\Mbar}{\overline{\cM}}

\newcommand{\spa}{ {\ \ \,} }

\newcommand{\Cbar}{\bar{C}}
\newcommand{\Dbar}{\bar{D}}
\newcommand{\Hbar}{\bar{H}}

\newcommand{\Zbar}{\bar{Z}}

\newcommand{\lambdabar}{\bar{\lambda}}
\newcommand{\etabar}{\bar{\eta}}

\newcommand{\eps}{\varepsilon}

\newcommand{\ST}{ {S_{\bT'}} }
\newcommand{\RT}{ {R_{\bT'}} }
\newcommand{\bST}{ {\bar{S}_{\bT'}} }
\newcommand{\bSTQ}{ \bST[\![ \tQ,\tau'' ]\!] }

\newcommand{\bSt}{ {\bar{S}_{\bT'}} }

\newcommand{\nov}{\Lambda_{\mathrm{nov}}}
\newcommand{\novT}{\bar{\Lambda}^{\bT'}_{\mathrm{nov}} }

\newcommand{\XX}{X_{\bff}}
\newcommand{\YY}{Y_{\bff}}

\newcommand{\chF}{\check{F}}
\newcommand{\chR}{\check{R}}

\newcommand{\chN}{\check{N}}

\newcommand{\inner}[1]{\langle  #1 \rangle}
\newcommand{\ceil}[1]{\lceil  #1 \rceil}
\newcommand{\floor}[1]{\lfloor  #1 \rfloor}

\newcommand{\double}[1]{\left\llangle #1 \right\rrangle}

\newtheorem{dummy}{dummy}[section]
\newtheorem{lemma}[dummy]{Lemma}
\newtheorem{theorem}[dummy]{Theorem}

\newtheorem{proposition}[dummy]{Proposition}
\newtheorem{remark}[dummy]{Remark}
\newtheorem{definition}[dummy]{Definition}
\newtheorem{example}[dummy]{Example}

\newtheorem{convention}[dummy]{Convention}

\newtheorem{assumption}[dummy]{Assumption}

\begin{document}

\title[All-genus open mirror symmetry for toric CY 3-orbifolds with multiple branes]{All-genus open mirror symmetry for toric Calabi-Yau 3-orbifolds with multiple branes}

\author{Bohan Fang}
\address{Bohan Fang, Beijing International Center for Mathematical Research, Peking University, 5 Yiheyuan Road, Beijing 100871, China}
\email{bohanfang@gmail.com}

\author{Chiu-Chu Melissa Liu}
\address{Chiu-Chu Melissa Liu, Department of Mathematics, Columbia University, 2990 Broadway, New York, NY 10027}
\email{ccliu@math.columbia.edu}

\author{Song Yu}
\address{Song Yu, Department of Mathematics, California Institute of Technology, 1200 E California Blvd, Pasadena, CA 91125}
\email{songyu@caltech.edu}

\author{Zhengyu Zong}
\address{Zhengyu Zong, Department of Mathematical Sciences, Tsinghua University, Haidian District, Beijing 100084, China}
\email{zyzong@mail.tsinghua.edu.cn}

\begin{abstract}
We prove the all-genus open mirror symmetry of a toric Calabi-Yau 3-orbifold where the boundary condition is given by multiple Aganagic-Vafa branes. Each brane could be either inner or outer, and admit a fractional framing. The B-model is given by the Chekhov-Eynard-Orantin topological recursion expanded in local coordinates specified by the brane data.

\end{abstract}
\maketitle

\setcounter{tocdepth}{1}

\tableofcontents


\section{Introduction}\label{sect:Intro}

\subsection{Background and motivation}

Let $\cX$ be a Calabi-Yau 3-fold. The \emph{Gromov-Witten theory} of $\cX$ concerns the enumeration of maps of curves to $\cX$. It is established as a mathematical theory of (A-model) topological strings in physics and carries many delicate structures. The computation of the Gromov-Witten invariants, on the other hand, is often challenging, especially in higher genera. When $\cX$ is \emph{toric}, the technique of virtual localization \cite{GP99} reduces the computation to that of Hodge integrals. When $\cX$ is additionally \emph{smooth}, the \emph{topological vertex} \cite{AKMV05,LLLZ09} provides an efficient algorithm for the integrals and is equivalent to the Gromov-Witten/Donaldson-Thomas correspondence \cite{MNOP06,MNOP06b,MOOP11}. It gives a combinatorial formula for a generating function of all-genus Gromov-Witten invariants. The algorithm has been generalized to the case of transverse $A_n$-singularities \cite{BCY12,RZ13,RZ15,Ross14,Zong15} but not yet to general toric Calabi-Yau 3-orbifolds.

The Gromov-Witten vertex may be interpreted as a generating function of \emph{open} invariants with boundary conditions specified by a type of Lagrangian submanifolds/orbifolds called \emph{Aganagic-Vafa branes} \cite{AV00}. Open Gromov-Witten theory concerns the enumeration of maps of bordered Riemann surfaces and provides a mathematical theory of topological open strings.

The Gromov-Witten theory of $\cX$ may be studied, via mirror symmetry, by the complex geometry of its B-model mirror. When $\cX$ is a semi-projective toric Calabi-Yau 3-orbifold, the mirror can be taken to be a family of affine curves $C_q$ known as the \emph{mirror curve} \cite{HV00}. The genus-zero Gromov-Witten potential $F_0$ of $\cX$ may be recovered from integrals of 1-forms on $C_q$ along loops, while the disk potential $F_{0,1}$ of $\cX$ relative to an Aganagic-Vafa brane may be recovered from the Abel-Jacobi map of $C_q$ \cite{AV00,AKV02,FL13,FLT22}. Based on the work of Eynard-Orantin \cite{EO07} and Mari\~{n}o \cite{Marino08}, Bouchard-Klemm-Mari\~{n}o-Pasquetti \cite{BKMP09, BKMP10} proposed a new formalism of the higher-genus B-model in terms of the Chekhov-Eynard-Orantin \emph{topological recursion} invariants $\omega_{g,n}$ on the mirror curve and conjectured an all-genus open-closed mirror correspondence between $\omega_{g,n}$ and the Gromov-Witten potentials of $\cX$, known as the \emph{Remodeling Conjecture}. More precisely, in the open string sector, the generating function $F_{g,n}$ of open invariants of $\cX$ accounting for genus-$g$ domains with $n$ boundary components may be recovered from $\omega_{g,n}$. In the closed string sector, the genus-$g$ potential $F_g$ of $\cX$ may be recovered by the free energy $\check{F}_g= \omega_{g,0}$. The above all-genus open-closed mirror symmetry for $\cX = \bC^3$ was proved by \cite{BCMS13,Zhu15,Chen18,Zhou09a,Zhou09b}. Eynard-Orantin \cite{EO15} provided a proof for smooth toric Calabi-Yau 3-folds.  Fang-Liu-Zong \cite{flz2020affine,flz2020remodeling} proved the all-genus open-closed mirror symmetry for toric Calabi-Yau 3-orbifolds where in the open sector, the boundary condition is given by a single Aganagic-Vafa outer brane. In \cite{FLYZ25}, the authors formulated and proved a version of all-genus mirror symmetry for the equivariant descendant Gromov-Witten theory of toric Calabi-Yau 3-orbifolds that respects the integral structures.

In this paper, we prove the general case of the all-genus open mirror symmetry for toric Calabi-Yau 3-orbifolds where the boundary condition can be given by any collection of multiple Aganagic-Vafa branes, each of which can be outer or inner and can admit a fractional framing. We now state and discuss our main result in detail.

\subsection{Main result}

Let $\cX$ be a semi-projective toric Calabi-Yau 3-orbifold defined by a 3-dimensional simplicial fan $\Sigma$. An \emph{open phase} in $\cX$ is a pair $\bfeta = (\cL, \nabla)$ of an Aganagic-Vafa brane $\cL$ together with a flat $U(1)$-connection $\nabla$ on it (see Section \ref{sect:AVBranes}). The brane $\cL$ intersects a unique leg (or 1-dimensional torus orbit) $\fl_\tau$ in $\cX$, indexed by a 2-cone $\tau$ in $\Sigma$, and we say that $\cL$ (or $\bfeta$) is \emph{outer} if $\fl_\tau$ is noncompact and \emph{inner} otherwise. We have $\pi_1(\cL) \cong \bZ \times G_\tau$ where $G_\tau$ is the generic stabilizer group of $\fl_\tau$ and is finite cyclic. We say that $\cL$ is \emph{effective} if $|G_\tau| = 1$ and \emph{ineffective} otherwise (in which case $\fl_\tau$ is referred to as a \emph{gerby} leg). The connection $\nabla$ determines an element $\eta \in G_\tau^*$ and we represent $\bfeta$ by the combinatorial data $(\tau, \eta)$. The set of all open phases will be denoted by $J_\Sigma$.

For $g \in \bZ_{\ge 0}$, $n \in \bZ_{\ge 1}$, and $\bfeta_1, \dots, \bfeta_n \in J_\Sigma$, consider the \emph{A-model open potential}
$$
    F_{g,n}(\btau; \tX_1, \dots, \tX_n) 
$$
which is a (scalar-valued) generating function of open Gromov-Witten invariants accounting for stable maps from bordered genus-$g$ Riemann surfaces with $n$ boundary components to $\cX$ where the $i$-th boundary component is constrained by the open phase $\bfeta_i$ (see Section \ref{sect:AmodelOpen}). It depends on the (extended) K\"ahler parameters $\btau$ of $\cX$ and A-model open moduli parameters $\tX_i$. It further depends on a choice of \emph{framing} which is a 1-dimensional subtorus of the Calabi-Yau torus of $\cX$ (see Section \ref{sect:Framing}). The open Gromov-Witten invariants are defined by virtual localization and restriction of the equivariant parameters to the framing torus.

The mirror curve $C_q$ of $\cX$ is a family of affine curves in $(\bC^*)^2$ parameterized by the B-model moduli parameters $q$, whose defining equation is specified by the fan $\Sigma$ (see Section \ref{sect:MirrorCurve}). An open phase $\bfeta \in J_\Sigma$ determines a connected region $U_{\bfeta} \subset C_q$ which is a small punctured disk in the outer case and a small annulus in the inner case.

Let $\omega_{g,n}$, $g \in \bZ_{\ge 0}$, $n \in \bZ_{\ge 1}$, be the Chekhov-Eynard-Orantin topological recursion invariant on $C_q$ defined by the input of two functions $\hx, \hy$ and the fundamental bidifferential of the second kind (see Section \ref{sect:TR}), which is a symmetric, meromorphic differential on $C_q^n$. Here, the function $\hx$ depends on the choice of framing above. For $\bfeta_1, \dots, \bfeta_n \in J_\Sigma$, the \emph{B-model open potential}
$$
    \chF_{g,n}(q; \hX_1, \dots, \hX_n) 
$$
is defined by the local expansion of $\omega_{g,n}$ on the local region $U_{\bfeta_1} \times \cdots \times U_{\bfeta_n} \subset C_q^n$ (see Section \ref{sect:BmodelOpen}). It depends on the B-model open moduli parameters $\hX_i$ which are local coordinates on (possibly finite covers of) the $U_{\bfeta_i}$'s respectively.

\begin{theorem}[={\bf Theorem \ref{thm:AllGenusMirror}}]\label{thm:Main}
For any $g \in \bZ_{\ge 0}$, $n \in \bZ_{\ge 1}$, and $\bfeta_1, \dots, \bfeta_n \in J_\Sigma$, we have
$$
    \chF_{g,n}(q; \hX_1, \dots, \hX_n) = (-1)^{g-1+n} F_{g,n}(\btau; \tX_1, \dots, \tX_n) 
$$
under the open-closed mirror map $\btau = \btau(q)$, $\tX_i = \tX_i(q, \hX_i)$ relating the A- and B-model moduli parameters.
\end{theorem}

\subsection{Remarks on the statement}\label{RemSec}

We remark on several aspects of generalities of Theorem \ref{thm:Main} as compared to previous works.

\subsubsection{Scalar-valued potentials for ineffective branes}\label{RemSec:ScalarValued}
Fang-Liu-Zong \cite{flz2020remodeling} considered open mirror symmetry where all the boundary components are constrained by the same Aganagic-Vafa outer brane $\cL$, intersecting a leg $\fl_\tau$. When $\cL$ is ineffective, the A- and B-model open potentials are defined as vector-valued generating functions where each component takes value in $H^*_{\CR}(\cB G_\tau;\bC) \cong \bC^{|G_\tau|}$. The A-model potential takes into account all possible monodromies at the boundary, indexed by elements of $G_\tau$, which is consistent with \cite{FLT22}. The B-model potential sums over all the $|G_\tau|$ local regions on the mirror curve corresponding to $\cL$, indexed by elements of $G_\tau^*$. In this paper, we consider scalar-valued potentials by including the data of $\eta \in G_\tau^*$ into the boundary conditions, or the open phases defined above. The data $\eta$ is naturally interpreted as the type of a flat $U(1)$-connection on the A-model side and corresponds to a \emph{connected} region on the B-model side. Our Theorem \ref{thm:Main} is thus a streamlined version of the vector-valued correspondence in \cite{flz2020remodeling} in the case of a single outer brane (and the case of disks in \cite{FLT22}). This definition of open phases also facilitates the study of the open Crepant Transformation Conjecture, as we will discuss in Section \ref{sect:OCTC}.

\subsubsection{Multiple branes}\label{RemSec:MultiBranes}
For open invariants with multiple boundary components, it is natural to allow different components to land on different Aganagic-Vafa branes in $\cX$, as studied in \cite{BKMP09,Zhou09b,EO15}. It is also required by the study of the open Crepant Transformation Conjecture, as we will see in Section \ref{sect:OCTC}, where the resolution of an ineffective brane would result in multiple branes. Theorem \ref{thm:Main} treats this general setup when $\cX$ is an orbifold. The key observation is that in defining the B-model open potential $\chF_{g,n}$, we may perform local expansions of different components of $\omega_{g,n}$ \emph{independently} on the different local regions $U_{\bfeta_i}$'s.

\subsubsection{Inner branes}\label{RemSec:InnerBranes}
For an outer phase $\bfeta$, the region $U_{\bfeta}$ is a small punctured disk centered at a puncture of the affine curve $C_q$, and the local expansion of a differential (e.g. a component of $\omega_{g,n}$) on $U_{\bfeta}$ may be defined by integration along a path starting at the puncture, as in \cite{flz2020remodeling}. For an inner phase, $U_{\bfeta}$ is a small annulus and the analog of the puncture is a node which only appears in the nodal degeneration of $C_q$ (cf. \cite[Section 5.4]{flz2020remodeling}). In this paper, we give a precise definition of local expansions on the annulus $U_{\bfeta}$ without passing to the nodal limit, thereby giving the definition of the B-model open potentials $\chF_{g,n}$ from $\omega_{g,n}$ when the open phases may be inner.

\subsubsection{Framing dependence and fractional framing}\label{RemSec:Framing}
Both the A- and B-model open potentials in Theorem \ref{thm:Main} depend on the choice of framing. For each open phase $\bfeta$, the framing associates a rational number to the Aganagic-Vafa brane $\cL$ which is more commonly known as the framing of the brane in the literature. Previous works on open mirror symmetry \cite{BKMP09,EO15,FLT22,flz2020remodeling} typically take an integral framing for a reference brane. In general, it is impossible to choose the framing torus so that the associated framings of all branes in $\cX$ are all integers, even in the simplest case $\cX = \bC^3$. Under the multiple-brane setup in this paper, we work in the general situation where the framings of the branes can be fractional. In this case, the B-model open moduli parameter $\hX$ is not directly a coordinate on the local region $U_{\bfeta}$ but a coordinate on a finite cover of $U_{\bfeta}$, where the degree of the cover is the denominator of the fraction. We handle the ambiguities arising from this issue.


\subsection{Remarks on the proof}
The proof of Theorem \ref{thm:Main} follows the overall strategy of \cite{flz2020remodeling}. The all-genus Gromov-Witten potentials of $\cX$ may be expressed by a graph sum formula via the Givental-Teleman quantization \cite{Givental01,Teleman12}, which was proved in \cite{Zong15} for general GKM orbifolds (see Theorem \ref{thm:Zong}). On the other hand, it was proved in \cite{DOSS14} that the topological recursion invariants $\omega_{g,n}$ can be expressed by an analogous graph sum formula (see Theorem \ref{thm:DOSS}). In \cite{flz2020remodeling}, the A- and B-model graph sum formulas are identified modulo a certain leaf factor, which is the case $(g,n)=(0,1)$ of disk potentials of Theorem \ref{thm:Main}.

Following \cite{AV00,AKV02,FL13}, Fang-Liu-Tseng \cite{FLT22} have established disk mirror symmetry for any Aganagic-Vafa brane $\cL$ in a toric Calabi-Yau 3-orbifold $\cX$, where $\cL$ may be effective or ineffective, outer or inner. They assumed $\cL$ to have integral framing, while for the purpose of proving Theorem \ref{thm:Main}, we need to allow fractional framings (see Section \ref{RemSec:Framing}). We first treat this generalization in our proof (see Theorem \ref{thm:DiskExpansion}).

\subsection{Future work}\label{sect:OCTC}
We remark on the application of Theorem \ref{thm:Main} to the open sector of the \emph{Crepant Transformation Conjecture} (CTC) \cite{Ruan02,Ruan06,BryanGraber-crepant,CoatesIritaniTseng-crepant, CoatesRuan13-crepant} for semi-projective toric Calabi-Yau 3-orbifolds. The conjecture relates the open Gromov-Witten theory of a pair of $K$-equivalent toric Calabi-Yau $3$-orbifolds $\cX_\pm$ relative to corresponding Aganagic-Vafa branes. Specifically, take a brane $\cL_-$ in $\cX_-$ whose generic stabilizer group has order $\fm$. If $\cL_-$ is disjoint from the exceptional locus of the crepant transformation, it corresponds to an isomorphic brane $\cL_+$ in $\cX_+$, and the $\fm$ associated open phases are naturally identified. Otherwise, it may be (partially) resolved into two or more branes in $\cX_+$ which together contribute $\fm$ open phases. 


\begin{example}\rm{
Consider the crepant resolution $\varphi \colon \cX_+ = \cA_1 \times \bC \to \cX_- = [\bC^2/\bmu_2]\times \bC$ and let $\cL_-$ be the ineffective outer brane in $\cX_-$ with $\fm = 2$ outer phases. The preimage of $\cL_-$ under $\varphi$ consists of two effective outer branes $\cL_+$, $\cL_{+,2}$ in $\cX_+$ which contribute two outer phases. See Figure \ref{fig:A1Sing} for an illustration. The open CTC relates the open potentials of $\cX_\pm$ with corresponding outer phases. 
}
\end{example}

\begin{figure}[htb]
$$
	\begin{tikzpicture}[scale=0.8]
		\draw (1,1) -- (2,0) -- (1,0) -- (1,2) -- (2,0);
		\draw[->] (3,1) -- (4,1);
		\draw (5.5,0) -- (5.5,2) -- (6.5,0) -- (5.5,0);
		\draw[ultra thick] (0.8, 1.4) -- (1.2, 1.4);
		\node at (0.4, 1.6) {$\cL_{+,2}$};
		\draw[ultra thick] (0.8, 0.4) -- (1.2, 0.4);
		\node at (0.4, 0.4) {$\cL_+$};
        \node at (1, -1) {$\cX_+ = \cA_1 \times \bC$};
        \draw[ultra thick] (5.3, 0.9) -- (5.7, 0.9);
		\node at (4.9, 1) {$\cL_-$};
        \node at (5.8, -1) {$\cX_- = [\bC^2/\bmu_2]\times \bC$};
    \end{tikzpicture}
$$
\caption{Crepant resolution of an $A_1$-singularity and an ineffective brane.}
\label{fig:A1Sing}
\end{figure}
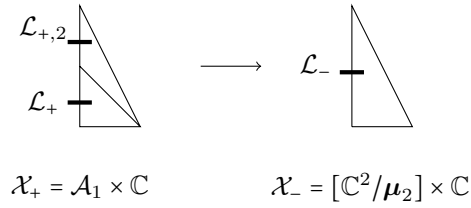

The pair $\cX_\pm$ give different phases of a \emph{global} stringy K\"ahler moduli and correspond to two limit points $\fs_\pm$ in the secondary variety $\mathcal M_B$. The mirror curves $C_{q_\pm}$ of $\cX_\pm$ are fibers of a global family of curves $\mathcal C$ over $\mathcal M_B$ near $\fs_\pm$ respectively. As analyzed in \cite{Yu25}, under a parallel transport from $C_{q_-}$ to $C_{q_+}$, the $\fm$ local regions mirror to the open phases associated to $\cL_-$ are in bijection with the $\fm$ local regions mirror to the induced open phases in $\cX_+$, and the precise bijection depends on the path of parallel transport. This leads to an identification of the B-model disk potentials under analytic continuation and thus the disk CTC. We note that \cite{Yu25} stated the disk CTC in terms of vector-valued potentials, as in \cite{FLT22,flz2020remodeling}, while the formulation of our Theorem \ref{thm:Main} provides a simpler statement in terms of scalar-valued potentials (see Section \ref{RemSec:ScalarValued}). The higher-genus open CTC may then be deduced from the case of disks in combination with the analytic continuation of the topological recursion invariants $\omega_{g,n}$. In particular, the modular transform in the higher-genus CTC naturally arises in the analytic continuation of the fundamental bidifferential $\omega_{0,2}$. We will investigate the all-genus CTC in the forthcoming work \cite{FLYZ-crepant}.

\subsection{Outline of the paper}
In Section \ref{sect:Geometry}, we review the geometry of toric Calabi-Yau 3-orbifolds and Aganagic-Vafa branes and discuss the framings of the branes. In Section \ref{sect:GW}, we define the all-genus open Gromov-Witten potentials in the generalities of Section \ref{RemSec} and give a graph sum formula for them. In Section \ref{sect:DiskMirror}, we study the geometry of mirror curves and generalize the disk mirror theorem to the case of fractional framings (Theorem \ref{thm:DiskExpansion}). In Section \ref{sect:AllGenus}, we define the all-genus B-model open potentials from topological recursion in the generalities of Section \ref{RemSec} and prove the all-genus open mirror symmetry (Theorem \ref{thm:Main}).

\subsection{Acknowledgments}
The authors would like to thank Jinghao Yu for helpful communications that inspired the proof of Proposition \ref{prop:AnnMirror}. The authors would like to thank Simons Center for Geometry and Physics for the support and hospitality during the program Recent Developments in Higher Genus Curve Counting in 2025 where important progress of this work was made.

The work of the first author is partially supported by National Key R\&D Program of China 2023YFA1009803, NSFC 12125101, NSFC 11890661 and NSFC 11831017. The work of the fourth author is partially supported by NSFC grant No. 11701315 and the Natural Science Foundation of Beijing, China grant No. 1252008.

\subsection*{AI disclaimer} 
No generative AI was used in the mathematical work or the writing of this paper.


\section{Open geometry on toric Calabi-Yau 3-orbifolds}\label{sect:Geometry}
In this section, we review the open geometry on toric Calabi-Yau 3-orbifolds relative to Aganagic-Vafa branes. We work over $\bC$.

\subsection{Toric Calabi-Yau 3-orbifolds}\label{sect:TCY3}
Let $X$ be a 3-dimensional simplicial toric variety. We assume that $X$ is \emph{Calabi-Yau}, i.e. $K_X \cong \cO_X$, and \emph{semi-projective}, i.e. the affinization map $X \to X_0 \coloneqq \Spec(H^0(X, \cO_X))$ is projective. Combinatorially, let $N \cong \bZ^3$ and let $\Sigma$ be the finite simplicial fan in $N_{\bR} \coloneqq N \otimes \bR$ associated to $X$. The dense algebraic torus acting on $X$ is $\bT \coloneqq N \otimes \bC^*$ and has cocharacter lattice $N$ and character lattice $\Hom(N, \bZ) =\colon M$. 

Let $\Sigma(d)$, $d = 0, 1, 2, 3$, denote the set of $d$-dimensional cones in $\Sigma$. We assume that every cone in $\Sigma$ is contained in some 3-cone. We set $\fp' \coloneqq |\Sigma(1)| - 3$ and list the rays in $\Sigma$ as $\Sigma(1) = \{\rho_1, \dots, \rho_{3+\fp'}\}$. For $i = 1, \dots, 3+\fp'$, let $b_i \in N$ be the primitive integral generator of the ray $\rho_i$. Then the Calabi-Yau condition implies that there exists $e_3^\vee \in M$ such that $\inner{e_3^\vee, b_i} = 1$ for all $i = 1, \dots, 3+\fp'$, where $\inner{-,-}$ denotes the natural pairing. Let $N' \coloneqq \ker (e_3^\vee\colon N \to \bZ)$ and $\bT' \coloneqq N' \otimes \bC^*$ be the Calabi-Yau 2-subtorus of $\bT$, whose cocharacter lattice is $N'$ and character lattice is $\Hom(N', \bZ) =\colon M'$.

We take a $\bZ$-basis $\{e_1^\vee, e_2^\vee, e_3^\vee\}$ of $M$ and let $\{e_1, e_2, e_3\}$ be the dual $\bZ$-basis of $N$. Let $(m_i, n_i, 1)$ be the coordinate of $b_i$ in this basis, and let $P$ be the convex hull in $N'_{\bR} \coloneqq N' \otimes \bR$ of all the points $(m_i, n_i)$. Then the semi-projectivity of $X$ implies that the support of $\Sigma$ is the cone over $P \times \{1\}$ in $N_{\bR}$.

Let $\cX$ be the \emph{canonical} toric Deligne-Mumford stack \cite{BCS05,FMN10} of the simplicial toric variety $X$, which has trivial generic stabilizers and is a toric Calabi-Yau 3-orbifold. It may be combinatorially characterized by the \emph{canonical} stacky fan $\bSi^\can = (N, \Sigma, (b_1, \dots, b_{3+\fp'}))$ of the simplicial fan $\Sigma$. Alternatively, we use the \emph{extended} stacky fan \cite{Jiang08} $\bSi^\ext = (N, \Sigma, (b_1, \dots, b_{3+\fp}))$ where $\fp \coloneqq |P \cap N'| -3$ and $b_{4+\fp'}, \dots, b_{3+\fp}$ is a listing of the additional lattice points in $(P \times \{1\}) \cap N$. We have a surjective homomorphism
$$
    \phi\colon \tN \coloneqq \bigoplus_{i=1}^{3+\fp} \bZ \tb_i \longrightarrow N, \qquad \tb_i \longmapsto b_i.
$$
Let $\bL \coloneqq \Ker(\phi) \cong \bZ^{\fp}$. We have the short exact sequence of lattices
\begin{equation}\label{eqn:NExactSeq}
    \xymatrix{
        0 \ar[r] & \bL \ar[r]^\psi & \tN \ar[r]^\phi & N \ar[r] & 0.
    }
\end{equation}

Given a cone $\sigma \in \Sigma(d)$, we introduce the index sets
$$
    I'_\sigma \coloneqq \{ i \in \{1, \dots, 3+\fp'\} \colon \rho_i \subseteq \sigma\}, \qquad I_\sigma \coloneqq \{1, \dots, 3+\fp\} \setminus I'_\sigma
$$
where $|I'_\sigma| = d$. Let $V(\sigma)$ denote the $\bT$-invariant closed subvariety of $X$ and $\cV(\sigma)$ denote the $\bT$-invariant closed substack of $\cX$, both having codimension $d$. 


\subsection{GIT quotient and symplectic quotient}
Tensoring \eqref{eqn:NExactSeq} with $\bC^*$ gives a short exact sequence of tori
\begin{equation}\label{eqn:TExactSeq}
    \xymatrix{
        1 \ar[r] & G \ar[r]^\psi & \tbT \ar[r]^\phi & \bT \ar[r] & 1
    }
\end{equation}
whose character lattices fit into the short exact sequence
$$    
    \xymatrix{
        0 \ar[r] & M \ar[r]^{\phi^\vee} & \tM \ar[r]^{\psi^\vee} & \bL^\vee \ar[r] & 0
    }
$$
obtained by applying $\Hom(-,\bZ)$ to \eqref{eqn:NExactSeq}. Consider the action of $G$ on $\tN \otimes \bC$ $\cong \bC^{3+\fp} = $ $ \Spec(\bC[Z_1, \dots, Z_{3+\fp}])$ via the inclusion $\psi$ above. In the linear subspace $\bC^{3+\fp'} = \Spec(\bC[Z_1, \dots, Z_{3+\fp'}])$, let $Z(\Sigma)$ be closed subvariety defined by the ideal generated by the monomials $\prod_{\rho_i \not \subseteq \sigma} Z_i$ for cones $\sigma$ in the fan $\Sigma$. Consider
$$
    U \coloneqq (\bC^{3+\fp'} \setminus Z(\Sigma)) \times (\bC^*)^{\fp - \fp'}
$$
which is a dense, $\tbT$-invariant, open subvariety of $\bC^{3+\fp}$. Then we have
$$
    X = U/G, \qquad \cX = [U/G].
$$
In particular, the above presentation of $\cX$ is a GIT quotient.

Consider the identifications $\tM \cong H^2_{\tbT}(\bC^{3+\fp}; \bZ)$ and  $\bL^\vee \cong H^2_G(\bC^{3+\fp}; \bZ)$. For $i = 1, \dots, 3+ \fp$, let $D_i^{\bT} \in \tM$ (resp. $D_i \in \bL^\vee$) be the $\tbT$-equivariant (resp. $G$-equivariant) Poincar\'e dual of the coordinate hyperplane $\{Z_i = 0\}$ in $\bC^{3+\fp}$. Then $\{D_1^{\bT}, \dots, D_{3+\fp}^{\bT}\}$ is the $\bZ$-basis of $\tM$ dual to the basis $\{\tb_1, \dots, \tb_{3+\fp}\}$ of $\tN$, and we have $\psi^\vee(D_i^{\bT}) = D_i$ for all $i$. The restriction to the open subset $U \subseteq \bC^{3+\fp}$ induces a surjective homomorphism
\begin{equation}\label{eqn:Kirwan}
    \kappa\colon \bL^\vee \cong H^2_G(\bC^{3+\fp}; \bZ) \longrightarrow H^2_G(U; \bZ) \cong H^2(\cX; \bZ)
\end{equation}
whose kernel is $\bigoplus_{i = 4+\fp'}^{3+\fp} \bZ D_i$. For $i = 1, \dots, 3 + \fp'$, let $\Dbar_i \coloneqq \kappa(D_i)$; when $i \ge 4$, let $\Dbar_i^{\bT'} \in H^2_{\bT'}(\cX; \bZ)$ be the unique lift of $\Dbar_i$ whose restriction to $\fp_{\si_0}$ is zero.

For $\sigma \in \Sigma(3)$, let
$
    \bK_\sigma^\vee \coloneqq \bigoplus_{i \in I_\sigma} \bZ D_i,
$
and define the \emph{extended $\sigma$-nef cone} to be
$$
    \tNef_\sigma \coloneqq \sum_{i \in I_\sigma} \bR_{\ge 0} D_i
$$
which is a top-dimensional cone in $\bL^\vee_{\bR} \coloneqq \bL^\vee \otimes \bR$. The \emph{extended nef cone} of $\cX$ is
$$
    \tNef(\cX) \coloneqq \bigcap_{\sigma \in \Sigma(3)} \tNef_\sigma
$$
whose interior $C(\cX)$ is the \emph{extended K\"ahler cone} of $\cX$. 

Let $G_{\bR} \cong U(1)^{\fp}$ be the maximal compact subgroup of $G$, whose dual Lie algebra is canonically identified with $\bL^\vee_{\bR}$. The action of $G$ on the symplectic manifold $\left(\bC^{3+\fp}, \sqrt{-1} \sum_{i = 1}^{3 + \fp} dZ_i \wedge d\Zbar_i \right)$ restricts to a Hamiltonian action by $G_{\bR}$ with moment map
$$
    \tmu\colon \bC^{3+\fp} \longrightarrow \bL^\vee_{\bR}.
$$
Let $\zeta \in C(\cX)$ be an extended K\"ahler class. Then the symplectic quotient $[\tmu^{-1}(\zeta)/G_{\bR}]$ is a K\"ahler orbifold equipped with a symplectic structure $\omega(\zeta)$ and is isomorphic to $\cX$ as a complex orbifold.

Let $\bT_{\bR}$ (resp. $\bT'_{\bR}$) be the maximal compact subgroup of $\bT$ (resp. $\bT'$), whose dual Lie algebra is canonically identified with $M_{\bR} \coloneqq M \otimes \bR$ (resp. $M'_{\bR} \coloneqq M' \otimes \bR$). The actions of $\bT$, $\bT'$ on the K\"ahler orbifold $(\cX, \omega(\zeta))$ restricts to Hamiltonian actions by $\bT_{\bR}, \bT'_{\bR}$ with moment maps
$$
    \mu_{\bT_{\bR}}\colon \cX \longrightarrow M_{\bR}, \qquad \mu_{\bT'_{\bR}}: \cX \longrightarrow M'_{\bR}.
$$
The \emph{toric graph} of $\cX$ is the image $\mu_{\bT'_{\bR}}(\cX^1)$ in the plane $M'_{\bR}$ of the union $\cX^1$ of $0$- and $1$-dimensional orbits in $\cX$.

\subsection{Stabilizers and inertia}

For a cone $\sigma$ in $\Sigma$, let $G_\sigma$ denote the group of generic stabilizers of the closed substack $\cV(\sigma)$, which is a finite subgroup of $G$, and let $G^*_\sigma \coloneqq \Hom(G_\sigma, \bC^*)$. 
As a set, $G_\sigma$ is in bijection with
$$
    \Box(\sigma) \coloneqq \left\{v \in N : v = \sum_{i \in I'_\sigma} c_i(v) b_i \text{ for some } c_i(v) \in [0,1) \cap \bQ \right\}.
$$
We define $c_i(v)$ by the above equation; note that $c_i(v) = 0$ for all $i \in I_\sigma$. 

Let $\bL_{\bQ} \coloneqq \bL \otimes \bQ$ and $\bL_{\bR} \coloneqq \bL \otimes \bR$. For $\sigma \in \Sigma(3)$, let
$
    \bK_\sigma \coloneqq \{\beta \in \bL_{\bQ} : \inner{D, \beta} \in \bZ \text{ for all } D \in \bK_\sigma^\vee\}
$
be the dual lattice of $\bK_\sigma^\vee$, where $\inner{-,-}$ denotes the natural pairing. Define the \emph{extended $\sigma$-Mori cone} to be
$$
    \tNE_\sigma \coloneqq \{\beta \in \bL_{\bR} : \inner{D, \beta} \ge 0 \text{ for all } D \in \tNef_\sigma\} 
$$
which is the dual cone of $\tNef_\sigma$. Let
$
    \bK_{\eff, \sigma} \coloneqq \bK_\sigma \cap \tNE_\sigma.
$
The \emph{extended Mori cone} of $\cX$ is
$$
    \tNE(\cX) \coloneqq \bigcup_{\sigma \in \Sigma(3)} \tNE_\sigma.
$$
Moreover, let
$$
    \bK \coloneqq \bigcup_{\sigma \in \Sigma(3)} \bK_\sigma, \qquad \bK_{\eff} \coloneqq \bK \cap \tNE(\cX) = \bigcup_{\sigma \in \Sigma(3)} \bK_{\eff, \sigma}.
$$

For $\sigma \in \Sigma(3)$, there is an identification $G_\sigma\cong \bK_\sigma / \bL$ and a bijection $\bK_\sigma / \bL \to \Box(\sigma)$ given by
$$
    v\colon \bK_\sigma \longrightarrow \Box(\sigma), \qquad \beta \longmapsto \sum_{i \in I'_\sigma}  \{-\inner{D_i, \beta}\} b_i.
$$
Here, $\{x\} = x - \floor{x}$ denotes the fractional part of $x \in \bR$. In particular, we have $c_i(v(\beta)) = \{-\inner{D_i, \beta}\}$. The above bijection restricts to a bijection $G_\tau \to \Box(\tau)$ for any $\tau \subseteq \sigma$.

The \emph{inertia stack} $\cI\cX$ of $\cX$ has connected components indexed by the set
$$
    \Box(\cX) \coloneqq \bigcup_{\sigma \in \Sigma(3)} \Box(\sigma).
$$
Let $\cX_v$ denote the component indexed by $v \in \Box(\cX)$. The \emph{age} of $\cX_v$ is defined to be
$$
    \age(v) \coloneqq \sum_{i = 1}^{3+\fp'} c_i(v) \quad \in \{0, 1, 2\}.
$$




\subsection{Equivariant Chen-Ruan cohomology}

The \emph{Chen-Ruan cohomology} \cite{CR02} of the toric orbifold $\cX$ (with $\bC$-coefficients) is
$$
    H^*_{\CR}(\cX; \bC) = \bigoplus_{v \in \Box(\cX)} H^*(\cX_v; \bC)[2\age(v)]
$$
where $[2\age(v)]$ denotes the degree shift by $2\age(v)$. For $v \in \Box(\cX)$, let $\one_v$ denote the unit of $H^*(\cX_v; \bC)$. We have
\begin{align*}
    & \dim_{\bC} H^0_{\CR}(\cX; \bC)  = 1, && \dim_{\bC} H^2_{\CR}(\cX; \bC) = \fp = \fg + \fn - 3,\\
    & \dim_{\bC} H^4_{\CR}(\cX; \bC) = \fg, && \dim_{\bC} H^*_{\CR}(\cX; \bC) = 1 + \fp + \fg = 2\fg - 2 + \fn
\end{align*}
where $\fg \coloneqq |\Int(P) \cap N'|$ (resp. $\fn \coloneqq | \partial P \cap N'|$) is the number of interior (reps. boundary) lattice points in $P$.

We also consider the $\bT'$-equivariant setting. Let $\su_1, \su_2$ be characters of $\bT'$ corresponding to the images of $e_1^\vee, e_2^\vee$ under the projection $M \to M'$. We use $\su_1, \su_2$ as equivariant parameters of $\bT'$ and view them as elements of $M'$. We have
$$
    M' = \bZ\su_1 \oplus \bZ\su_2, \qquad H^*_{\bT'}(\pt) = \bZ[\su_1, \su_2].
$$
Let
$$
    R_{\bT'} \coloneqq H^*_{\bT'}(\pt; \bC) = \bC[\su_1, \su_2], \qquad
    S_{\bT'} \coloneqq \bC(\su_1, \su_2).
$$
The $\bT'$-equivariant Chen-Ruan cohomology of $\cX$ (with $\bC$-coefficients) is
$$
    H^*_{\CR, \bT'}(\cX; \bC) = \bigoplus_{v \in \Box(\cX)} H^*_{\bT'}(\cX_v; \bC)[2\age(v)].
$$
Together with the $\bT'$-equivariant Chen-Ruan cup product $\star_{\cX}$ and the $\bT'$-equivariant Poincar\'e pairing, the extension $H^*_{\CR, \bT'}(\cX; \bC) \otimes_{R_{\bT'}} S_{\bT'}$ is a Frobenius algebra over $S_{\bT'}$.


For $\sigma \in \Sigma(3)$, let $\cX_\sigma \cong [\bC^3/G_\sigma]$ be the associated affine toric Calabi-Yau 3-orbifold which is an open substack of $\cX$. The $\bT'$-equivariant Chen-Ruan cohomology of $\cX_\sigma$ is
$$
    H^*_{\CR, \bT'}(\cX_\sigma; \bC) = \bigoplus_{h \in G_\sigma} \bC \one_h
$$
where $\one_h$ has degree $2\age(h)$. Let $\{\sw_{i, \sigma}\}_{i \in I'_\sigma} \subset M'_{\bQ} \coloneqq M \otimes \bQ$ denote the tangent $\bT'$-weights at the fixed point $\cV(\sigma)$ along the three coordinate axes. Then the $\bT'$-equivariant Chen-Ruan cup product is given by
$$
    \one_h \star_{\cX_\sigma} \one_{h'} = \one_{hh'} \prod_{i \in I'_\sigma}^3 \sw_{i, \sigma}^{c_i(h) + c_i(h') - c_i(hh')}
$$
and the $\bT'$-equivariant Poincar\'e pairing is given by
$$
    \left(\one_h, \one_{h'}\right)_{\cX_\sigma, \bT'} = \frac{\delta_{hh', 1}}{|G_\sigma| \prod_{i \in I'_\sigma} \sw_{i, \sigma}^{\delta_{c_i(h), 0}}}.
$$

Let $\bSt$ be the minimal field extension of $S_{\bT'}$ that contains
$$
    \left\{\sw_{i, \sigma}^{c_{i}(h)} : \sigma \in \Sigma(3), i \in I'_\sigma, h \in G_\sigma \right\} \cup \left\{ \prod_{i \in I'_\sigma} \sw_{i, \sigma}^{\frac{1}{2}} : \sigma \in \Sigma(3) \right\}.
$$
For $\sigma \in \Sigma(3)$ and $h \in G_\sigma$, let
$$
    \bar{\one}_h \coloneqq \frac{\one_h}{\prod_{i \in I'_\sigma}\sw_{i, \sigma}^{c_{i}(h)}} \qquad \in H^*_{\CR, \bT'}(\cX_\sigma; \bC) \otimes_{R_{\bT'}} \bSt.
$$
Moreover, for $\gamma \in G_\sigma^*$ with $\chi_\gamma\colon G_\sigma \to \bC^*$ denoting the corresponding character of $G_\sigma$, let
$$
    \bar{\phi}_\gamma \coloneqq \frac{1}{|G_\sigma|} \sum_{h \in G_\sigma} \chi_\gamma(h^{-1}) \bar{\one}_h.
$$
Then we have
$$
    \bar{\phi}_\gamma \star_{\cX_\sigma} \bar{\phi}_{\gamma'} = \delta_{\gamma, \gamma'} \bar{\phi}_\gamma, \qquad \left(\bar{\phi}_\gamma, \bar{\phi}_{\gamma'}\right)_{\cX_\sigma, \bT'} = \frac{\delta_{\gamma, \gamma'}}{|G_\sigma|^2 \prod_{i \in I'_\sigma} \sw_{i, \sigma}}.
$$
Therefore, $\{\bar{\phi}_\gamma\}_{\gamma \in G_\sigma^*}$ gives a canonical basis of the semi-simple Frobenius algebra
$$
    \left(H^*_{\CR, \bT'}(\cX_\sigma; \bC) \otimes_{R_{\bT'}} \bSt, \star_{\cX_\sigma}, \left(-, -\right)_{\cX_\sigma, \bT'} \right)
$$
over $\bSt$.

For $\sigma \in \Sigma(3)$, let $\iota_{\sigma, o}\colon \cX_\sigma \to \cX$ denote the open immersion. The pullbacks along these maps give an isomorphism of Frobenius algebras
$$
    \bigoplus_{\sigma \in \Sigma(3)} \iota_{\sigma, o}^*\colon H^*_{\CR, \bT'}(\cX; \bC) \otimes_{R_{\bT'}} \bSt \longrightarrow \bigoplus_{\sigma \in \Sigma(3)} H^*_{\CR, \bT'}(\cX_\sigma; \bC) \otimes_{R_{\bT'}} \bSt
$$
over $\bSt$. The locally-defined canonical basis elements above provide a canonical basis globally. Specifically, we introduce the index set
$$
    I_\Sigma \coloneqq \{ \bsi = (\sigma, \gamma) : \sigma \in \Sigma(3), \gamma \in G_\sigma^* \}.
$$
For $\bsi = (\sigma, \gamma) \in I_\Sigma$, let $\phi_{\bsi}$ be the unique element in $H^*_{\CR, \bT'}(\cX; \bC) \otimes_{R_{\bT'}} \bSt$ such that $\phi_{\bsi} \big|_{\cX_\sigma} = \bar{\phi}_\gamma$ and $\phi_{\bsi} \big|_{\cV(\sigma')} = 0$ for all $\sigma' \in \Sigma(3) \setminus \{\sigma\}$. Then
$$
    \phi_{\bsi} \star_{\cX} \phi_{\bsi'} = \delta_{\bsi, \bsi'} \phi_{\bsi}, \qquad \left(\phi_{\bsi}, \phi_{\bsi'}\right)_{\cX, \bT'} = \frac{\delta_{\bsi, \bsi'}}{|G_\sigma|^2 \prod_{i \in I'_\sigma} \sw_{i, \sigma}}.
$$
It follows that $\{\phi_{\bsi}\}_{\bsi \in I_\Sigma}$ is a canonical basis of the semi-simple Frobenius algebra
$$
    \left(H^*_{\CR, \bT'}(\cX; \bC) \otimes_{R_{\bT'}} \bSt, \star_{\cX_\sigma}, \left(-, -\right)_{\cX, \bT'} \right)
$$
over $\bSt$. For $\bsi = (\sigma, \gamma) \in I_\Sigma$, we define $\Delta^\bsi \coloneqq |G_\sigma|^2 \prod_{i \in I'_\sigma} \sw_{i, \sigma}$ define $\hat{\phi}_\bsi \coloneqq \sqrt{\Delta^{\bsi}} \phi_{\bsi}$.
Then
$$
    \hat{\phi}_{\bsi} \star_{\cX} \hat{\phi}_{\bsi'} = \delta_{\bsi, \bsi'} \sqrt{\Delta^{\bsi}} \hat{\phi}_{\bsi}, \qquad \left(\hat{\phi}_{\bsi}, \hat{\phi}_{\bsi'}\right)_{\cX, \bT'} = \delta_{\bsi, \bsi'}.
$$
We thus refer to $\{\hat{\phi}_{\bsi}\}_{\bsi \in I_\Sigma}$ as the \emph{classical normalized canonical basis} of $H^*_{\CR, \bT'}(\cX; \bC) \otimes_{R_{\bT'}} \bSt$.

\subsection{Flags and coordinates}\label{sect:Flags}
A \emph{flag} in $\Sigma$ is a pair of cones $\bff = (\tau, \sigma)$ where $\tau \in \Sigma(2)$, $\sigma \in \Sigma(3)$, and $\tau \subset \sigma$. It corresponds to the inclusion of the $\bT$-fixed point $\fp_\sigma \coloneqq \cV(\sigma)$ (resp. $p_\sigma \coloneqq V(\sigma)$) in the $\bT$-invariant line $\fl_\tau \coloneqq \cV(\tau)$ (resp. $l_\tau \coloneqq V(\tau)$) in $\cX$ (resp. $X$). The stabilizer groups fit into a short exact sequence
$$
    \xymatrix{
           1 \ar[r] & G_\tau \ar[r] & G_\sigma \ar[r] & \bmu_{\fr_{(\tau,\si)}} \ar[r] & 1
    }
$$
where $G_\tau$ is a cyclic subgroup of order $\fm_\tau \coloneqq |G_\tau|$ and the quotient is a cyclic group of order $\fr_{\bff} \coloneqq \frac{|G_\sigma|}{|G_\tau|}$. Let $F(\Sigma)$ denote the set of flags in $\Sigma$. 

A flag $\bff = (\tau,\si)$ determines an ordered triple of indices $(i_1 = i_1^\bff, i_2 = i_2^\bff, i_3 = i_3^\bff)$ in $\{1,\ldots, 3+\fp'\}$. They are characterized by $I_\si' = \{ i_1, i_2, i_3\}$, $I_\tau'=\{ i_2, i_3\}$, and
$b_{i_1}, b_{i_2}, b_{i_3}$ are vertices of a triangle in $P \times \{1\}$ in the counterclockwise order.
They also determine a $\bZ$-basis $\{e_1^{\bff}, e_2^{\bff}, e_3^{\bff}\}$ of $N$ such that
$$
b_{i_1} = \fr_{\bff} e_1^{\bff}   -\fs_{\bff}e_2^{\bff}  + e_3^{\bff} ,\qquad
b_{i_2} = \fm_\tau e_2^{\bff}   + e_3^{\bff},\qquad
b_{i_3} = e_3^{\bff}
$$
for a unique $\fs_{\bff} \in \{0,1,\ldots, \fr_{\bff}-1\}$. Then
$\bZ e_1^{\bff} \oplus \bZ e_2^{\bff} = \bZ e_1\oplus \bZ e_2$ and $e_1^{\bff}\wedge
e_2^{\bff} = e_1\wedge e_2$. For $i = 1, \dots, 3+\fp$, we denote the coordinates of $b_i$ under this basis as $(m_i^{\bff}, n_i^{\bff}, 1)$, where
$$
    m_{i_1}^{\bff} = \fr_{\bff}, \quad n_{i_1}^{\bff} = -\fs_{\bff}, \quad m_{i_2}^{\bff} = \fm_\tau, \quad n_{i_2}^{\bff} = m_{i_3}^{\bff} = n_{i_3}^{\bff} = 0.
$$
Let $\{e_1^{\bff\vee}, e_2^{\bff\vee}, e_3^{\vee}\}$ denote the dual $\bZ$-basis of $M$. Let $\su_1^{\bff}, \su_2^{\bff} \in M'$ denote the characters of $\bT'$ corresponding to the images of $e_1^{\bff\vee}, e_2^{\bff\vee}$ under the projection $M \to M'$. For $j = 1,2,3$, let $\sw_{\bff, j} \in M'_{\bQ}$ denote the tangent $\bT'$-weight at $\fp_\sigma$ along the $\bT$-invariant line corresponding to the two cone whose index set is $\{i_1, i_2, i_3\} \setminus \{i_j\}$. We have $$
    \sw_{\bff, 1} = \frac{1}{\fr_{\bff}} \su_1^{\bff}, \quad \sw_{\bff,2} = \frac{\fs_{\bff}}{\fr_{\bff}\fm_\tau}\su_1^{\bff} + \frac{1}{\fm_\tau}\su_2^{\bff}, \quad \sw_{\bff,3} = - \frac{\fs_{\bff} + \fm_\tau}{\fr_{\bff}\fm_\tau}\su_1^{\bff} - \frac{1}{\fm_\tau}\su_2^{\bff}.
$$

Let $\Sigma(2)_c \subset \Sigma(2)$ denote the subset of 2-cones $\tau$ such that $\fl_\tau$ is compact, or equivalently, $\tau$ is contained in two different 3-cones $\sigma_\pm \in \Sigma(3)$. In this case, we say that $\tau$ is an \emph{inner} 2-cone, and denote the two reference flags as $\bff_\pm \coloneqq (\tau, \sigma_\pm)$. Otherwise, $\tau$ is contained in a unique 3-cone $\sigma \in \Sigma(3)$. In this case, we say that $\tau$ is an \emph{outer} 2-cone, and will use $\bff = (\tau, \sigma)$ as the reference flag.

In this paper, we fix a reference flag $\bff_0 = (\tau_0,\si_0)$ such that $\tau_0$ is outer. We assume that the $\bZ$-basis $\{e_1^\vee, e_2^\vee, e_3^\vee\}$ of $M$ chosen in Section \ref{sect:TCY3} is such that $e_1^\vee = e_1^{\bff_0\vee}$, $e_2^\vee = e_2^{\bff_0\vee}$. We also assume without loss of generality that $i_1^{\bff_0} = 1$, $i_2^{\bff_0} = 2$, $i_3^{\bff_0} = 3$. Let
$$
\fr = \fr_{\bff_0}, \qquad \fs = \fs_{\bff_0}, \qquad \fm=\fm_{\tau_0}.
$$
Then $
    b_1 = \fr e_1 - \fs e_2 + e_3,
    b_2 = \fm e_2 + e_3,
    b_3 = e_3,
$
and $(m_i, n_i) = (m_i^{\bff_0}, n_i^{\bff_0})$ for all $i$.

For $\bff \in F(\Sigma)$, we have the changes of bases
$$
    e_1^{\bff} = a_{\bff} e_1 + b_{\bff} e_2, \qquad e_2^{\bff} = c_{\bff} e_1 + d_{\bff} e_2,
$$
$$
    e_1^{\bff\vee} = d_{\bff}e_1^\vee - c_{\bff}e_2^{\vee}, \qquad e_2^{\bff\vee} = -b_{\bff}e_1^\vee + a_{\bff}e_2^{\vee}
$$
for $a_{\bff}, b_{\bff}, c_{\bff}, d_{\bff} \in \bZ$ with $a_{\bff}d_{\bff} - b_{\bff}c_{\bff} = 1$.


Furthermore, we clarify the coordinates around an inner 2-cone $\tau \in \Sigma(2)_c$. Let $\bff_+ = (\tau, \sigma_+)$, $\bff_- = (\tau, \sigma_-)$ be the two associated flags. We have
$$
    \su_1^{\bff_+} + \su_1^{\bff_-} = \fr_{\bff_+} \sw_{\bff_+, 1} + \fr_{\bff_-} \sw_{\bff_-, 1} = 0
$$
which implies that 
\begin{equation}\label{eqn:InnerCoeffRelation}
    c_{\bff_+} + c_{\bff_-} = d_{\bff_+} + d_{\bff_-} = 0.
\end{equation}
Consider the compact curve $l_\tau \subset X$ which represents a class $[l_\tau] \in \bK_{\eff} \subset \bL_{\bQ}$. We have
$$
    \inner{D_{i_1^{\bff^{\pm}}}, [l_\tau]} = \frac{1}{\fr_{\bff_\pm}}, 
$$
$$
    \inner{D_{i_2^{\bff^+}}, [l_\tau]} = \inner{D_{i_3^{\bff^-}}, [l_\tau]} = \frac{\sw_{\bff_+,2} - \sw_{\bff_-, 3}}{\su_1^{\bff_+}} = \frac{\fs_{\bff_+}}{\fr_{\bff_+}\fm_\tau} - \frac{\fs_{\bff_-}}{\fr_{\bff_-}\fm_\tau} - \frac{1}{\fr_{\bff_-}} + \frac{1}{\fm_\tau} \frac{\su_2^{\bff_+} + \su_2^{\bff_-}}{\su_1^{\bff_+}},
$$
$$
    \inner{D_{i_3^{\bff^+}}, [l_\tau]} = \inner{D_{i_2^{\bff^-}}, [l_\tau]} = \frac{\sw_{\bff_+,3} - \sw_{\bff_-, 2}}{\su_1^{\bff_+}} = -\frac{\fs_{\bff_+}}{\fr_{\bff_+}\fm_\tau} + \frac{\fs_{\bff_-}}{\fr_{\bff_-}\fm_\tau} - \frac{1}{\fr_{\bff_+}} - \frac{1}{\fm_\tau} \frac{\su_2^{\bff_+} + \su_2^{\bff_-}}{\su_1^{\bff_+}}.
$$
Note that the sum of the above two degrees sum up to $- \frac{1}{\fr_{\bff_+}} - \frac{1}{\fr_{\bff_-}}$. We have $\inner{D_i, [l_\tau]} = 0$ for $i \not \in \{i_1^{\bff_+}, i_1^{\bff_-}, i_2^{\bff_+}, i_3^{\bff_+}\}$.

\subsection{Aganagic-Vafa branes}\label{sect:AVBranes}
For $\tau \in \Sigma(2)$, let $\cL_\tau$ be an \emph{Aganagic-Vafa Lagrangian suborbifold} of the K\"ahler orbifold $\cX \cong [\tmu^{-1}(\zeta)/G_{\bR}]$ such that the unique $\bT$-invariant line it intersects is $\fl_\tau$. It takes form $\cL_\tau \cong [\tL_\tau/G_{\bR}]$ where
$$
    \tL_\tau = \tmu^{-1}(\zeta) \cap \left\{ \sum_{i=1}^{3+\fp} l_i^1 |Z_i|^2 = c_1, \sum_{i=1}^{3+\fp} l_i^2 |Z_i|^2 = c_2, \sum_{i=1}^{3 + \fp} \arg(Z_i) = c_3  \right\}
$$
for some $l^1, l^2 \in \bZ^{3+\fp}$ satisfying $\sum_{i=1}^{3+\fp} l^1_i = \sum_{i=1}^{3+\fp} l^2_i = 0$. The constants $c_1, c_2, c_3 \in \bR$ are chosen such that the image $\mu_{\bT'_{\bR}}(\cL_\tau)$ is a point in the toric graph of $\cX$ lying in the interior of the edge corresponding to $\tau$. Let $L_\tau \subset X$ be the coarse moduli space of $\cL_\tau$. When $\tau \in \Sigma(2)_c$ is inner, we say that $\cL_\tau$ is an \emph{inner brane} and use the two reference flags $\bff_\pm = (\tau, \sigma_\pm)$. When $\tau \not \in \Sigma(2)_c$ is outer, we say that $\cL_\tau$ is an \emph{outer brane} and use the reference flag $\bff = (\tau, \sigma)$. We denote
$$
    \cL = \bigsqcup_{\tau \in \Sigma(2)} \cL_\tau
$$
whose coarse moduli space is $L = \bigsqcup_{\tau \in \Sigma(2)} L_\tau \subset X$.

For $\bff = (\tau, \sigma) \in F(\Sigma)$, let $\alpha_{\bff}$ denote the disk bounded by $L_\tau$ in $l_\tau$ and centered at $p_\sigma$. When $\tau \in \Sigma(2)_c$ is inner, we have the relation
$$
    [\alpha_{\bff_+}] + [\alpha_{\bff_-}] = [l_\tau]
$$
in $H_2(X, L; \bZ)$. Denoting this relation by $\sim$, we have
$$
    H_2(X, L; \bZ) \cong H_2(X; \bZ) \oplus \bigoplus_{\bff \in F(\Sigma)} \bZ [\alpha_{\bff}] \bigg/ \sim.
$$
Similarly, the semigroups $E(\cX)$, $E(\cX, \cL)$ of effective curve classes are related by
$$
    E(\cX, \cL) \cong E(\cX) \oplus \bigoplus_{\bff \in F(\Sigma)} \bZ_{\ge 0}[\alpha_{\bff}] \bigg/ \sim.
$$

For $\tau \in \Sigma(2)$, we have $\pi_1(\cL_\tau) \cong \bZ \times G_\tau$. Flat $U(1)$-connections $\nabla$ on $\cL_\tau$ are thus parameterized by
$$
    \Hom(\pi_1(\cL_\tau), U(1)) \cong U(1) \times G_\tau^*.
$$
If the second component in the above parameterization is $\eta \in G_\tau^*$, we say that $\nabla$ is \emph{of type} $\eta$. The pair of an Aganagic-Vafa brane and a flat $U(1)$-connection on it, which may be viewed as an \emph{open phase} of the present open geometry, is indexed by the set
$$
    J_\Sigma \coloneqq \{\bfeta = (\tau, \eta) : \tau \in \Sigma(2), \eta \in G_\tau^*\}.
$$
For $\bfeta \in J_\Sigma$, we will use $\cL_{\bfeta}$ to denote the pair $(\cL_\tau, \nabla)$ where $\nabla$ is of type $\eta$.



\subsection{Framing}\label{sect:Framing}
The boundary condition of the open geometry also consists of a uniform choice of \emph{framing} of the Aganagic-Vafa branes $\{\cL_\tau\}_{\tau \in \Sigma(2)}$. Specifically, we choose a parameter $\sff \in \bQ$ written as
$$
    \sff = \frac{\sfb}{\sfa}
$$
for coprime integers $\sfa, \sfb \in \bZ$ such that $\sfa > 0$. This determines a 1-dimensional subtorus
$$
    \bT_{\sff} \coloneqq \ker(\sfa\su_2 - \sfb\su_1)
$$
of the Calabi-Yau 2-torus $\bT'$ which we call the \emph{framing subtorus}. Let $M_{\sff} \coloneqq \Hom(\bT_{\sff}, \bC^*) \cong \bZ \su$ be the character lattice of $\bT_{\sff}$, where $\su$ is the generator whose dual $\sv \in N_{\sff}$ is mapped to the primitive cocharacter $\sv_{\sff} = \sfa e_1 + \sfb e_2 \in N'$ of $\bT'$. We denote the restriction from $\bT'$-weights to $\bT_{\sff}$-weights as
$$
    \big|_{\bT_{\sff}} = \inner{-, \sv_{\sff}} \su\colon M' \longrightarrow M_{\sff}, \qquad \su_1 \longmapsto \sfa \su, \quad \su_2 \longmapsto \sfb \su.
$$

For $\bff = (\tau, \sigma) \in F(\Sigma)$, we have
$$
    \su_1^{\bff} \big|_{\bT_{\sff}} = (\sfa d_{\bff} - \sfb c_{\bff})\su, \qquad \su_2^{\bff} \big|_{\bT_{\sff}} = (-\sfa b_{\bff} + \sfb a_{\bff}) \su.
$$
Recall that $a_{\bff}d_{\bff} - b_{\bff}c_{\bff} = 1$. Since $\sfa$ and $\sfb$ are coprime, $\sfa d_{\bff} - \sfb c_{\bff}$ and $-\sfa b_{\bff} + \sfb a_{\bff}$ are also coprime. We write
$$
    \sfa_{\bff} \coloneqq \sfa d_{\bff} - \sfb c_{\bff} = \inner{\su_1^{\bff}, \sv_{\sff}}, \qquad \sfb_{\bff} \coloneqq -\sfa b_{\bff} + \sfb a_{\bff} = \inner{\su_2^{\bff}, \sv_{\sff}}.
$$

\begin{assumption}\label{assump:GenericFraming}\rm{
We choose $\sff \in \bQ$ generically such that $\sfa_{\bff} \neq 0$ for any flag $\bff \in F(\Sigma)$.
}
\end{assumption}

This means that
$$
    \sw_{\bff, 1} \big|_{\bT_{\sff}} = \frac{1}{\text{$\fr_{\bff}$}} \su_1^{\bff} \big|_{\bT_{\sff}}  =  \frac{\sfa_{\bff}}{\fr_{\bff}} \su  \neq 0.
$$

\begin{definition}\rm{
Define the \emph{framing} of a flag $\bff \in F(\Sigma)$ to be
$$
    \sff_{\bff} \coloneqq \frac{\su_2^{\bff}}{\su_1^{\bff}} \bigg|_{\bT_{\sff}} = \frac{\sfb_{\bff}}{\sfa_{\bff}} \quad \in \bQ.
$$
}
\end{definition}

\begin{definition}\rm{
We define the \emph{sign function}
$$
    \sgn\colon F(\Sigma) \longrightarrow \{\pm 1\}
$$
by the sign of the nonzero integer $\sfa_{\bff}$.
}
\end{definition}

By construction, we have $\sff_{\bff_0} = \sff$ and $\sgn(\bff_0) = +1$. For $\tau \in \Sigma(2)$, let
$$
    \sfa_\tau \coloneqq \sgn(\bff) \sfa_{\bff} = |\sfa_{\bff}|
$$
for some choice of $\bff = (\tau, \sigma)$. For $\tau \not \in \Sigma(2)_c$, the choice of $\bff$ is unique. For $\tau \in \Sigma(2)_c$, Equation \eqref{eqn:InnerCoeffRelation} implies the definition of $\sfa_\tau$ does not depend on the choice of $\bff$ between the two possibilities. We set $\bff_\pm$ to be the flag whose sign is $\pm 1$.


See Figure \ref{fig:Framing} for choices of framings in the examples of $\bC^3$ and the resolved conifold.

\begin{figure}[htb]
$$
	\begin{tikzpicture}[scale=0.8]
		\draw (0,0) -- (1,0) -- (2,0.5);
        \draw (0,-1) -- (1,0);
        \node[above] at (0.5,0){$\tau_0$};
        \node[left] at (0,0){$1$};
        \node[left] at (0,-1){$-2$};
        \node[right] at (2,0.5){$-\frac{1}{2}$}; 
        \node at (1,0){$\bullet$};

        \draw (4,0) -- (5,0) -- (6,0.33);
        \draw (4,-0.5) -- (5,0);
        \node[above] at (4.5,0){$\tau_0$};
        \node[left] at (4,0){$2$};
        \node[left,below] at (4,-0.5){$-\frac{3}{2}$};
        \node[right] at (6,0.33){$-\frac{1}{3}$}; 
        \node at (5,0){$\bullet$};

        \draw (9,-0.5) -- (10,-0.5) -- (12,0.5) -- (13,0.5);
        \draw (9,-1.5) -- (10,-0.5);
        \draw (12,0.5) -- (13,1.5);
        \node at (10,-0.5){$\bullet$};
        \node at (12,0.5){$\bullet$};
        \node[above] at (9.5,-0.5){$\tau_0$};
        \node[left] at (9,-0.5){$1$};
        \node[left] at (9,-1.5){$-2$};
        \node[right] at (13,0.5){$1$};
        \node[right] at (13,1.5){$-2$};
        
        \node[below] at (11,0){$-\frac{1}{2}$}; 

    \end{tikzpicture}
$$
\caption{Two choices of framings for $\bC^3$ and a choice of framing for the resolved conifold. In each example, the framing for a 2-cone is marked on the corresponding edge in the toric graph, and the reference cone $\tau_0$ is taken to be the top-left one. The quantities $\sfa_\tau$ can be obtained from the denominators of the framings. Half edges pointing to the left (resp. right) have positive (resp. negative) sign.}
\label{fig:Framing}
\end{figure}
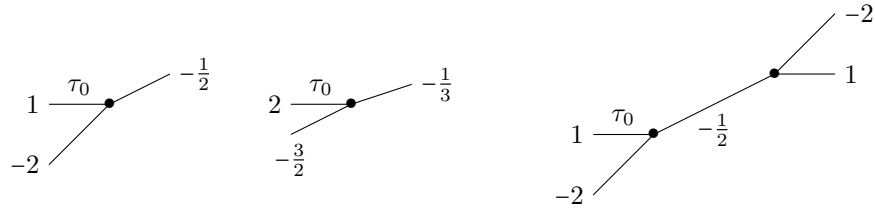


\section{Open Gromov-Witten theory}\label{sect:GW}
In this section, we introduce the background of equivariant descendant and open Gromov-Witten theory. In particular, we give a graph sum formula for the all-genus open Gromov-Witten potentials (Theorem \ref{thm:AmodelOpenGraph}).

\subsection{A-model moduli parameters}\label{sect:AmodelParameter}
We start by defining the moduli parameters for the open-closed Gromov-Witten theory of $\cX$. For an effective curve class $\beta \in E(\cX)$, let
$$
    Q^\beta = e^{-\int_\beta \omega}
$$
be the corresponding Novikov variable with respect to the symplectic form $\omega = \omega(\zeta)$ on $\cX$. Let $\btau' \in H^2_{\bT'}(\cX; \bC)$ be a parameter that lies in the span of $\Dbar_4^{\bT'}, \dots, \Dbar_{3+\fp'}^{\bT'}$. We introduce the variable
$$
    \tQ^\beta = e^{\int_\beta \btau'} Q^\beta.
$$

Let $\bfeta = (\tau, \eta) \in J_\Sigma$ and $\nabla$ be a flat $U(1)$-connection on $\cL_\tau$ of type $\eta$. For $\bff = (\tau, \sigma) \in F(\Sigma)$, we introduce the open moduli parameter
$$
    \tX_{\bff, \eta} = e^{\int_{\alpha_{\bff}} (\btau'-\omega)} \Hol_{\nabla}(\partial \alpha_{\bff}).
$$
When $\cL_\tau$ is outer, we set
$$
    \tX_{\bfeta} \coloneqq \tX_{\bff, \eta}.
$$
When $\cL_\tau$ is inner, we set
$$
    \tX_{\bfeta} \coloneqq \tX_{\bff_+, \eta}.
$$
We have
$$
    \tX_{\bff_+, \eta}\tX_{\bff_-, \eta} = e^{\int_{l_\tau}\btau'}Q^{[l_\tau]} = \tQ^{[l_\tau]}
$$
and thus
$$
    \tX_{\bff_-, \eta} = \tQ^{[l_\tau]} \tX_{\bfeta}^{-1}.
$$

\subsection{Equivariant descendant Gromov-Witten theory}
In this subsection, we review the equivariant descendant Gromov-Witten theory of $\cX$.

\subsubsection{Equivariant Gromov-Witten invariants}
Given $g, n \in \bZ_{\ge 0}$ and $\beta \in E(\cX)$, let $\Mbar_{g,n}(\cX, \beta)$ be the moduli space of genus-$g$, $n$-pointed, degree-$\beta$ twisted stable maps to $\cX$. For $i = 1, \dots, n$, let $\ev_i\colon \Mbar_{g,n}(\cX,\beta)\to \cI\cX$ be the evaluation map
at the $i$-th marked point. The $\bT'$-action on $\cX$ induces
$\bT'$-actions on the moduli space $\Mbar_{g,n}(\cX,\beta)$ and on the inertia stack $\cI\cX$, and
the evaluation map $\ev_i$ is $\bT'$-equivariant. Similarly, we can define the moduli space $\Mbar_{g,n}(X,\beta)$.

For $i=1,\dots,n$, let $\bL_i$ be the $i$-th tautological line bundle over $\Mbar_{g,n}(X, \beta)$ formed
by the cotangent line at the $i$-th marked point. Define the $i$-th descendant class $\psi_i$ as
$$
\psi_i \coloneqq c_1(\bL_i)\in H^2(\Mbar_{g,n}(X, \beta);\bQ).
$$
The $\bT'$-action on $X$ induces a $\bT'$-action on
$\Mbar_{g,n}(X, \beta)$, and we choose a $\bT'$-equivariant
lift $\psi_i^{\bT'}\in H^2_{\bT'}(\Mbar_{g.n}(X, \beta);\bQ)$
of $\psi_i$. Consider the map $p\colon \Mbar_{g,n}(\cX, \beta)\to \Mbar_{g,n}(X, \beta)$ induced by $\cX\to X$, which is $\bT'$-equivariant. For $i=1,\dots,n$, let
$$
\hat{\psi}_i \coloneqq p^*\psi_i \in H^2(\Mbar_{g,n}(\cX, \beta);\bQ), \quad 
\hat{\psi}_i^{\bT'} \coloneqq p^*\psi_i^{\bT'} \in H^2_{\bT'}(\Mbar_{g,n}(\cX, \beta);\bQ).
$$

Let $\gamma_1,\dots, \gamma_n\in H_{\bT'}^*(\cX,\bC)$ and $a_1,\dots,a_n\in \bZ_{\ge 0}$. We define the genus-$g$, degree-$\beta$, $\bT'$-equivariant \emph{descendant Gromov-Witten invariant}
\begin{align*}
	\langle \tau_{a_1}(\gamma_1), \dots, \tau_{a_n}(\gamma_n)\rangle^{\cX,\bT'}_{g,n,\beta} & = 
	\langle \gamma_1\hat{\psi}^{a_1}, \dots,\gamma_n\hat{\psi}^{a_n}\rangle^{\cX,\bT'}_{g,n,\beta}\\
	& \coloneqq \int_{[\Mbar_{g,n}(\cX, \beta)^{\bT'}]^{w,\bT'}} \frac{\iota^*\big(\prod_{i=1}^n \ev_i^*(\gamma_i)(\hat{\psi}_i^{\bT'})^{a_i}\big)}{e_{\bT'}(N^\vir)}
	\quad \in \ST= \bC(\su_1,\su_2)
\end{align*}
where $\Mbar_{g,n}(\cX, \beta)^{\bT'}$ is the $\bT'$-fixed locus, $e_{\bT'}(N^\vir)$ is the $\bT'$-equivariant Euler class of the virtual normal bundle of $\Mbar_{g,n}(\cX,\beta)^{\bT'}$ in $\Mbar_{g,n}(\cX,\beta)$, $[\Mbar_{g,n}(\cX, \beta)^{\bT'}]^{w,\bT'}$ is the weighted virtual fundamental class, and $\iota\colon  \Mbar_{g,n}(\cX, \beta)^{\bT'}\hookrightarrow \Mbar_{g,n}(\cX, \beta)$ is the inclusion map. The invariant
$$
\langle \gamma_1,\dots, \gamma_n \rangle_{g,n,\beta}^{\cX,\bT'}\coloneqq\langle \tau_0(\gamma_1), \dots, \tau_0(\gamma_n)\rangle^{\cX,\bT'}_{g,n,\beta}   
$$
is called the genus-$g$, degree-$\beta$ \emph{primary Gromov-Witten invariant}.
We adopt the following convention for unstable integrals:
\begin{equation}\label{eqn:Unstable}
	\left\langle \frac{b}{w-\hat{\psi}} \right\rangle^{\cX,\bT'}_{0,1,0} = w(\one,b)_{\cX,\bT'}, \quad 
    \left\langle a,\frac{b}{w-\hat{\psi}} \right\rangle^{\cX,\bT'}_{0,2,0} = (a,b)_{\cX,\bT'}, \quad
	\left\langle \frac{a}{w_1-\hat{\psi}},\frac{b}{w_2-\hat{\psi}} \right\rangle^{\cX,\bT'}_{0,2,0} = \frac{(a,b)_{\cX,\bT'}}{w_1 + w_2}.
\end{equation}

Define the Novikov ring
$$
	\nov\coloneqq\widehat{\bC[E(\cX)]}= \left\{ \sum_{\beta \in E(\cX)} c_\beta Q^\beta: c_\beta\in \bC \right\}.
$$
For $\gamma_1,\dots,\gamma_n\in H^*_{\CR,\bT'}(\cX)\otimes_\RT \bST$ and $a_1,\dots, a_n\in \bZ_{\ge 0}$, define the following generating function
$$
	\left\llangle  \tau_{a_1}(\gamma_1), \dots, \tau_{a_n}(\gamma_n) \right\rrangle^{\cX,\bT'}_{g,n} = 
	\left\llangle  \gamma_1\hat{\psi}^{a_1},  \dots, \gamma_n\hat{\psi}^{a_n} \right\rrangle^{\cX,\bT'}_{g,n}
	\coloneqq\sum_{m \ge 0} \sum_{\beta \in E(\cX)}\frac{Q^\beta}{m!} \left\langle
	\gamma_1\hat{\psi}^{a_1}, \dots, \gamma_n\hat{\psi}^{a_n}, t^m \right\rangle^{\cX,\bT'}_{g,n+m,\beta}
$$
where $t\in H^*_{\CR,\bT'}(\cX)\otimes_{R_\bT}\bST$. We adopt the convention \eqref{eqn:Unstable} in the correlators $\double{-}^{\cX,\bT'}_{0,1}$ and $\double{-}^{\cX,\bT'}_{0,2}$.
We choose an $\ST$-basis $\{ T_i: i=0,1,\dots,E-1\}$ of $H^*_{\CR,\bT'}(\cX;\ST)$ such that
$$
T_0= \text{$\one$},\qquad T_a=\bar{D}_{3+a}^{\bT'} \quad \text{for  $a=1,\dots,\fp'$},\qquad
T_a= \text{$\one_{b_{3+a}}$} \quad \textup{for $a=\fp'+1,\dots, \fp$}
$$
and that for $i=\fp+1,\dots,E-1$, $T_i$ is of the form $T_aT_b$ for some $a,b\in \{1,\dots, \fp\}$.
Write $t=\sum_{i=0}^{E-1}\tau_i T_i$, and let $\tau'=(\tau_1,\dots, \tau_{\fp'})$,
$\tau''=(\tau_0, \tau_{\fp'+1},\dots, \tau_{E-1})$.
By the divisor equation,
$$
	\left\llangle  \gamma_1\hat{\psi}^{a_1},  \dots, \gamma_n\hat{\psi}^{a_n} \right\rrangle^{\cX,\bT'}_{g,n}
	=\sum_{m=0}^\infty \sum_{\beta \in E(\cX)}\frac{\tQ^\beta}{m!} \left\langle
	\gamma_1\hat{\psi}^{a_1}, \dots, \gamma_n\hat{\psi}^{a_n}, (t'')^m \right\rangle^{\cX,\bT'}_{g,n+m,d} \quad \in \bSTQ
$$
where $\tQ^\beta = Q^\beta \exp(\sum_{a=1}^{\fp'}\tau_a \langle T_a, \beta \rangle)$ (as in Section \ref{sect:AmodelParameter}) and $t'' = \tau_0 T_0 + \sum_{i=\fp'+1}^{E-1}\tau_i T_i$. In particular, the restriction to $Q=1$ is well-defined.

\subsubsection{Quantum cohomology and canonical coordinates}\label{sec:QH}

For any  $a,b,c\in H_{\CR,\bT'}^*(\cX;\bST)$, define the quantum product $\star_t$ by
$$	
(a\star_t b,c)_{\cX,\bT'} = \left\llangle a,b,c \right\rrangle_{0,3}^{\cX,\bT'} \qquad \in \bSTQ.
$$
Let
$$
	\novT\coloneqq \bST\otimes_{\bC}\nov =\bST [\![ E(\cX)]\!].
$$
Then $H^*_{\CR,\bT'}(\cX;\novT)$ is a free $\novT$-module of rank $E$. 
Consider the formal scheme
$$
\hH \coloneqq\Spec(\novT[\![ t^{\bsi}:\bsi\in I_\Si ]\!]).
$$
The tangent sheaf $\cT_{\hH}$ is a sheaf of free $\cO_{\hH}$-modules of rank $E$.
Given an open set $U$ in $\hH$, we have
$$
\cT_{\hH}(U)  \cong \bigoplus_{\bsi\in I_\Si}\cO_{\hat{H}}(U) \frac{\partial}{\partial t^{\bsi}}.
$$
The quantum product and the $\bT'$-equivariant Poincar\'{e} pairing define the structure of a \emph{formal Frobenius manifold} on $\hH$ over $\Lambda_{\cX}^{\bT'}$:
$$
\frac{\partial}{\partial t^{\bsi}} \star_t \frac{\partial}{\partial t^{\bsi'}}
=\sum_{\brho\in I_\Si} \left\llangle \hat{\phi}_{\bsi},\hat{\phi}_{\bsi'},\hat{\phi}_{\brho} \right\rrangle_{0,3}^{\cX,\bT'}
\frac{\partial}{\partial t^{\brho}}
\in \Gamma(\hat{H}, \cT_{\hat{H}}), \qquad
\left( \frac{\partial}{\partial t^{\bsi}},\frac{\partial}{\partial t^{\bsi'}}\right)_{\cX,\bT'} =\delta_{\bsi,\bsi'}.
$$
The ring $ \big(\Gamma(\hat{H}, \cT_{\hat{H}}), \star_t \big)$ is called the \emph{equivariant big quantum cohomology ring} and is denoted by $QH^*_{\bT'}(\cX)$.

The semi-simplicity of the classical cohomology $H^*_{\CR, \bT'}(\cX;\bC)\otimes_\RT \bST$ implies the semi-simplicity of the quantum cohomology $QH^*_{\bT'}(\cX)$. In fact, there exists a canonical basis $\{\phi_{\bsi}(t)\}_{\bsi\in I_\Si}$ of $QH^*_{\bT'}(\cX)$ characterized by the property that for all $\bsi\in I_\Si$,
$$
\phi_{\bsi}(t)\longrightarrow  \phi_{\bsi},\quad\mathrm{when }\quad t,\tQ\to 0.
$$
We define $\{\phi^{\bsi}(t)\}_{\bsi\in I_\Si}$ to be the basis dual to $\{\phi_{\bsi}(t)\}_{\bsi\in I_\Si}$ with respect to the metric $(-,-)_{\cX,\bT'}$.
The canonical coordinates $\{ u^{\bsi}=u^{\bsi}(t):\bsi\in I_\Si\}$ on the formal Frobenius
manifold $\hat{H}$ are characterized by
$$
\frac{\partial}{\partial u^{\bsi}} = \phi_{\bsi}(t).
$$
The canonical coordinates are unique up to additive constants in $\novT$. We choose canonical coordinates
such that they vanish when $\tQ=0,\hat{t}^\bsi=0,\bsi\in I_\Si$.

We define $\Delta^{\bsi}(t)$ by the following equation
$$
(\phi_{\bsi}(t), \phi_{\bsi'}(t))_{\cX,\bT'} =\frac{\delta_{\bsi,\bsi'}}{\Delta^{\bsi}(t)}.
$$
The normalized canonical basis of $(\hat{H},\star_t)$ is defined as
$$
\{ \hat{\phi}_{\bsi}(t)\coloneqq \sqrt{\Delta^{\bsi}(t)}\phi_{\bsi}(t): \bsi\in I_\Si\}.
$$
We have
$$
\hat{\phi}_{\bsi}(t)\star_t \hat{\phi}_{\bsi'}(t) =\delta_{\bsi, \bsi'}\sqrt{\Delta^{\bsi}(t)}\hat{\phi}_{\bsi}(t),\qquad
(\hat{\phi}_{\bsi}(t), \hat{\phi}_{\bsi'}(t))_{\cX,\bT'}=\delta_{\bsi,\bsi'}.
$$
Define the transition matrix $\Psi=(\Psi_{\bsi'}^{\spa  \bsi})$ as 
$$
\hat{\phi}_{\bsi'}=\sum_{\bsi\in I_\Si} \Psi_{\bsi'}^{\spa \bsi} \hat{\phi}_\bsi(t).
$$

\subsubsection{A-model $R-$matrix}\label{sect:AmodelRmatrix}

We consider the Dubrovin connection $\nabla^z$, which is a family
of connections parameterized by $z\in \bC\cup \{\infty\}$, on the tangent bundle
$T_{\hat{H}}$ of the formal Frobenius manifold $\hat{H}$:
$$
\nabla^z_{\bsi}=\frac{\partial}{\partial t^{\bsi}} -\frac{1}{z} \hat{\phi}_{\bsi}\star_t.
$$
The commutativity (resp. associativity)
of $*_t$ implies that $\nabla^z$ is a torsion
free (resp. flat) connection on $T_{\hat{H}}$ for all $z$. The equation
\begin{equation}\label{eqn:qde}
	\nabla^z \mu=0
\end{equation}
for a section $\mu\in \Gamma(\hat{H},\cT_{\hat{H}})$ is called the $\bT'$-equivariant
{\em big quantum differential equation} (big QDE). Let
$$
\cT_{\hat{H}}^{f,z}\subset \cT_{\hat{H}}
$$
be the subsheaf of flat sections with respect to the connection $\nabla^z$.
For each $z$, $\cT_{\hat{H}}^{f,z}$ is a sheaf of
$\novT$-modules of rank $E$.

An element $L(z)\in \End(\cT_{\hat H})$ is called a {\em fundamental solution} to the $\bT'$-equivariant big QDE if
the $\cO_{\hat{H}}(\hat{H})$-linear map
$$
L(z)\colon  \Gamma(\hat{H},\cT_{\hat{H}}) \longrightarrow \Gamma(\hat{H},\cT_{\hat{H}})
$$
restricts to a $\novT$-linear isomorphism
$$
L(z)\colon  \Gamma(\hat{H},\cT_H^{f,\infty})=\bigoplus_{\bsi\in I_{\Si}} \novT \frac{\partial}{\partial t^{\bsi}}
\longrightarrow \Gamma(\hat{H},\cT_H^{f,z})
$$
between rank-$E$ free $\novT$-modules.

Let $U$ be the diagonal matrix whose diagonal entries are the canonical coordinates.
The results in \cite{Givental97, Givental98} and \cite{Zong15} imply the following statement.
\begin{theorem}\label{R-matrix}
	There exists a unique matrix power series $R(z)= \one + R_1z+R_2 z^2+\cdots$
	satisfying the following properties:
	\begin{itemize}
		\item The entries of each $R_k$ lie in $\bSTQ$.
		\item $\tS=\Psi R(z) e^{U/z}$  is a fundamental solution to the $\bT'$-equivariant
		big QDE \eqref{eqn:qde}.
		\item $R$ satisfies the unitary condition $R^T(-z)R(z)=\one$.
		\item For $(\sigma, \gamma), (\rho, \delta) \in I_\Sigma$, we have
		$$			
			\lim_{\tQ,\tau''\to 0} R_{\rho,\delta}^{\spa\si,\gamma}(z)
			= \frac{\delta_{\rho,\si}}{|G_\si|}\sum_{h\in G_\si}\chi_\delta(h) \chi_\gamma(h^{-1})
			\prod_{i=1}^3 \exp\left( \sum_{m=1}^\infty \frac{(-1)^m}{m(m+1)}B_{m+1}(c_i(h))
			\left(\frac{z}{\w_i(\si)}\right)^m \right)
		$$
		where $B_m(x)$ is the $m$-th \emph{Bernoulli polynomial} defined by the identity
		$
		\frac{te^{tx}}{e^t-1}=\sum_{m\ge 0}\frac{B_m(x)t^m}{m!}
		$.
	\end{itemize}
\end{theorem}
The matrix $R(z)$ in Theorem \ref{R-matrix} is called the {\em A-model $R$-matrix}.

\subsubsection{The $\cS$-operator}

For any $a,b\in H_{\CR,\bT'}^*(\cX;\bST)$, the \emph{$\cS$-operator} is defined as
$$
(a,\cS(b))_{\cX,\bT'} = \left\llangle a,\frac{b}{z-\hat{\psi}}\right\rrangle^{\cX,\bT'}_{0,2}.
$$
The $\cS$-operator, viewed as an element in $\End(\cT_{\hat{H}})$, is a fundamental solution to the $\bT'$-equivariant
big QDE \eqref{eqn:qde} \cite[Section 10.2]{CK99} \cite{Iritani09}. For later use we record some of its well-known properties; see e.g. \cite{Givental97, GT13}. First, it satisfies the unitary condition 
$$
\cS^T(-z)\cS(z) = \one.
$$
Moreover, for any basis $\{\phi_\alpha\}$ of $H^*_{\CR, \bT'}(\cX;\bC)$ and dual basis $\{\phi^\alpha\}$ under the $\bT'$-equivariant Poincar\'e pairing, we have the identity
\begin{equation}\label{eqn:SIdentity}
	\frac{1}{z_1 + z_2}\sum_{\alpha} \left\llangle \phi_\alpha,\frac{a}{z_1-\hat{\psi}}\right\rrangle^{\cX,\bT'}_{0,2} \left\llangle \phi^\alpha,\frac{b}{z_2-\hat{\psi}}\right\rrangle^{\cX,\bT'}_{0,2} = \left\llangle \frac{a}{z_1-\hat{\psi}}, \frac{b}{z_2-\hat{\psi}}\right\rrangle^{\cX,\bT'}_{0,2}.
\end{equation}

We introduce some additional notation. For $\bsi, \bsi'\in I_\Si$, define
$$
S^{\bsi'}_{\spa \bsi}(z) \coloneqq (\phi^{\bsi'}, \cS(\phi_{\bsi}))_{\cX,\bT'},\qquad S_{\bsi'}^{\spa \widehat{\bsi} }(z) \coloneqq (\phi_{\bsi'}, \cS(\hat{\phi}^{\bsi}))_{\cX,\bT'}.
$$
We define
$$
S^{\widehat{\underline{\bsi}}}_{\spa \widehat{\underline{\bsi'}} }(z)
\coloneqq (\hat{\phi}_{\bsi}(t), \cS(\hat{\phi}_{\bsi'}(t))).
$$
Then $(S^{ \widehat{\underline{\bsi}}  }_{\spa \widehat{\underline{\bsi'}} }(z))$ is the matrix of the $\cS$-operator with
respect to the normalized canonical basis  $\{ \hat{\phi}_{\bsi}(t): \bsi\in I_\Si\} $:
$$
\cS(\hat{\phi}_{\bsi'}(t))=\sum_{\bsi\in I_\Si} \hat{\phi}_{\bsi}(t)
S^{\widehat{\underline{\bsi}} }_{\spa \widehat{\underline{\bsi'}} }(z).
$$
Moreover, we define
$$
S^{\widehat{\underline{\bsi}}}_{\spa \bsi'}(z)
\coloneqq (\hat{\phi}_{\bsi}(t), \cS(\phi_{\bsi'})).
$$
Then $(S^{ \widehat{\underline{\bsi}}  }_{\spa \bsi'}(z))$ is the matrix of the $\cS$-operator with
respect to the  basis $\{\phi_{\bsi}:\bsi\in I_\Si\}$ and
$\{ \hat{\phi}_{\bsi}(t): \bsi\in I_\Si\} $:
$$
\cS(\phi_{\bsi'})=\sum_{\bsi\in I_\Si} \hat{\phi}_{\bsi}(t)
S^{\widehat{\underline{\bsi}} }_{\spa \bsi'}(z).
$$

\subsection{A-model descendant graph sum}\label{sect:AmodelDescendantGraph}
In this subsection, we introduce the graph sum formula for the generating functions of descendant Gromov-Witten invariants.

Given a connected graph $\Ga$, we introduce the following notation:
\begin{itemize}
	\item $V(\Ga)$ is the set of vertices.
	\item $E(\Ga)$ is the set of edges.
	\item $H(\Ga)$ is the set of half edges.
	\item $L^o(\Ga)$ is the set of ordinary leaves. The ordinary
	leaves are ordered: $L^o(\Ga)=\{l_1,\dots,l_n\}$ where
	$n$ is the number of ordinary leaves.
	\item $L^1(\Ga)$ is the set of dilaton leaves. The dilaton leaves are unordered.
\end{itemize}
With the above notation, we introduce the following labels:
\begin{itemize}
	\item (genus) $g\colon  V(\Ga)\to \bZ_{\ge 0}$.
	\item (marking) $\bsi\colon  V(\Ga) \to I_\Si$. This induces
	$\bsi \colon L(\Ga)=L^o(\Ga)\cup L^1(\Ga)\to I_\Si$, as follows:
	if $l\in L(\Ga)$ is a leaf attached to a vertex $v\in V(\Ga)$,
	define $\bsi(l)=\bsi(v)$.
	\item (height) $k\colon  H(\Ga)\to \bZ_{\ge 0}$.
\end{itemize}

Given an edge $e$, let $h_1(e),h_2(e)$ be the two half edges associated to $e$. The order of the two half edges does not affect the graph sum formula in this paper. Given a vertex $v\in V(\Ga)$, let $H(v)$ denote the set of half edges
emanating from $v$. The valency of the vertex $v$ is defined to be the cardinality of the set $H(v)$: $\val(v)=|H(v)|$.
A labeled graph $\vGa=(\Ga,g,\bsi,k)$ is {\em stable} if
$$
2g(v)-2 + \val(v) >0
$$
for all $v\in V(\Ga)$. We define the \emph{genus} of a stable labeled graph
$\vGa$ as
$$
g(\vGa)\coloneqq \sum_{v\in V(\Ga)}g(v)  + |E(\Ga)|- |V(\Ga)|  +1
=\sum_{v\in V(\Ga)} (g(v)-1) + \left(\sum_{e\in E(\Gamma)} 1\right) +1.
$$
Let $\bGa(\cX)$ denote the set of all stable labeled graphs
$\vGa=(\Gamma,g,\bsi,k)$. For $g, n \in \bZ_{\ge 0}$ define
$$
\bGa_{g,n}(\cX)=\{ \vGa=(\Gamma,g,\bsi,k)\in \bGa(\cX): g(\vGa)=g, |L^o(\Ga)|=n\}.
$$

For $i = 1, \dots, n$, introduce formal variables
$$
    \bu_i =\bu_i(z)= \sum_{a \in \bZ_{\ge 0}} (u_i)_a z^a
$$
where $(u_i)_k \in H^*_{\CR,\bT'}(\cX;\bC)\otimes_\RT\bST$. We assign weights to leaves, edges, and vertices of a labeled graph $\vGa\in \bGa(\cX)$ as follows:
\begin{itemize}
	\item {\em Ordinary leaves.}
	To each ordinary leaf $l_i \in L^o(\Ga)$ with  $\bsi(l_i)= \bsi\in I_\Si$
	and  $k(l)= k \ge 0$, we assign the following descendant  weight:
		$$
		(\cL^{\bu})^{\bsi}_k(l_i) \coloneqq [z^k] \left(\sum_{\brho, \bsi' \in I_\Si}
		\left(\frac{\bu_i^{\bsi'}(z)}{\sqrt{\Delta^{\bsi'}(t)} }
		S^{\widehat{\underline{\brho}} }_{\spa
			\widehat{\underline{\bsi'}}}(z)\right)_+ R(-z)_{\brho}^{\spa \bsi} \right),
		$$
		where $(\cdot)_+$ means taking the nonnegative powers of $z$.
		
		\item {\em Dilaton leaves.} To each dilaton leaf $l \in L^1(\Ga)$ with $\bsi(l)=\bsi
		\in I_\Si$
		and $k(l)=k \ge 2$, we assign
		$$
		(\cL^1)^{\bsi}_k \coloneqq [z^{k-1}]\left(-\sum_{\bsi'\in I_\Si}
		\frac{1}{\sqrt{\Delta^{\bsi'}(t)}}
		R_{\bsi'}^{\spa \bsi}(-z) \right).
		$$
		
		\item {\em Edges.} To an edge connecting a vertex marked by $\bsi\in I_\Si$ and a vertex
		marked by $\bsi'\in I_\Si$, and with heights $k$ and $l$ at the corresponding half-edges, we assign
		$$
		\cE^{\bsi,\bsi'}_{k,l} \coloneqq [z^k w^l]
		\left(\frac{1}{z+w} \left(\delta_{\bsi\bsi'}-\sum_{\brho\in I_\Si}
		R_{\brho}^{\spa \bsi}(-z) R_{\brho}^{\spa \bsi'}(-w)\right) \right).
		$$
		\item {\em Vertices.} To a vertex $v$ with genus $g(v)=g\in \bZ_{\ge 0}$ and with
		marking $\bsi(v)=\bsi$, with $n_1$ ordinary
		leaves and half-edges attached to it with heights $k_1, ..., k_{n_1} \in \bZ_{\ge 0}$ and $n_2$
		dilaton leaves with heights $k_{n_1+1}, \dots, k_{n_1+n_2}\in \bZ_{\ge 0}$, we assign
		$$
		\left(\sqrt{\Delta^{\bsi}(t)}\right)^{2g(v)-2+\val(v)}\langle  \tau_{k_1}\cdots\tau_{k_{n_1+n_2}}\rangle_g
		$$
		where $\langle  \tau_{k_1}\cdots\tau_{k_{n_1+n_2}}\rangle_{g}\coloneqq \displaystyle \int_{\Mbar_{g,n_1+n_2}}\psi_1^{k_1} \cdots \psi_{n_1+n_2}^{k_{n_1+n_2}}$.
	\end{itemize}
	
	We define the \emph{A-model weight} of a labeled graph $\vGa\in \bGa(\cX)$ to be
	\begin{align*}
		w_A^{\bu}(\vGa) \coloneqq & \prod_{v\in V(\Ga)} \left(\sqrt{\Delta^{\bsi(v)}(t)}\right)^{2g(v)-2+\val(v)} \left\langle \prod_{h\in H(v)} \tau_{k(h)} \right\rangle_{g(v)}
		\prod_{e\in E(\Ga)} \cE^{\bsi(v_1(e)),\bsi(v_2(e))}_{k(h_1(e)),k(h_2(e))}\\
		& \cdot \prod_{l\in L^1(\Ga)}(\cL^1)^{\bsi(l)}_{k(l)}\prod_{j=1}^n(\cL^{\bu})^{\bsi(l_i)}_{k(l_i)}(l_i).
	\end{align*}

Consider the generating function
$$
	\double{ \bu_1,\dots, \bu_n }_{g,n}^{\cX,\bT'} =
	\double{ \bu_1(\hat{\psi}),\dots, \bu_n(\hat{\psi}) }_{g,n}^{\cX,\bT'}
	=\sum_{a_i \in \bZ_{\ge 0}}
	\double{ (u_1)_{a_1}\hat{\psi}^{a_1}, \dots, (u_n)_{a_n}\hat{\psi}^{a_n} }_{g,n}^{\cX,\bT'}.
$$
We have the following graph sum formula.

\begin{theorem}[{\cite{Zong15}}]\label{thm:Zong}
	For $2g-2+n>0$, we have
	$$
	\llangle \bu_1,\dots, \bu_n\rrangle_{g,n}^{\cX,\bT'}=\sum_{\vGa\in \bGa_{g,n}(\cX)}\frac{w_A^{\bu}(\vGa)}{|\Aut(\vGa)|}.
	$$
\end{theorem}
Define the formal variables $\bar{\bu}_i^{\bsi}(z)$ as
$$
\sum_{ \bsi \in I_\Si}\bar{\bu}_i^{\bsi}(z)\phi_{\bsi}(0)=\sum_{ \bsi \in I_\Si}\bu_i^{\bsi}(z)\phi_{\bsi}(\btau).
$$
Then in the above theorem, we may rewrite the A-model ordinary leaf as
\begin{equation}\label{eqn:AmodelOrdinaryLeaf}
	(\cL^\bu)^{\bsi}_k(l_i) = [z^k] \sum_{\brho, \bsi' \in I_\Si}
	\left( \bar{\bu}_i^{\bsi'}(z)
	S^{\widehat{\underline{\brho}} }_{\spa \bsi'}(z)\right)_+ 
	R(-z)_{\brho}^{\spa \bsi}.
\end{equation}

\subsection{Open Gromov-Witten theory}
In this subsection, we consider the open Gromov-Witten theory of $(\cX, \cL, \sff)$. We first set up some notation.

\subsubsection{Conventions for roots}
For $m \in \bZ_{\neq 0}$, we use the following notation for the $m$-th roots:
$$
    \omega_m \coloneqq \exp\left(\frac{2\pi\sqrt{-1}}{m}\right), \qquad \xi_m \coloneqq \exp\left(-\frac{\pi\sqrt{-1}}{m}\right).
$$
When $m > 0$, consider the group $\bmu_m \subset \bC^*$ of $m$-th roots of unity. For $\eta \in \bmu_m^*$, define a corresponding element $\etabar \in \{0, \dots, m-1\}$ by
$$
    \chi_\eta(\omega_m^k) = \exp\left(\frac{2\pi\sqrt{-1}\etabar k}{m}\right)
$$
where $\chi_\eta\colon  \bmu_m \to \bC^*$ is the character corresponding to $\eta$.

\subsubsection{Twisting factors}\label{sect:TwistingFactor}
Let $\bff = (\tau, \sigma) \in F(\Sigma)$. For $\lambda \in G_\tau$, let $\lambdabar \in \{0, \dots, \fm_\tau-1\}$ be such that $\lambda$ is identified with $\omega_{\fm_\tau}^{\lambdabar}$ under the isomorphism
$$
    \chi_{(\tau_3^{\text{$\bff$}}, \sigma)} \colon G_\tau \longrightarrow \bmu_{\fm_\tau} \subset \bC^*
$$
provided by the $G_\sigma$-representation on the tangent line of $\fl_{\tau_3^{\bff}}$ at $\fp_\sigma$. To be more explicit, we choose an isomorphism $\pi_1(\cL_\tau) \cong \bZ \times G_\tau$ such that the homomorphism $h\colon  \pi_1(\cL_\tau) \to G_\sigma = N/N_\sigma$ is given by
$$
    h(d, \lambda) = (d e_1^{\bff} - \lambdabar e_2^{\bff}) + N_\sigma.
$$

\begin{remark}\label{rem:FLTLambdaConvention}\rm{
The choice above is different from that in \cite[Sections 3.8, 4.4]{FLT22} when $\sff_{\bff} \in \bZ$, under which
$$
    (d, \lambda) \longmapsto (d (e_1^{\bff} + \sff_{\bff} e_2^{\bff}) - \lambdabar e_2^{\bff}) + N_\sigma.
$$
We adopt the above modification of this convention to accommodate the general case $\sff_{\bff} \in \bQ$.
}
\end{remark}

Recall that $\{x\}$ denotes the fractional part of $x \in \bR$.
Let $\tau \in \Sigma(2)$, and $\eta \in G_\tau^*$, $d \in \bZ$, $\lambda \in G_\tau$. When $\cL_\tau$ is outer , we define the \emph{twisting factor}
$$
    \eps_\eta^{\bff}(d, \lambda) \coloneqq \xi_{\sfa_{\bff}}^{\sfa_{\bff}\{-\frac{d}{\sfa_{\bff}}\}} (\xi_{\sfa_{\bff}\fm_\tau}\omega_{\sfa_{\bff}\fm_\tau}^{\etabar})^{\sfa_{\bff}\fm_\tau \{- \frac{d\sff_{\bff}}{\fm_\tau} - \frac{\lambdabar}{\fm_\tau}\}}.
$$
Similarly, when $\cL_\tau$ is inner, we define
$$
    \eps_\eta^{\bff_+}(d, \lambda) \coloneqq \xi_{\sfa_{\bff_+}}^{\sfa_{\bff_+}\{-\frac{d}{\sfa_{\bff_+}}\}} (\xi_{\sfa_{\bff_+}\fm_\tau}\omega_{\sfa_{\bff_+}\fm_\tau}^{\etabar})^{\sfa_{\bff_+}\fm_\tau \{- \frac{d\sff_{\bff_+}}{\fm_\tau} - \frac{\lambdabar}{\fm_\tau}\}}
$$
and
$$
    \eps_\eta^{\bff_-}(d, \lambda) \coloneqq \xi_{\sfa_{\bff_-}}^{\sfa_{\bff_-}\{-\frac{d}{\sfa_{\bff_-}}\}} (\xi_{\sfa_{\bff_-}\fm_\tau}\omega_{\sfa_{\bff_-}\fm_\tau}^{\etabar})^{-\sfa_{\bff_-}\fm_\tau \{- \frac{d\sff_{\bff_-}}{\fm_\tau} - \frac{\lambdabar}{\fm_\tau}\}},
$$
where note the extra sign in $\eps_\eta^{\bff_-}$. 

\subsubsection{Disk factors}
Let $\bff = (\tau, \sigma) \in F(\Sigma)$. Let $d \in \bZ_{>0}$ and $\lambda \in G_\tau$. For $j = 1, 2, 3$, let
$$
    w_j^{\bff} \coloneqq \frac{\sw_{\bff, j}}{\su_1^{\bff}}\bigg|_{\bT_{\sff}},
$$
i.e.
$$
    w_1^{\bff} = \frac{1}{\fr_{\bff}}, \qquad  w_2^{\bff} = \frac{\fs_{\bff} + \fr_{\bff}\sff_{\bff}}{\fr_{\bff}\fm_\tau}, \qquad w_3^{\bff} = - \frac{\fs_{\bff} + \fm_\tau + \fr_{\bff}\sff_{\bff}}{\fr_{\bff}\fm_\tau}.
$$
Let $\epsilon_1^{\bff} = \inner{d w_1^{\bff}}$, $\epsilon_2^{\bff}$, $\epsilon_3^{\bff} \in \bQ \cap [0,1)$ be such that $h(d, \lambda) \in G_\sigma$ acts on the tangent spaces of the three $\bT'$-invariant lines at $\fp_\sigma$ with weights $e^{2\pi\sqrt{-1}\epsilon_i^{\bff}}$ respectively. 

The \emph{disk factor} corresponding to the data $\bff, d, \lambda$ is defined by the following formula as in \cite[Section 3.11]{FLT22}, which is based on \cite{Ross14}:
$$    
    D_{\bff, d, \lambda}  = (-1)^{\floor{d w_3^{\bff}-\epsilon_3^{\bff}}+ \ceil{\frac{d}{\sfa_{\bff}}}}\left(\frac{\su_1^{\bff}}{d}\right)^{\age(h(d,\lambda))} \frac{\fr_{\bff}}{d \floor{d w_1^{\bff}}!} \prod_{a = 1}^{\floor{d w_1^{\bff}}+\age(h(d,\lambda))-1} \left(\frac{d\sw_{\bff, 2}}{\su_1^{\bff}} + a - \epsilon_2^{\bff} \right).
$$

\begin{remark} \rm{
We note that the sign convention above differs from that in \cite{FLT22} but agrees with that in \cite{FL13}. Under this sign convention, in the case of a smooth toric Calabi-Yau 3-fold with a single outer brane, the integrality of open BPS invariants holds \cite{Yu24}.
}\end{remark}

\subsubsection{Open Gromov-Witten invariants}\label{sect:OpenGWInv}
Let $g \in \bZ_{\ge 0}$ and $n \in \bZ_{> 0}$. Let $\bfeta_1, \dots, \bfeta_n \in J_\Sigma$. For $i = 1, \dots, n$, take $\mu_i \in \bZ_{\neq 0}$ and $\lambda_i \in G_{\tau_i}$. We require that $\mu_i >0$ if $\cL_{\tau_i}$ is outer.

To simplify notation, for the rest of Section \ref{sect:GW}, we adopt the following practice that unifies the outer and inner cases:
\begin{itemize}
    \item When $\cL_{\tau_i}$ is outer, $\bff_{i,+} = (\tau_i, \sigma_{i,+})$ denotes the unique choice of flag (i.e. $\bff_i = (\tau_i, \sigma_i)$), and $\bff_{i, -}$ and $\sigma_{i, -}$ have no meaning.
    
    \item In the ``$\pm$'' we take ``$+$'' when $\mu_i > 0$ and ``$-$'' when $\mu_i < 0$. In particular, we only have the ``$+$'' case when $\cL_{\tau_i}$ is outer.
\end{itemize}
The data above determines an effective class
$$
    \beta + \sum_{i = 1}^n |\mu_i| [\alpha_{\bff_{i, \pm}}] \qquad \in E(\cX, \cL).
$$

Let $\btau \in H^2_{\CR, \bT'}(\cX; \bC)$ be a parameter in the small phase space, decomposed as
\begin{equation}\label{eqn:btauDecompose}
   \btau = \btau_0 + \btau' + (\btau'' - \btau_0)
\end{equation}
with $\btau_0 \in H^2_{\bT'}(\pt; \bC)$, $\btau' \in H^2_{\bT'}(\cX; \bC)$ lies in the span of $\Dbar_4^{\bT'}, \dots, \Dbar_{3+\fp'}^{\bT'}$, and $\btau'' - \btau_0$ is supported on the twisted sectors.
For $\beta \in E(\cX)$ and $l \in \bZ_{\ge 0}$, consider the open Gromov-Witten invariant
$$
    \inner{(\btau)^l}_{g, \beta, (\mu_1, \lambda_1), \dots, (\mu_n, \lambda_n)} \in \bC
$$
where there are $l$ interior insertions of the class $\btau$. It is defined in \cite[Section 3]{FLT22} via $\bT'_{\bR}$-localization on the moduli space of open stable maps and weight restriction to $\bT_{\sff}$. It is also computed by \cite[Section 3]{FLT22} via localization computation as follows:
\begin{equation}\label{eqn:OGWLocalize}
    \inner{(\btau)^l}_{g, \beta, (\mu_1, \lambda_1), \dots, (\mu_n, \lambda_n)} = \prod_{i = 1}^n D_{\bff_{i, \pm}, |\mu_i|, \lambda_i} \int_{[\text{$\Mbar_{g, n + l}(\cX, \beta)$}]^\vir} \frac{\prod_{i = 1}^n \ev_i^* (\iota_{\sigma_{i, \pm}, *} \one_{h_\pm(|\mu_i|, \lambda_i)^{-1}}) \prod_{i = 1}^l \ev_{n + i}^*(\btau)}{\prod_{i = 1}^n  \left(\frac{\su_1^{\bff_{i, \pm}}}{|\mu_i|} - \hat{\psi} \right) \frac{\su_1^{\bff_{i, \pm}}}{|\mu_i|} } \bigg|_{\bT_{\sff}}.
\end{equation}
Here, we again adopt the convention \eqref{eqn:Unstable}; compare \cite[Proposition 5.1]{flz2020affine}.


\subsection{A-model open potentials}\label{sect:AmodelOpen}
From now on, we adhere to the following convention.

\begin{convention}\rm{
In the correlators $\double{-}^{\cX,\bT'}_{g,n}$ for all $g, n$, we take the specialization
$$
    t = \btau, \qquad Q = 1.
$$
}
\end{convention}

Let $g \in \bZ_{\ge 0}$, $n \in \bZ_{> 0}$, and $\bfeta_1, \dots, \bfeta_n \in J_\Sigma$. For $i = 1, \dots, n$, let $\tX_i$ denote a copy of the open moduli variable $\tX_{\bfeta_i}$ (defined in Section \ref{sect:AmodelParameter}). We define the generating function
\begin{equation}\label{eqn:AmodelFgn}
    \begin{aligned}
        F_{g,n}(\btau; \tX_1, \dots, \tX_n) \coloneqq & \sum_{\beta \in E(\cX)} \sum_{\mu_i \in \bZ_{\neq 0}} \sum_{\lambda_i \in G_{\tau_i}} \sum_{l \in \bZ_{\ge 0}} \frac{\inner{(\btau)^l}_{g, \beta, (\mu_1, \lambda_1), \dots, (\mu_n, \lambda_n)}}{l!} \\
        & \prod_{\mu_i > 0} (\eps_{\eta_i}^{\bff_{i,+}}(\mu_i, \lambda_i) \tX_i^{\mu_i})
        \prod_{\mu_i < 0} (\eps_{\eta_i}^{\bff_{i,-}}(-\mu_i, \lambda_i)  e^{-\mu_i\int_{l_{\tau_i}}\btau'} \tX_i^{\mu_i})  
    \end{aligned}  
\end{equation}
which is a formal Laurent series in $\tX_1, \dots, \tX_n$ (whose coefficients are $\bC$-valued formal functions in $\btau$) where in each term the power of each $\tX_i$ is nonzero. If $\cL_{\tau_i}$ is outer, the power of $\tX_i$ is always positive. By the localization computation \eqref{eqn:OGWLocalize}, we have
$$
    \begin{aligned}
        F_{g,n}(\btau; \tX_1, \dots, \tX_n) 
		= & \sum_{\mu_i \in \bZ_{\neq 0}} \sum_{\lambda_i \in G_{\tau_i}} \double{ \frac{\iota_{\sigma_{1, \pm}, *} \one_{h_\pm(|\mu_1|, \lambda_1)^{-1}}}{ \big( \su_1^{\bff_{1, \pm}}/|\mu_1| - \hat{\psi} \big) \su_1^{\bff_{1, \pm}}/|\mu_1|  }, \dots, \frac{\iota_{\sigma_{n, \pm}, *} \one_{h_\pm(|\mu_n|, \lambda_n)^{-1}}}{\big(\su_1^{\bff_{n, \pm}}/|\mu_n| - \hat{\psi} \big)  \su_1^{\bff_{n, \pm}}/|\mu_n|  } }_{g,n}^{\cX,\bT'}   \\
        & \prod_{i = 1}^n D_{\bff_{i, \pm}, |\mu_i|, \lambda_i} 
        \prod_{\mu_i > 0} (\eps_{\eta_i}^{\bff_{i,+}}(\mu_i, \lambda_i) \tX_i^{\mu_i})
        \prod_{\mu_i < 0} (\eps_{\eta_i}^{\bff_{i,-}}(-\mu_i, \lambda_i)  e^{-\mu_i\int_{l_{\tau_i}}\btau'} \tX_i^{\mu_i})  \bigg|_{\bT_{\sff}}   \\
        = & \sum_{h_i \in G_{\sigma_{i, \pm}}} \sum_{\substack{\pm \mu_i \in \bZ_{> 0}, \lambda_i \in G_{\tau_i} \\ h_\pm(|\mu_i|, \lambda_i) = h_i}} \sum_{a_i \in \bZ_{\ge 0}}
        \double{ \iota_{\sigma_{1, \pm}, *} \one_{h_1^{-1}} \hat{\psi}^{a_1} , \dots, \iota_{\sigma_{n, \pm}, *} \one_{h_n^{-1}}\hat{\psi}^{a_n} }_{g,n}^{\cX,\bT'}   \\
        & \prod_{i = 1}^n \left(\frac{\su_1^{\bff_{i, \pm}}}{|\mu_i|}\right)^{-a_i-2} D_{\bff_{i, \pm}, |\mu_i|, \lambda_i} 
        \prod_{\mu_i > 0} (\eps_{\eta_i}^{\bff_{i,+}}(\mu_i, \lambda_i) \tX_i^{\mu_i})
        \prod_{\mu_i < 0} (\eps_{\eta_i}^{\bff_{i,-}}(-\mu_i, \lambda_i)  e^{-\mu_i\int_{l_{\tau_i}}\btau'} \tX_i^{\mu_i})  \bigg|_{\bT_{\sff}}.         
    \end{aligned}  
$$

To connect to the descendant graph sum formula later, we change the insertions $\iota_{\sigma_{i, \pm}, *} \one_{h_i^{-1}}$ to the classical canonical basis elements. Recall that for any $\sigma, \sigma' \in \Sigma(3)$ and $h \in G_\sigma$, we have
$$
    \iota_{\sigma'}^* \iota_{\sigma, *} \one_{h^{-1}} = \delta_{\sigma, \sigma'} \prod_{j \in I'_{\sigma}} \sw_{j, \sigma}^{\delta_{0, c_j(h^{-1})}} \one_{h^{-1}}.
$$
Moreover, for any $\gamma \in G_\sigma^*$, we have
$$
    \iota_{\sigma'}^* \phi_{(\sigma, \gamma)} = \frac{\delta_{\sigma, \sigma'}}{|G_\sigma|} \sum_{h \in G_\sigma} \chi_\gamma(h) \prod_{j \in I'_{\sigma}} \sw_{j, \sigma}^{-c_j(h^{-1})} \one_{h^{-1}}.
$$
Therefore, we have
$$
    \iota_{\sigma, *} \one_{h^{-1}} = \sum_{\gamma \in G_\sigma^*} \chi_{\gamma}(h^{-1}) \prod_{j \in I'_{\sigma}} \sw_{j, \sigma}^{1-c_j(h)} \phi_{(\sigma, \gamma)} 
$$
by the orthogonality of characters. 

We therefore introduce the following notation. Let $i = 1, \dots, n$. For $\gamma \in G_{\sigma_{i, +}}^*$ and $a \in \bZ$, define
$$
    \txi_a^\gamma(\tX_i) \coloneqq \sum_{h \in G_{\sigma_{i,+}}} \chi_\gamma(h^{-1}) \prod_{j \in I'_{\sigma_{i, +}}} \sw_{j, \sigma_{i, +}}^{1 - c_j(h)} 
    \sum_{\substack{\mu_i \in \bZ_{>0}, \lambda_i \in G_{\tau_i} \\ h_+(\mu_i, \lambda_i) = h}} 
    \left(\frac{\su_1^{\bff_{i, +}}}{\mu_i} \right)^{-a-2} 
    D_{\bff_{i, +}, \mu_i, \lambda_i} \eps_{\eta_i}^{\bff_{i,+}}(\mu_i, \lambda_i) \tX_i^{\mu_i}.
$$
Similarly, when $\cL_{\tau_i}$ is inner, for $\gamma \in G_{\sigma_{i, -}}^*$ and $a \in \bZ$, define
$$
    \txi_a^\gamma(\tX_i) \coloneqq \sum_{h \in G_{\sigma_{i, -}}} \chi_\gamma(h^{-1}) \prod_{j \in I'_{\sigma_{i, -}}} \sw_{j, \sigma_{i, -}}^{1 - c_j(h)} 
    \sum_{\substack{\mu_i \in \bZ_{<0}, \lambda_i \in G_{\tau_i} \\ h_-(-\mu_i, \lambda_i) = h}} 
    \left(\frac{\su_1^{\bff_{i, -}}}{-\mu_i} \right)^{-a-2} 
    D_{\bff_{i, -}, - \mu_i, \lambda_i} \eps_{\eta_i}^{\bff_{i,-}}(-\mu_i, \lambda_i)  e^{-\mu_i\int_{l_{\tau_i}}\btau'} \tX_i^{\mu_i}.
$$
Observe that
$$
    \tX_i \frac{d}{d\tX_i} \txi_a^\gamma(\tX_i) = \pm \su_1^{\bff_{i, \pm}} \txi_{a+1}^\gamma(\tX_i).
$$
Then we have
$$
    F_{g,n}(\btau; \tX_1, \dots, \tX_n)
    = \sum_{\gamma_i \in G_{\sigma_{i, \pm}}^*} \sum_{a_i \in \bZ_{\ge 0}}
    \double{ \phi_{(\sigma_{1, \pm}, \gamma_1)} \hat{\psi}^{a_1}, \dots, \phi_{(\sigma_{n, \pm}, \gamma_n)} \hat{\psi}^{a_n} }_{g,n}^{\cX,\bT'}
    \prod_{i = 1}^n \txi_{a_i}^{\gamma_i}(\tX_i) \bigg|_{\bT_{\sff}}.
$$

For $i = 1, \dots, n$ and $\bsi = (\sigma, \gamma) \in I_\Sigma$, define
$$
    \txi_{i}^{\bsi}(z, \tX_i) \coloneqq \begin{cases}
        \sum_{a \in \bZ_{\ge -2}} z^a \txi_a^\gamma(\tX_i) & \text{if } \sigma = \sigma_{i, \pm},\\
        0 & \text{otherwise}.
    \end{cases}
$$
This quantity satisfies that
\begin{equation}\label{eqn:AmodelXiDerivative}
    [z^a] \tX_i \frac{d}{d\tX_i} \txi_{i}^{\bsi}(z, \tX_i) = \su_1^{\bff_{i, +}} [z^{a+1}] \txi_{i}^{\bsi}(z, \tX_i)
\end{equation}
for any $a \in \bZ_{\ge -2}$. Then the above computation may be rewritten as follows.

\begin{lemma}
We have
\begin{equation}\label{eqn:AmodelFgnCanonical}
    F_{g,n}(\btau; \tX_1, \dots, \tX_n)
    = [z_1^{-1} \cdots z_n^{-1}] \sum_{\bsi_i \in I_\Sigma}
    \double{ \frac{\phi_{\bsi_1}}{z_1 - \hat{\psi}}, \dots, \frac{\phi_{\bsi_n}}{z_n - \hat{\psi}} }_{g,n}^{\cX,\bT'}
    \prod_{i = 1}^n \txi_{i}^{\bsi_i}(z_i, \tX_i) \bigg|_{\bT_{\sff}}.
\end{equation}
\end{lemma}

In the case $(g, n) = (0, 1)$, Equation \eqref{eqn:AmodelFgnCanonical} specializes to
\begin{equation}\label{eqn:AmodelDiskCanonical}
    F_{0,1}(\btau; \tX_1) = [z_1^{-1}] \sum_{\bsi_1 \in I_\Sigma}
    \double{ \frac{\phi_{\bsi_1}}{z_1 - \hat{\psi}} }_{0,1}^{\cX,\bT'}
    \txi_{1}^{\bsi_1}(z_1, \tX_1)  \bigg|_{\bT_{\sff}}
    = [z_1^{-2}] \sum_{\bsi_1 \in I_\Sigma}
    \left( \one, \cS(\phi_{\bsi_1}) \right)_{\cX,\bT'} \txi_{1}^{\bsi_1}(z_1, \tX_1) \bigg|_{\bT_{\sff}}.
\end{equation}
In the case $(g, n) = (0, 2)$, Equation \eqref{eqn:AmodelFgnCanonical} specializes to
\begin{equation}\label{eqn:AmodelAnnCanonical}    
    F_{0,2}(\btau; \tX_1, \tX_2) = [z_1^{-1}z_2^{-1}] \sum_{\bsi_1, \bsi_2 \in I_\Sigma}
    \double{ \frac{\phi_{\bsi_1}}{z_1 - \hat{\psi}}, \frac{\phi_{\bsi_2}}{z_2 - \hat{\psi}} }_{0,2}^{\cX,\bT'}
    \txi_{1}^{\bsi_1}(z_1, \tX_1) \txi_{2}^{\bsi_2}(z_2, \tX_2)  \bigg|_{\bT_{\sff}}.
\end{equation}


\subsection{A-model open graph sum}\label{sect:AmodelOpenGraph}
We introduce the following specialization of equivariant parameters.

\begin{convention}\label{conv:SpecializeU} \rm{
Let $u \in \bC \setminus \{0\}$ and put
$u_1 = \sfa u, u_2 = \sfb u$.
We make the specialization of the equivariant parameters
$$
    \su = u, \qquad \su_1 = u_1, \qquad \su_2 = u_2.
$$
}\end{convention}

In this subsection, we further specialize to $u = 1$, i.e. $u_1 = \sfa$, $u_2 = \sfb$.

For $i = 1, \dots, n$, $\bsi \in I_\Si$, and $k \in \bZ_{\ge 0}$, the \emph{A-model open leaf} is defined by
$$
    (\cL^O)^{\bsi}_k(l_i) \coloneqq [z^k] \sum_{\brho, \bsi' \in I_\Si}
    \left( \txi_{i}^{\bsi'}(z, \tX_i)
    S^{\widehat{\underline{\brho}} }_{\spa \bsi'}(z)\right)_+ 
    R(-z)_{\brho}^{\spa \bsi}  \bigg|_{\substack{\su_1 = \sfa \\ \su_2 = \sfb}}.
$$
Then, Equation \eqref{eqn:AmodelFgnCanonical} and the descendant graph sum formula in Theorem \ref{thm:Zong} imply the following open graph sum formula.

\begin{theorem}\label{thm:AmodelOpenGraph}
For $2g - 2 + n > 0$, we have
$$ 
	F_{g,n}(\btau; \tX_1, \dots, \tX_n) = \sum_{\vGa \in \bGa_{g,n}(\cX)}\frac{w_A^O(\vGa)}{|\Aut(\vGa)|},
$$
where $w_A^O(\vGa)$ is defined by substituting the ordinary leaves $\cL^\bu$ in $w_A^{\bu}(\vGa)$ with the open leaves $\cL^O$ and applying the specialization $\big|_{\substack{\su_1 = \sfa \\ \su_2 = \sfb}}$.
\end{theorem}



\section{Mirror curves and disk mirror symmetry}\label{sect:DiskMirror}
In this section, we study the mirror curves of toric Calabi-Yau 3-orbifolds. Moreover, we prove disk mirror symmetry which retrieves the disk potential $F_{0,1}$ from local coordinate expansions on the mirror curve (Theorem \ref{thm:DiskExpansion}).

\subsection{Notation}
We choose $H_1, \dots, H_{\fp} \in \bL^\vee \cap \tNef(\cX)$ that satisfy the following conditions. We denote $\Hbar_a \coloneqq \kappa(H_a)$ where $\kappa$ is defined in \eqref{eqn:Kirwan}.
\begin{itemize}
	\item We have $H_a = D_{3+a}$ for $a = \fp'+1, \dots, \fp$. In particular, $\Hbar_a = 0$ for such $a$.

	\item The set $\{H_1, \dots, H_{\fp}\}$ is a $\bQ$-basis of $\bL^\vee_{\bQ} \coloneqq \bL^\vee \otimes \bQ$. In particular, $\{\Hbar_1, \dots, \Hbar_{\fp'}\}$ is a $\bQ$-basis of $H^2(\cX; \bQ)$.
	
	\item Given any 3-cone $\sigma \in \Sigma(3)$, and writing $H_a = \sum_{i \in I_\sigma} s_{ai}^\sigma D_i$ for unique coefficients $s_{ai}^\sigma \in \bQ_{\ge 0}$, we require that $s_{ai}^\sigma \in \bZ_{\ge 0}$ for all $\sigma, a, i$.
	
\end{itemize}
For $a = 1, \dots, \fp'$, let $\Hbar_a^{\bT'} \in H^2_{\bT'}(\cX; \bZ)$ be the unique lift of $\Hbar_a$ whose restriction to $\fp_{\si_0}$ is zero, where $\sigma_0$ is the fixed reference 3-cone. We define the coefficients $m_i^{(a)}$ inverse to $s_{ai}^{\sigma_0}$ by
$$
    D_i = \sum_{a = 1}^{\fp} m_i^{(a)} H_a, \qquad i = 4, \dots, 3+\fp.
$$

Let $q = (q_1, \dots, q_{\fp})$ be formal variables. For $\beta \in \bK_{\eff} \subset \bK$, we define
$$
	q^\beta \coloneqq \prod_{a = 1}^{\fp} q_a^{\inner{H_a, \beta}}
$$
which is a monomial in $q$.

\subsection{Mirror curves}\label{sect:MirrorCurve}
Consider $(\bC^*)^2 = M' \otimes \bC^*$. Let $X,Y$ be the coordinates corresponding to the basis elements $e_1^\vee$, $e_2^\vee$ fixed by the reference flag $\bff_0$. Let
$$
    x = -\log X, \qquad y = -\log Y
$$
be coordinates on the universal cover.

Given any $\si\in \Si(3)$ and $i$, we define
$$
a_i^\si(q)\coloneqq
\begin{cases}
1, & i\in I'_\si,\\
\prod_{a=1}^{\fp} q_a^{s^\si_{ai}}, & i\in I_\si,
\end{cases}
$$
and we omit the superscript ``$\sigma$'' when $\sigma = \sigma_0$. For a flag $\bff = (\tau, \sigma) \in F(\Sigma)$, we introduce the coordinates 
$$
    X_{\bff} = X^{a_{\bff}} Y^{b_{\bff}} a_{i_1^{\bff}}(q)^{\frac{1}{\fr_\bff}} a_{i_2^{\bff}}(q)^{\frac{\fs_\bff}{\fr_\bff \fm_\tau}} a_{i_3^{\bff}}(q)^{-\frac{\fm_\tau + \fs_\bff}{\fr_\bff \fm_\tau}},  \qquad 
    Y_{\bff} = X^{c_{\bff}} Y^{d_{\bff}} \left(\frac{a_{i_2^{\bff}}(q)}{a_{i_3^{\bff}}(q)}\right)^\frac{1}{\fm_\tau}
$$
and
$$
    x_{\bff} = -\log X_{\bff}, \qquad y_{\bff} = -\log Y_{\bff}.
$$

The \emph{mirror curve} of $\cX$ is the affine curve
$$
    C_q\coloneqq\{(X,Y)\in (\bC^*)^2: H(X,Y,q)=0\}
$$
where
$$
H(X,Y,q)\coloneqq  \sum_{i=1}^{3+\fp} a_i(q) X^{m_i} Y^{n_i}.
$$

For a flag $\bff = (\tau, \sigma) \in F(\Sigma)$, we define
$$
H_{\bff} (\XX , \YY, q)\coloneqq \sum_{i=1}^{3+\fp} a_i^\si(q) \XX^{m_i^{\bff}} \YY^{n_i^{\bff}}.
$$
We have the relation
$$
    H(X,Y,q) = a_{i_3^{\bff}}(q) X^{m_{i_3^{\bff}}} Y^{m_{i_3^{\bff}}}H_{\bff}(X_{\bff}, Y_{\bff}, q).
$$
Thus $H_{\bff}$ provides an alternative parameterization of the mirror curve.

\subsection{Meromorphic functions and Liouville form}
Consider the specialization $\su = u, \su_1 = u_1 = \sfa u, \su_2 = u_2 = \sfb u$ as in Convention \ref{conv:SpecializeU}. Define
$$
    \hx = u_1 x + u_2 y,\qquad \hy = \frac{y}{u_1}.
$$
We have
$$
    \omega_{0,1} \coloneqq \hy d\hx = y (dx + \sff dy).
$$
Let
$$
    \Crit \coloneqq \{p_{\bsi}\}_{\bsi \in I_\Sigma}
$$
denote the set of critical points of $\hx$.

\begin{assumption}\label{assump:GenericFraming2}\rm{
In addition to Assumption \ref{assump:GenericFraming}, we choose $\sff$ generically such that $\hx$ is holomorphic Morse and $d\hx$ intersects the zero section of the cotangent bundle transversely.
}
\end{assumption}

For a flag $\bff = (\tau, \sigma) \in F(\Sigma)$, we define
$$
    \hx_{\bff} = u_1^{\bff} x_{\bff} + u_2^{\bff} y_{\bff},\qquad \hy_{\bff} = \frac{y_{\bff}}{u_1^{\bff}}.
$$
Here,
$$
    u_1^{\bff} = (\sfa d_{\bff} - \sfb c_{\bff}) u = \sfa_{\bff} u, \qquad u_2^{\bff} = (-\sfa b_{\bff} + \sfb a_{\bff}) u = \sfb_{\bff} u.
$$

For $\sigma \in \Sigma(3)$, the \emph{tropical distance} between $\sigma$ and $\sigma_0$ (as defined in \cite[Section 5.2]{FLYZ25}) is
$$
    c_\sigma \coloneqq \sfa_{\bff} u \sum_{j = 1}^3 w_j^{\bff} \log a_{i_j^{\bff}}(q)
$$
for any flag $\bff = (\tau, \sigma)$.\footnote{We note that in \cite{FLYZ25}, the symbol ``$w_j^{\bff}$'' denotes the quantity $u_1^{\bff} w_j^{\bff} = \sfa_{\bff} u w_j^{\bff}$ here.} Alternatively, if we set
$$
    T_1(\sigma) \coloneqq \left( a_{i_1^{\bff}}(q)^{n_{i_2^{\bff}} - n_{i_3^{\bff}}} a_{i_2^{\bff}}(q)^{n_{i_3^{\bff}} - n_{i_1^{\bff}}}  a_{i_3^{\bff}}(q)^{n_{i_1^{\bff}} - n_{i_2^{\bff}}} \right)^{\frac{1}{|G_\sigma|}},
$$
$$
    T_2(\sigma) \coloneqq \left( a_{i_1^{\bff}}(q)^{m_{i_3^{\bff}} - m_{i_2^{\bff}}} a_{i_2^{\bff}}(q)^{m_{i_1^{\bff}} - m_{i_3^{\bff}}}  a_{i_3^{\bff}}(q)^{m_{i_2^{\bff}} - m_{i_1^{\bff}}} \right)^{\frac{1}{|G_\sigma|}}
$$
which are monomials in $q^{\pm1}$, we have
$$
    c_\sigma = \sfa u \log T_1(\sigma) + \sfb u \log T_2(\sigma).
$$
Then, for any $\bff = (\tau, \sigma) \in F(\Sigma)$, we have
$$
    \hx_{\bff} = \hx - c_\sigma.
$$
In particular, $d\hx_{\bff} = d\hx$ and the set $\Crit$ does not depend on the choice of the reference flag $\bff_0$.

\subsection{Framed coordinates and ramifications}
Let $\Cbar_q$ be the compactification of $C_q$. We assume that $q$ belongs to an open parameter domain such that $|q|$ is sufficiently small and $\Cbar_q$ is smooth. The meromorphic function
$$
    \sP\coloneqq X^{\sfa}Y^{\sfb}\colon  \Cbar_q \longrightarrow \bP^1
$$
is a ramified covering of smooth projective curves. For a flag $\bff = (\tau, \sigma) \in F(\Sigma)$, we have
$$
    \XX^{\sfa_{\bff}}\YY^{\sfb_{\bff}} = X^{\sfa}Y^{\sfb} T_1(\sigma)^{\sfa} T_2(\sigma)^{\sfb}.
$$

Let $\tau \in \Sigma(2)$. When $\cL_\tau$ is outer, the region on $C_q$ defined by $|X_{\bff}|$ being sufficiently small consists of $\fm_\tau$ punctured (open) disks indexed by $G_\tau^*$. The region is disjoint from $\Crit$. For $\bfeta =  (\tau, \eta)$, let $U_{\bfeta}$ denote the component indexed by $\eta$. Its closure in $\Cbar_q$ contains the puncture corresponding to $\bfeta$ and lying in the limit $X_{\bff} \to 0$. The puncture is mapped to $0$ (resp. $\infty$) under $\sP$ if $\sgn(\bff) = +1$ (resp. $-1$). We say that the puncture is \emph{positive} in the former case and \emph{negative} in the latter case. By construction, the punctures corresponding to $\tau_0$ are all positive. On $U_{\bfeta}$, consider the local framed coordinate
$$
    \hX_{\bff, \eta} \coloneqq \xi_{\sfa_{\bff}}X_{\bff}Y_{\bff}^{\sff_{\bff}}.
$$
The definition here involves locally choosing a branch of $\sfa_{\bff}$-th root of $Y_{\bff}$ which we will specify later in Section \ref{sect:LocalExpansion}. On $U_{\bfeta}$, the covering $\sP$ restricts to a degree-$\sfa_\tau$ covering of punctured disks. We set
$$
    \hX_{\bfeta} \coloneqq \hX_{\bff, \eta}.
$$

\begin{remark}\rm{
The degree of the covering $\sP$ is
$$
    \sum_{\substack{(\tau, \eta) \in J_\Sigma \\ \text{ positive}}} \sfa_\tau = \sum_{\substack{(\tau, \eta) \in J_\Sigma \\ \text{ negative}}} \sfa_\tau.
$$
This is consistent with the Riemann-Hurwitz formula since each of the $2\fg - 2 + \fn$ ramification points in $\Crit \subset C_q$ has index two.
}
\end{remark}

When $\cL_\tau$ is inner, the region on $C_q$ defined by $|X_{\bff_+}|$ and $|X_{\bff_-}|$ both being sufficiently small consists of $\fm_\tau$ (open) annuli indexed by $G_\tau^*$. The region is disjoint from $\Crit$. For $\bfeta =  (\tau, \eta)$, let $U_{\bfeta}$ denote the component indexed by $\eta$. On $U_{\bfeta}$, consider the local framed coordinates
$$
    \hX_{\bff_+, \eta} \coloneqq \xi_{\sfa_{\bff_+}}X_{\bff_+}Y_{\bff_+}^{\sff_{\bff_+}}, \qquad \hX_{\bff_-, \eta} \coloneqq \xi_{\sfa_{\bff_-}}X_{\bff_-}Y_{\bff_-}^{\sff_{\bff_-}}.
$$
The definition here involves locally choosing a branch of $\sfa_{\bff}$-th root of $Y_{\bff_\pm}$ which we will specify later in Section \ref{sect:LocalExpansion}. On $U_{\bfeta}$, the covering $\sP$ restricts to a degree-$\sfa_{\tau}$ covering of annuli. In this case recall that $\sfa_\tau = \pm \sfa_{\bff_\pm}$. We may compute that
$$
    \frac{T_1(\sigma_+)^{\sfa} T_2(\sigma_+)^{\sfb}}{T_1(\sigma_-)^{\sfa} T_2(\sigma_-)^{\sfb}} = q^{\sfa[l_\tau]}
$$
(cf. \cite[Lemma 4.4]{Yu25}) and thus
$$
    \hX_{\bff_+, \eta} \hX_{\bff_-, \eta} = q^{[l_\tau]}.
$$
We set 
$$
    \hX_{\bfeta} \coloneqq \hX_{\bff_+, \eta}.
$$
Then
$$
    \hX_{\bff_-, \eta} = q^{[l_\tau]} \hX_{\bfeta}^{-1}.
$$

See Figures \ref{fig:CoveringC3}, \ref{fig:CoveringP1} for the coverings induced by the framings in the examples in Figure \ref{fig:Framing}.

\begin{figure}[htb]
$$
	\begin{tikzpicture}[scale=0.8]
		\draw (0,0) -- (1,0) -- (2,0.5);
        \draw (0,-1) -- (1,0);
        \node[left] at (0,0){$1$};
        \node[left] at (0,-1){$-2$};
        \node[right] at (2,0.5){$-\frac{1}{2}$}; 
        \node at (1,0){$\bullet$};

        \coordinate (C1) at (5,1);

        \draw ($(C1)+(-1,0.12)$) .. controls ($(C1)+(-0.3,0.05)$) and ($(C1)+(0.3,0.15)$) .. ($(C1)+(1,0.62)$);
        \draw ($(C1)+(-1,-0.12)$) .. controls ($(C1)+(-0.2,-0.05)$) and ($(C1)+(-0.2,-0.15)$) .. ($(C1)+(-1,-0.88)$);
        \draw ($(C1)+(-1,-1.12)$) .. controls ($(C1)+(-0.3,-0.5)$) and ($(C1)+(0.2,-0.1)$) .. ($(C1)+(1,0.38)$);

        \draw[draw=gray, very thick] ($(C1)+(-1,0)$) ellipse [x radius=0.06, y radius=0.12];
        \draw[draw=gray, very thick] ($(C1)+(-1,-1)$) ellipse [x radius=0.06, y radius=0.12];
        \draw[draw=gray, very thick] ($(C1)+(1,0.5)$) ellipse [x radius=0.06, y radius=0.12];

        \draw[->] ($(C1)+(0,-1.3)$) -- ($(C1)+(0,-2.3)$);
        \node[left] at ($(C1)+(0,-1.8)$){$\sP$};
        \node[right] at ($(C1)+(0,-1.8)$){$2:1$};

        \draw ($(C1)+(-1,-2.88)$) -- ($(C1)+(1,-2.88)$);
        \draw ($(C1)+(-1,-3.12)$) -- ($(C1)+(1,-3.12)$);
        \draw[draw=gray, very thick] ($(C1)+(-1,-3)$) ellipse [x radius=0.06, y radius=0.12];
        \draw[draw=gray, very thick] ($(C1)+(1,-3)$) ellipse [x radius=0.06, y radius=0.12];

        \draw (9,0) -- (10,0) -- (11,0.33);
        \draw (9,-0.5) -- (10,0);
        \node[left] at (9,0){$2$};
        \node[left,below] at (9,-0.5){$-\frac{3}{2}$};
        \node[right] at (11,0.33){$-\frac{1}{3}$}; 
        \node at (10,0){$\bullet$};

        \coordinate (C2) at (14,1);

        \draw ($(C2)+(-1,0.12)$) .. controls ($(C2)+(-0.35,0.05)$) and ($(C2)+(0.3,0.15)$) .. ($(C2)+(1,0.45)$);
        \draw ($(C2)+(-1,-0.12)$) .. controls ($(C2)+(-0.2,0)$) and ($(C2)+(-0.2,-0.1)$) .. ($(C2)+(-1,-0.38)$);
        \draw ($(C2)+(-1,-0.62)$) .. controls ($(C2)+(-0.25,-0.2)$) and ($(C2)+(0.3,0)$) .. ($(C2)+(1,0.21)$);

        \draw[draw=gray, very thick] ($(C2)+(-1,0)$) ellipse [x radius=0.06, y radius=0.12];
        \draw[draw=gray, very thick] ($(C2)+(-1,-0.5)$) ellipse [x radius=0.06, y radius=0.12];
        \draw[draw=gray, very thick] ($(C2)+(1,0.33)$) ellipse [x radius=0.06, y radius=0.12];

        \draw[->] ($(C2)+(0,-1.3)$) -- ($(C2)+(0,-2.3)$);
        \node[left] at ($(C2)+(0,-1.8)$){$\sP$};
        \node[right] at ($(C2)+(0,-1.8)$){$3:1$};

        \draw ($(C2)+(-1,-2.88)$) -- ($(C2)+(1,-2.88)$);
        \draw ($(C2)+(-1,-3.12)$) -- ($(C2)+(1,-3.12)$);
        \draw[draw=gray, very thick] ($(C2)+(-1,-3)$) ellipse [x radius=0.06, y radius=0.12];
        \draw[draw=gray, very thick] ($(C2)+(1,-3)$) ellipse [x radius=0.06, y radius=0.12];
    \end{tikzpicture}
$$
\caption{The coverings $\sP$ for two choices of framings for $\bC^3$. The local regions $U_{\bfeta}$ are illustrated by the thickened circles and the degree of $\sP$ on each $U_{\bfeta}$ is given by the corresponding $\sfa_\tau$.}
\label{fig:CoveringC3}
\end{figure}
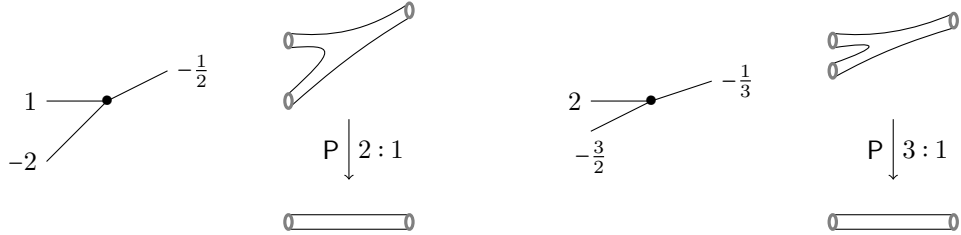

\begin{figure}[htb]
$$
    \begin{tikzpicture}
        \draw (9,-0.5) -- (10,-0.5) -- (12,0.5) -- (13,0.5);
        \draw (9,-1.5) -- (10,-0.5);
        \draw (12,0.5) -- (13,1.5);
        \node at (10,-0.5){$\bullet$};
        \node at (12,0.5){$\bullet$};
        \node[left] at (9,-0.5){$1$};
        \node[left] at (9,-1.5){$-2$};
        \node[right] at (13,0.5){$1$};
        \node[right] at (13,1.5){$-2$};
        
        \node[below] at (11,0){$-\frac{1}{2}$}; 

         \coordinate (C1) at (17,0.8);

        \draw ($(C1)+(-2,-0.38)$) .. controls ($(C1)+(-1.3,-0.45)$) and ($(C1)+(-0.7,-0.35)$) .. ($(C1)+(0,0.12)$);
        \draw ($(C1)+(-2,-0.62)$) .. controls ($(C1)+(-1.2,-0.55)$) and ($(C1)+(-1.2,-0.65)$) .. ($(C1)+(-2,-1.38)$);
        \draw ($(C1)+(-2,-1.62)$) .. controls ($(C1)+(-1.3,-1)$) and ($(C1)+(-0.8,-0.6)$) .. ($(C1)+(0,-0.12)$);
        \draw ($(C1)+(2,0.38)$) .. controls ($(C1)+(1.3,0.45)$) and ($(C1)+(0.7,0.35)$) .. ($(C1)+(0,-0.12)$);
        \draw ($(C1)+(2,0.62)$) .. controls ($(C1)+(1.2,0.55)$) and ($(C1)+(1.2,0.65)$) .. ($(C1)+(2,1.38)$);
        \draw ($(C1)+(2,1.62)$) .. controls ($(C1)+(1.3,1)$) and ($(C1)+(0.8,0.6)$) .. ($(C1)+(0,0.12)$);

        \draw[draw=gray, very thick] ($(C1)+(-2,-0.5)$) ellipse [x radius=0.06, y radius=0.12];
        \draw[draw=gray, very thick] ($(C1)+(-2,-1.5)$) ellipse [x radius=0.06, y radius=0.12];
        \draw[draw=gray, very thick] ($(C1)+(0,0)$) ellipse [x radius=0.06, y radius=0.12];
        \draw[draw=gray, very thick] ($(C1)+(2,0.5)$) ellipse [x radius=0.06, y radius=0.12];
        \draw[draw=gray, very thick] ($(C1)+(2,1.5)$) ellipse [x radius=0.06, y radius=0.12];

        \draw[->] ($(C1)+(0,-1.3)$) -- ($(C1)+(0,-2.3)$);
        \node[left] at ($(C1)+(0,-1.8)$){$\sP$};
        \node[right] at ($(C1)+(0,-1.8)$){$2:1$};

        \draw ($(C1)+(-2,-2.88)$) -- ($(C1)+(2,-2.88)$);
        \draw ($(C1)+(-2,-3.12)$) -- ($(C1)+(2,-3.12)$);
        \draw[draw=gray, very thick] ($(C1)+(-2,-3)$) ellipse [x radius=0.06, y radius=0.12];
        \draw[draw=gray, very thick] ($(C1)+(2,-3)$) ellipse [x radius=0.06, y radius=0.12];
        \draw[draw=gray, very thick] ($(C1)+(0,-3)$) ellipse [x radius=0.06, y radius=0.12];
    \end{tikzpicture}
$$
\caption{The covering $\sP$ for a choice of framing for the resolved conifold. The degree on the inner region is two while the degree on each outer region is one.}
\label{fig:CoveringP1}
\end{figure}
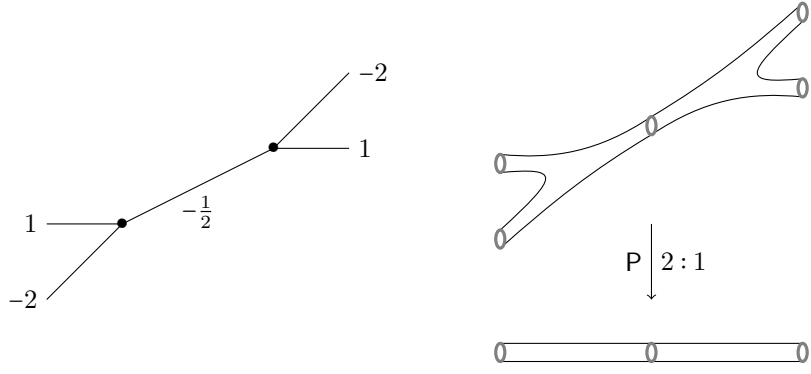

\subsection{Closed mirror map}
We now define the \emph{closed mirror map} which relates the B-model moduli parameters $q$ to the small quantum cohomology parameter $\btau$. It is given by the first-order term in the asymptotics expansion in $z^{-1}$ of the $\bT'$-equivariant \emph{small $I$-function} of $\cX$ \cite{Iritani09,CCIT15}
$$	
    I_{\bT'}(q,z) \coloneqq e^{(\sum_{a = 1}^{\fp'} \Hbar_a^{\bT'}\log q_a)/z} \sum_{\beta \in \bK_{\eff}} q^\beta \prod_{i = 1}^{3 + \fp'} \frac{\Gamma\left(1 + \frac{\Dbar_i^{\bT'}}{z} - \{-\inner{D_i, \beta}\} \right)}{\Gamma\left(1 + \frac{\Dbar_i^{\bT'}}{z} + \inner{D_i, \beta} \right)} \prod_{i = 4+\fp'}^{3+\fp} \frac{\Gamma\left(1 - \{-\inner{D_i, \beta}\} \right)}{\Gamma\left(1 + \inner{D_i, \beta} \right)} \frac{\one_{v(\beta)}}{z^{\age(v(\beta))}},
$$
i.e.,
$$
	I_{\bT'}(q,z) = \one + z^{-1}\btau(q) + o(z^{-1}).
$$
By the $\bT'$-equivariant small toric mirror theorem \cite{Givental98,CCK15,CCIT15}, we have
$$
    J_{\bT'}(\btau, z) = I_{\bT'}(q, z)
$$
under the closed mirror map $\btau = \btau(q)$ where $J_{\bT'}$ is the $\bT'$-equivariant \emph{small $J$-function} of $\cX$ defined by
$$
	(J_{\bT'}(\btau,z),b)_{\cX,\bT'} = (\one,\cS(b))_{\cX,\bT'}
$$
for any $b\in H_{\CR,\bT'}^*(\cX;\bST)$.

We give an explicit expression for the closed mirror map $\btau = \btau(q)$. For $i = 1, \dots, 3+\fp'$, set
$$
    \lambda_i = \Dbar_i^{\bT'} \big|_{\fp_{\sigma_0}} \in H^2_{\bT'}(\pt; \bQ)
$$
which is nonzero only for $i \in I'_{\sigma_0}$. For $i = 1, \dots, 3+\fp'$, let
$$
    \Omega_i \coloneqq \{\beta \in \bK_{\eff} : v(\beta) = 0, \inner{D_i, \beta} <0, \inner{D_j, \beta} \text{ for } j \in \{1, \dots, 3+\fp\} \setminus \{i\} \}
$$
and
$$
    A_i(q) \coloneqq \sum_{\beta \in \Omega_i}  q^\beta \frac{(-1)^{-\inner{D_i, \beta} -1}(-\inner{D_i, \beta} - 1)!}{\prod_{j \in \{1, \dots, 3 + \fp\} \setminus \{i\}} \inner{D_j, \beta}!}.
$$
For $i = 4 + \fp', \dots, 3+\fp$, let
$$
    \Omega_i \coloneqq \{\beta \in \bK_{\eff} : v(\beta) = b_i,  \inner{D_j, \beta} \not \in \bZ_{<0} \text{ for } j \in \{1, \dots, 3+\fp\} \}
$$
and
$$
    A_i(q) \coloneqq \sum_{\beta \in \Omega_i}  q^\beta \prod_{j = 1}^{3+\fp} \frac{\Gamma\left(1 - \{-\inner{D_j, \beta}\} \right)}{\Gamma\left(1 + \inner{D_j, \beta} \right)}.
$$
Based on the decomposition \eqref{eqn:btauDecompose}, we write
\begin{equation}\label{eqn:btauHBasis}
    \btau = \btau_0 + \sum_{a = 1}^{\text{$\fp'$}} \btau_a \Hbar_a^{\bT'} + \sum_{a = \text{$\fp'$}+1}^{\text{$\fp$}} \btau_a \one_{b_{3+a}}.
\end{equation}
Then
$$
    \btau_0(q) = \sum_{i = 1}^{3+\fp'} \lambda_i A_i(q), \qquad
    \btau_a(q) = \begin{cases}
        \log q_a + \sum_{i = 1}^{3+\fp'} m_i^{(a)} A_i(q) & \text{if } 1 \le a \le \fp',\\
        A_{3+a}(q) & \text{if } \fp'+1 \le a \le \fp.
    \end{cases}
$$

\subsection{Open mirror map}
For $\bff = (\tau, \sigma) \in F(\Sigma)$, and $\eta \in G_\tau^*$, the \emph{open mirror map} is defined by
$$
    \log \tX_{\bff, \eta} = \log \hX_{\bff, \eta} + \sum_{i = 1}^{3 + \text{$\fp'$}} \frac{\Dbar_i^{\bT'} \big|_{\text{$\fp$}_\sigma}}{\su_1^{\bff}} \bigg|_{\bT_{\bff}} A_i(q) = \log \hX_{\bff, \eta} + \sum_{j = 1}^{3} w_j^{\bff} A_{i_j^{\bff}}(q). 
$$
The definition is independent of $\eta$. For $\bff_0$, the definition reads
$$
    \log \tX_{\bff_0, \eta} = \log \hX_{\bff_0, \eta} + \frac{1}{\fr} A_1(q) + \frac{\fs + \fr \sff}{\fr\fm} A_2(q) - \frac{\fs + \fm + \fr \sff}{\fr \fm}A_3(q).
$$
When $\cL_\tau$ is inner, we check that
$$
    \log \hX_{\bff_+, \eta} + \sum_{j = 1}^{3} w_j^{\bff_+} A_{i_j^{\bff_+}}(q) + \log \hX_{\bff_-, \eta} + \sum_{j = 1}^{3} w_j^{\bff_-} A_{i_j^{\bff_-}}(q) = \log q^{[l_\tau]} + \sum_{i = 1}^{3 + \fp'} \inner{D_i, [l_\tau]} A_i(q)
$$
which agrees with
$$
    \log \tX_{\bff_+, \eta} + \log \tX_{\bff_-, \eta} = \inner{\btau', [l_\tau]} 
$$
under the closed mirror map.

Let $\bfeta = (\tau, \eta) \in J_\Sigma$. Let $\bff$ be the unique choice when $\cL_\tau$ is outer and $\bff = \bff_+$ when $\cL_\tau$ is inner. We have
$$
    \log \tX_{\bfeta} = \log \hX_{\bfeta} + \sum_{j = 1}^{3} w_j^{\bff} A_{i_j^{\bff}}(q). 
$$

\begin{remark}\rm{
Under the closed mirror map, the coordinates $q$ on the secondary variety can be viewed as moduli parameters for $\cX$. In addition, when $\cX$ is smooth, the open mirror map realizes the regions $U_{\bfeta}$ on the mirror curve as the complexified moduli space of Aganagic-Vafa branes in $\cX$. Indeed, the tropicalization of $C_q$, or the amoeba, is a region in $\bR^2$ close to the toric graph of $\cX$ \cite{IV96,Mikhalkin00,Viro01} and $U_{\bfeta}$ is contained in a tube whose image is around the edge corresponding to the 2-cone $\tau$. The tubes arise alternatively in the cell decomposition of $C_q$ induced by the Morse function $\Re(\hx)$. When $\cX$ is in general an orbifold, the mirror curve can be viewed as the complexified moduli space of pairs of an Aganagic-Vafa brane together with a flat $U(1)$-connection.
}\end{remark}

\subsection{Open curve classes and twisting}
Let $\tau \in \Sigma(2)$. Let $\bff$ be the unique choice when $\cL_\tau$ is outer and $\bff = \bff_+$ when $\cL_\tau$ is inner. Define
$$
    \bK_\tau \coloneqq \{(\beta, d) \in \bK_{\eff, \sigma} \times \bZ_{\neq 0} : \inner{D_{i_1^{\bff}}, \beta} + dw_1^{\bff} \in \bZ_{\ge 0}\}.
$$
For $(\beta, d) \in \bK_\tau$, we have
$$
    (\inner{D_{i_2^{\bff}}, \beta} + dw_2^{\bff}) + (\inner{D_{i_3^{\bff}}, \beta} + dw_3^{\bff}) \in \bZ.
$$
Let $\lambda = \lambda(\beta, d) \in G_\tau$ be the unique element such that $v(\beta) = h(d, \lambda(\beta, d))$, where $h$ is taken for the flag $\bff$. This condition implies that
$$
    \inner{D_{i_2^{\bff}}, \beta} + d\frac{\fs_{\bff}}{\fr_{\bff}\fm_\tau} - \frac{\lambdabar}{\fm_\tau} \in \bZ, \qquad \inner{D_{i_3^{\bff}}, \beta} + d\left(- \frac{1}{\fr_{\bff}} - \frac{\fs_{\bff}}{\fr_{\bff}\fm_\tau}\right) + \frac{\lambdabar}{\fm_\tau} \in \bZ 
$$
by our convention for $\lambdabar$ in Section \ref{sect:TwistingFactor}. It follows that
$$
    \{\inner{D_{i_2^{\bff}}, \beta} + dw_2^{\bff}\} = \left\{\frac{d\sff_{\bff}}{\fm_\tau} + \frac{\lambdabar}{\fm_\tau} \right\}, \qquad 
    \{\inner{D_{i_3^{\bff}}, \beta} + dw_3^{\bff}\} = \left\{ - \frac{d\sff_{\bff}}{\fm_\tau} - \frac{\lambdabar}{\fm_\tau} \right\}.
$$

Let $\eta \in G_\tau^*$. When $\cL_\tau$ is outer, the twisting factor $\eps_\eta^{\bff}$ defined in Section \ref{sect:TwistingFactor} reads
$$
    \eps_\eta^{\bff}(d, \lambda) = \xi_{\sfa_{\bff}}^{\sfa_{\bff}\{-\frac{d}{\sfa_{\bff}}\}} (\xi_{\sfa_{\bff}\fm_\tau}\omega_{\sfa_{\bff}\fm_\tau}^{\etabar})^{\sfa_{\bff}\fm_\tau \{\inner{D_{i_3^{\bff}}, \beta} + dw_3^{\bff}\}}.
$$
We have
$$
    (-1)^{\floor{\inner{D_{i_3^{\bff}}, \beta} + dw_3^{\bff}}+ \ceil{\frac{d}{\sfa_{\bff}}}} \eps_\eta^{\bff}(d, \lambda) = \xi_{\sfa_{\bff}}^{-d} (\xi_{\sfa_{\bff}\fm_\tau}\omega_{\sfa_{\bff}\fm_\tau}^{\etabar})^{\sfa_{\bff}\fm_\tau (\inner{D_{i_3^{\bff}}, \beta} + dw_3^{\bff})}.
$$
When $\cL_\tau$ is inner, the twisting factors $\eps_\eta^{\bff_\pm}$ read
$$
    \eps_\eta^{\bff_+}(d, \lambda) = \xi_{\sfa_{\bff_+}}^{\sfa_{\bff_+}\{-\frac{d}{\sfa_{\bff_+}}\}} (\xi_{\sfa_{\bff_+}\fm_\tau}\omega_{\sfa_{\bff_+}\fm_\tau}^{\etabar})^{\sfa_{\bff_+}\fm_\tau \{\inner{D_{i_3^{\bff_+}}, \beta} + dw_3^{\bff_+}\}},
$$
$$
    \eps_\eta^{\bff_-}(d, \lambda) = \xi_{\sfa_{\bff_-}}^{\sfa_{\bff_-}\{-\frac{d}{\sfa_{\bff_-}}\}} (\xi_{\sfa_{\bff_-}\fm_\tau}\omega_{\sfa_{\bff_-}\fm_\tau}^{\etabar})^{- \sfa_{\bff_-}\fm_\tau \{\inner{D_{i_3^{\bff_-}}, \beta} + dw_3^{\bff_-}\}}.
$$
We have
$$
    (-1)^{\floor{\inner{D_{i_3^{\bff_+}}, \beta} + dw_3^{\bff_+}}+ \ceil{\frac{d}{\sfa_{\bff_+}}}} \eps_\eta^{\bff_+}(d, \lambda) = \xi_{\sfa_{\bff_+}}^{-d} (\xi_{\sfa_{\bff_+}\fm_\tau}\omega_{\sfa_{\bff_+}\fm_\tau}^{\etabar})^{\sfa_{\bff_+}\fm_\tau (\inner{D_{i_3^{\bff_+}}, \beta} + dw_3^{\bff_+})},
$$
$$
    (-1)^{\floor{\inner{D_{i_3^{\bff_-}}, \beta} + dw_3^{\bff_-}}+ \ceil{\frac{d}{\sfa_{\bff_-}}}} \eps_\eta^{\bff_-}(d, \lambda) = \xi_{\sfa_{\bff_-}}^{-d} (\xi_{\sfa_{\bff_-}\fm_\tau}\omega_{\sfa_{\bff_-}\fm_\tau}^{\etabar})^{\sfa_{\bff_-}\fm_\tau (- \inner{D_{i_3^{\bff_-}}, \beta} - dw_3^{\bff_-})}.
$$
Again note the extra sign in the expressions for $\bff_-$.

\subsection{Pullback of disk potentials under open-closed mirror map}
Now we compute the pullback of the disk potentials under the open-closed mirror map. The statements in this subsection are direct generalizations of the results in \cite[Section 4.2]{FLT22}. Let $\bfeta = (\tau, \eta) \in J_\Sigma$. When $\cL_\tau$ is outer, the A-model disk potential defined in \eqref{eqn:AmodelFgn} is
$$
    F_{0,1}(\btau; \tX_{\bfeta}) \coloneqq \sum_{\beta \in E(\cX)} \sum_{\mu \in \bZ_{> 0}} \sum_{\lambda \in G_\tau} \sum_{l \in \bZ_{\ge 0}} \frac{\inner{(\btau)^l}_{0, \beta, (\mu, \lambda)}}{l!} \eps_\eta^{\bff}(\mu, \lambda) \tX_{\bfeta}^{\mu}.
$$
When $\cL_\tau$ is inner, the A-model disk potential is
\begin{align*}
    F_{0,1}(\btau; \tX_{\bfeta}) \coloneqq & \sum_{\beta \in E(\cX)} \sum_{\mu \in \bZ_{> 0}} \sum_{\lambda \in G_\tau} \sum_{l \in \bZ_{\ge 0}} \frac{\inner{(\btau)^l}_{0, \beta, (\mu, \lambda)}}{l!} \eps_\eta^{\bff_+}(\mu, \lambda)  \tX_{\bfeta}^\mu\\
    & + \sum_{\beta \in E(\cX)} \sum_{\mu \in \bZ_{< 0}} \sum_{\lambda \in G_\tau} \sum_{l \in \bZ_{\ge 0}} \frac{\inner{(\btau)^l}_{0, \beta, (\mu, \lambda)}}{l!} \eps_\eta^{\bff_-}(-\mu, \lambda)  e^{-\mu\int_{l_\tau}\btau'} \tX_{\bfeta}^\mu.
\end{align*}

Under the open-closed mirror map, we may compute that
$$
    F_{0,1}(\btau; \tX_{\bfeta}) = W_{\bfeta}(q, \hX_{\bfeta})
$$
where the right-hand side is defined as
\begin{align*}
    W_{\bfeta}(q, \hX_{\bfeta}) & \coloneqq \sum_{(\beta, d) \in \bK_\tau} q^\beta \hX_{\bfeta}^{d} \frac{(-1)^{\floor{\inner{D_{i_3^{\bff}}, \beta} + dw_3^{\bff}}+ \ceil{\frac{d}{\sfa_{\bff}}}} \eps_\eta^{\bff}(d, \lambda(\beta, d)) }{\fm_\tau d (\inner{D_{i_1^{\bff}}, \beta} + dw_1^{\bff})! \prod_{i \in I_\sigma} \inner{D_i, \beta}!} \cdot \frac{\Gamma(-\inner{D_{i_3^{\bff}}, \beta} - dw_3^{\bff} )}{\Gamma(\inner{D_{i_2^{\bff}}, \beta} + dw_2^{\bff} + 1)}\\
    & = \sum_{(\beta, d) \in \bK_\tau} q^\beta \hX_{\bfeta}^{d} \frac{\xi_{\sfa_{\bff}}^{-d} (\xi_{\sfa_{\bff}\fm_\tau}\omega_{\sfa_{\bff}\fm_\tau}^{\etabar})^{\sfa_{\bff}\fm_\tau (\inner{D_{i_3^{\bff}}, \beta} + dw_3^{\bff})} }{\fm_\tau d (\inner{D_{i_1^{\bff}}, \beta} + dw_1^{\bff})! \prod_{i \in I_\sigma} \inner{D_i, \beta}!} \cdot \frac{\Gamma(-\inner{D_{i_3^{\bff}}, \beta} - dw_3^{\bff} )}{\Gamma(\inner{D_{i_2^{\bff}}, \beta} + dw_2^{\bff} + 1)}.
\end{align*}

In more detail, when $\cL_\tau$ is outer, the condition $\inner{D_{i_1^{\bff}}, \beta} + dw_1^{\bff} \in \bZ_{\ge 0}$ implies that $d >0$, and thus
$$
    \bK_\tau = \{(\beta, d) \in \bK_{\eff, \sigma} \times \bZ_{> 0} : \inner{D_{i_1^{\bff}}, \beta} + dw_1^{\bff} \in \bZ_{\ge 0}\}.
$$
On the other hand, when $\cL_\tau$ is inner, the subset of $\bK_\tau$ with $d>0$ is identified with
$$
    \bK_\tau^+ \coloneqq \{(\beta, d) \in \bK_{\eff, \sigma_+} \times \bZ_{>0} : \inner{D_{i_1^{\bff_+}}, \beta} + dw_1^{\bff_+} \in \bZ_{\ge 0}\}.
$$
The subset of $\bK_\tau$ with $d < 0$ is the image of the injective map
$$
    \bK_\tau^- \coloneqq \{(\beta, d) \in \bK_{\eff, \sigma_-} \times \bZ_{>0} : \inner{D_{i_1^{\bff_-}}, \beta} + dw_1^{\bff_-} \in \bZ_{\ge 0}\} \longrightarrow \bK_\tau, \qquad (\beta, d) \longmapsto (\beta + d[l_\tau] , -d).
$$
For $(\beta, d) \in \bK_\tau^\pm$, let $\lambda_\pm(\beta, d) \in G_\tau$ to be the unique element such that $v(\beta) = h_\pm(d, \lambda_\pm(\beta, d))$, where $h_\pm$ is taken for the flag $\bff_\pm$. Then the pullback of $F_{0,1}$ is 
\begin{align*}
    W_{\bfeta}(q, \hX_{\bfeta}) = & \sum_{(\beta, d) \in \bK_\tau^+} q^\beta \hX_{\bfeta}^d \frac{(-1)^{\floor{\inner{D_{i_3^{\bff_+}}, \beta} + dw_3^{\bff_+}} + \ceil{\frac{d}{\sfa_{\bff_+}}}} \eps_\eta^{\bff_+}(d, \lambda_+(\beta, d)) }{\fm_\tau d (\inner{D_{i_1^{\bff_+}}, \beta} + dw_1^{\bff_+})! \prod_{i \in I_{\sigma_+}} \inner{D_i, \beta}!} \cdot \frac{\Gamma(-\inner{D_{i_3^{\bff_+}}, \beta} - dw_3^{\bff_+})}{\Gamma(\inner{D_{i_2^{\bff_+}}, \beta} + dw_2^{\bff_+} + 1)}\\
    & + \sum_{(\beta, d) \in \bK_\tau^-} q^{\beta + d[l_\tau]} \hX_{\bfeta}^{-d} \frac{(-1)^{\floor{\inner{D_{i_3^{\bff_-}}, \beta} + dw_3^{\bff_-}} + \ceil{\frac{d}{\sfa_{\bff_-}}}} \eps_\eta^{\bff_-}(d, \lambda_-(\beta, d)) }{\fm_\tau d (\inner{D_{i_1^{\bff_-}}, \beta} + dw_1^{\bff_-})! \prod_{i \in I_{\sigma_-}} \inner{D_i, \beta}!} \cdot \frac{\Gamma(-\inner{D_{i_3^{\bff_-}}, \beta} - dw_3^{\bff_-})}{\Gamma(\inner{D_{i_2^{\bff_-}}, \beta} + dw_2^{\bff_-} + 1)}\\
    = & \sum_{(\beta, d) \in \bK_\tau^+} q^\beta \hX_{\bfeta}^d \frac{\xi_{\sfa_{\bff_+}}^{-d} (\xi_{\sfa_{\bff_+}\fm_\tau}\omega_{\sfa_{\bff_+}\fm_\tau}^{\etabar})^{\sfa_{\bff_+}\fm_\tau (\inner{D_{i_3^{\bff_+}}, \beta} + dw_3^{\bff_+})} }{\fm_\tau d (\inner{D_{i_1^{\bff_+}}, \beta} + dw_1^{\bff_+})! \prod_{i \in I_{\sigma_+}} \inner{D_i, \beta}!} \cdot \frac{\Gamma(-\inner{D_{i_3^{\bff_+}}, \beta} - dw_3^{\bff_+})}{\Gamma(\inner{D_{i_2^{\bff_+}}, \beta} + dw_2^{\bff_+} + 1)}\\
    & + \sum_{(\beta, d) \in \bK_\tau^-} q^{\beta + d[l_\tau]} \hX_{\bfeta}^{-d} \frac{\xi_{\sfa_{\bff_-}}^{-d} (\xi_{\sfa_{\bff_-}\fm_\tau}\omega_{\sfa_{\bff_-}\fm_\tau}^{\etabar})^{\sfa_{\bff_-}\fm_\tau (- \inner{D_{i_3^{\bff_-}}, \beta} - dw_3^{\bff_-})} }{\fm_\tau d (\inner{D_{i_1^{\bff_-}}, \beta} + dw_1^{\bff_-})! \prod_{i \in I_{\sigma_-}} \inner{D_i, \beta}!} \cdot \frac{\Gamma(-\inner{D_{i_3^{\bff_-}}, \beta} - dw_3^{\bff_-})}{\Gamma(\inner{D_{i_2^{\bff_-}}, \beta} + dw_2^{\bff_-} + 1)}.
\end{align*}
To derive the expression of $W_{\bfeta}$ above, for $(\beta, d) \in \bK_\tau^-$, we use the relations
$$
    \inner{D_{i_1^{\bff_+}}, \beta} = \inner{D_{i_1^{\bff_+}}, \beta + d[l_\tau]} - d w_1^{\bff_+}, \qquad 
    \inner{D_{i_1^{\bff_-}}, \beta} + dw_1^{\bff_-} = \inner{D_{i_1^{\bff_-}}, \beta + d[l_\tau]},
$$
$$
    \inner{D_{i_3^{\bff_-}}, \beta} + d w_3^{\bff_-} = \inner{D_{i_2^{\bff_+}}, \beta + d[l_\tau]} - d w_2^{\bff_+}, \qquad
    \inner{D_{i_2^{\bff_-}}, \beta} + d w_2^{\bff_-} = \inner{D_{i_3^{\bff_+}}, \beta + d[l_\tau]} - d w_3^{\bff_+},
$$
and $\inner{D_i, \beta} = \inner{D_i, \beta + d[l_\tau]}$ for $i \not \in \{i_1^{\bff_+}, i_1^{\bff_-}, i_2^{\bff_+}, i_3^{\bff_+}\}$. Then we have
\begin{align*}
    & \frac{\xi_{\sfa_{\bff_-}}^{-d} (\xi_{\sfa_{\bff_-}\fm_\tau}\omega_{\sfa_{\bff_-}\fm_\tau}^{\etabar})^{\sfa_{\bff_-}\fm_\tau (- \inner{D_{i_3^{\bff_-}}, \beta} - dw_3^{\bff_-})} }{\fm_\tau d (\inner{D_{i_1^{\bff_-}}, \beta} + dw_1^{\bff_-})! \prod_{i \in I_{\sigma_-}} \inner{D_i, \beta}!} \cdot \frac{\Gamma(-\inner{D_{i_3^{\bff_-}}, \beta} - dw_3^{\bff_-})}{\Gamma(\inner{D_{i_2^{\bff_-}}, \beta} + dw_2^{\bff_-} + 1)} \\
    & = (-1)^{1 + (\inner{D_{i_2^{\bff_+}}, \beta + d[l_\tau]} - d w_2^{\bff_+}) + (\inner{D_{i_3^{\bff_+}}, \beta + d[l_\tau]} - d w_3^{\bff_+})} \\
    & \qquad \cdot \frac{\xi_{\sfa_{\bff_+}}^{d} (\xi_{\sfa_{\bff_+}\fm_\tau}\omega_{\sfa_{\bff_+}\fm_\tau}^{\etabar})^{\sfa_{\bff_+}\fm_\tau (- \inner{D_{i_2^{\bff_+}}, \beta + d[l_\tau]} + d w_2^{\bff_+})} }{\fm_\tau d (\inner{D_{i_1^{\bff_+}}, \beta + d[l_\tau]} - dw_1^{\bff_+})! \prod_{i \in I_{\sigma_+}} \inner{D_i, \beta + d[l_\tau]}!} \cdot \frac{\Gamma(-\inner{D_{i_3^{\bff_+}}, \beta + d[l_\tau]} + dw_3^{\bff_+})}{\Gamma(\inner{D_{i_2^{\bff_+}}, \beta + d[l_\tau]} - dw_2^{\bff_+} + 1)} \\
    & = \frac{\xi_{\sfa_{\bff_+}}^{d} (\xi_{\sfa_{\bff_+}\fm_\tau}\omega_{\sfa_{\bff_+}\fm_\tau}^{\etabar})^{\sfa_{\bff_+}\fm_\tau (\inner{D_{i_3^{\bff_+}}, \beta + d[l_\tau]} - d w_3^{\bff_+})} }{\fm_\tau (-d) (\inner{D_{i_1^{\bff_+}}, \beta + d[l_\tau]} - dw_1^{\bff_+})! \prod_{i \in I_{\sigma_+}} \inner{D_i, \beta + d[l_\tau]}!} \cdot \frac{\Gamma(-\inner{D_{i_3^{\bff_+}}, \beta + d[l_\tau]} + dw_3^{\bff_+})}{\Gamma(\inner{D_{i_2^{\bff_+}}, \beta + d[l_\tau]} - dw_2^{\bff_+} + 1)} 
\end{align*}
which is the contribution of $(\beta + d[l_\tau], -d) \in \bK_\tau$.

\subsection{Local expansions on the mirror curve}\label{sect:LocalExpansion}
We now obtain the functions $W_{\bfeta}$ from local expansions of coordinates on the mirror curve and thereby establish the disk mirror theorem. Let $\bfeta = (\tau, \eta) \in J_\Sigma$. Let $\bff$ be the unique choice when $\cL_\tau$ is outer and $\bff = \bff_+$ when $\cL_\tau$ is inner. Locally on $U_{\bfeta}$, we assume that
\begin{equation}\label{eqn:LogY}
    \log Y_{\bff} = - \frac{\pi\sqrt{-1}}{\fm_\tau} + \frac{2\pi\sqrt{-1}}{\fm_\tau} \etabar + \frac{v(q, \hX_{\bff, \eta})}{\fm_\tau}
\end{equation}
where $v$ is a power series in $q$ and power/Laurent series in $\hX_{\bff, \eta}$. We choose an $\sfa_{\bff}$-th root of $Y_{\bff}$ such that
$$
    \log Y_{\bff}^{1/\sfa_{\bff}} = - \frac{\pi\sqrt{-1}}{\sfa_{\bff}\fm_\tau} + \frac{2\pi\sqrt{-1}}{\sfa_{\bff}\fm_\tau} \etabar + \frac{v(q, \hX_{\bff, \eta})}{\sfa_{\bff}\fm_\tau}.
$$
To determine the term $v(q, \hX_{\bff, \eta})$, we express the coordinates $(X_{\bff}, Y_{\bff})$ in terms of $(\hX_{\bff, \eta}, Y_{\bff}^{1/\sfa_{\bff}})$ as
$$
    X_{\bff} = \xi_{\sfa_{\bff}}^{-1} \hX_{\bff, \eta} (Y_{\bff}^{1/\sfa_{\bff}})^{-\sfb_{\bff}} , \qquad Y_{\bff} = (Y_{\bff}^{1/\sfa_{\bff}})^{\sfa_{\bff}}.
$$
Then, the mirror curve equation with respect to the flag $\bff$ is
$$
    0 = H_{\bff}(\hX_{\bff, \eta}, Y_{\bff}^{1/\sfa_{\bff}}, q) = \sum_{i=1}^{3+\fp} a_i^\si(q) \xi_{\sfa_{\bff}}^{-m_i^{\bff}} \hX_{\bff, \eta}^{m_i^{\bff}} (Y_{\bff}^{1/\sfa_{\bff}})^{-\sfb_{\bff} m_i^{\bff} + \sfa_{\bff} n_i^{\bff}}.
$$
We apply the following result.

\begin{lemma}[{\cite[Lemma 4.7]{FLT22}}] \label{lem:47}
The solution $v$ to the equation
$$
    1 - e^v + \sum_{a = 0}^k t_a e^{r_a v} = 0
$$  
around $t_0 = \cdots = t_k = 0, v= 0$ is in the following power series form
$$
    v(t_0, \dots, t_k) = \sum_{\substack{l_0, \dots, l_k \in \bZ_{\ge 0} \\ (l_0, \dots, l_k) \neq \vec{0}}} \frac{(-1)^{l_0 + \cdots + l_k - 1}\Gamma((l_0 + \cdots + l_k) - (r_0l_0 + \cdots + r_k l_k))}{l_0! \cdots l_k! \Gamma(1 - (r_0l_0 + \cdots + r_k l_k))} t_0^{l_0} \cdots t_k^{l_k}.
$$
\end{lemma}

For $i = 1, \dots, 3 + \fp$, set 
$$
    t_i = a_i^\si(q) \xi_{\sfa_{\bff}}^{-m_i^{\bff}} (\xi_{\sfa_{\bff}\fm_\tau}\omega_{\sfa_{\bff}\fm_\tau}^{\etabar})^{-\sfb_{\bff} m_i^{\bff} + \sfa_{\bff} n_i^{\bff}} \hX_{\bff, \eta}^{m_i^{\bff}}, \qquad 
    r_i = \frac{-\sff_{\bff} m_i^{\bff} + n_i^{\bff}}{\fm_\tau}.
$$
For $i \in I'_\tau$ we have
$$
    t_{i_3^{\bff}} = 1, \quad r_{i_3^{\bff}} = 0, \qquad t_{i_2^{\bff}} = -1, \quad r_{i_2^{\bff}} = 1.
$$
With these coefficients, the mirror curve equation with respect to $\bff$ is written as
$$
    0 = 1 - e^v + \sum_{i \in I_\tau} t_i e^{r_i v}.
$$
By Lemma \ref{lem:47}, we find
$$
    v = \sum_{\substack{l_i \in \bZ_{\ge 0} \\ (l_i)_{i \in I_\tau} \neq \vec{0}}}  \frac{(-1)^{-1 + \sum_{i \in I_\tau} l_i} \Gamma(\sum_{i \in I_\tau} l_i(1 - r_i))}{\prod_{i \in I_\tau} l_i! \Gamma(1 - \sum_{i \in I_\tau} r_il_i)} \prod_{i \in I_\tau} t_i^{l_i}.
$$
A class $(\beta, d) \in \bK_\tau$ corresponds to the term in the above where
$$
    l_{i_1^{\bff}} = \inner{D_{i_1^{\bff}}, \beta} + dw_1^{\bff}, \qquad \qquad 
    l_i = \inner{D_i, \beta}, \quad i \in I_\sigma.
$$
We compute
$$
    \sum_{i \in I_\tau} r_il_i = - \inner{D_{i_2^{\bff}}, \beta} -dw_2^{\bff} , \qquad 
    \sum_{i \in I_\tau} l_i(1 - r_i) = - \inner{D_{i_3^{\bff}}, \beta} - dw_3^{\bff},
$$
$$
    (-1)^{-1 + \sum_{i \in I_\tau} l_i} \prod_{i \in I_\tau} t_i^{l_i} = - q^\beta \hX_{\bff, \eta}^d  \xi_{\sfa_{\bff}}^{-d} (\xi_{\sfa_{\bff}\fm_\tau}\omega_{\sfa_{\bff}\fm_\tau}^{\etabar})^{ \sfa_{\bff} \fm_\tau(\inner{D_{i_3^{\bff}}, \beta} +dw_3^{\bff} )}.
$$
Note that there are additional terms in $v$ that would arise from classes of form $(\beta, 0)$ which are excluded from $\bK_\tau$, but these terms have no $\hX_{\bff, \eta}$-dependence. It follows that
$$
   \left(\hX_{\bfeta} \frac{\partial}{\partial \hX_{\bfeta}} \right)^2 W_{\bfeta}(q, \hX_{\bfeta}) = -  \hX_{\bfeta} \frac{\partial}{\partial \hX_{\bfeta}} \log Y_{\bff} \big|_{U_{\bfeta}}.
$$
Therefore, we arrive at the following conclusion.

\begin{theorem}\label{thm:DiskExpansion}
For any $\bfeta \in J_\Sigma$, we have
$$
    \left(\tX_{\bfeta} \frac{\partial}{\partial \tX_{\bfeta}} \right)^2 F_{0,1}(\btau; \tX_{\bfeta}) = - \hX_{\bfeta} \frac{\partial}{\partial \hX_{\bfeta}} \log Y_{\bff} \big|_{U_{\bfeta}}
$$
under the open-closed mirror map $\btau = \btau(q)$, $\tX_{\bfeta} = \tX_{\bfeta}(q, \hX_{\bfeta})$.
\end{theorem}

Theorem \ref{thm:DiskExpansion} is a generalization of \cite[Section 4.4]{FL13}, \cite[Theorem 4.5]{FLT22} to the case of fractional framings. Note that the expression for the open mirror map implies that $\tX_{\bfeta} \frac{\partial}{\partial \tX_{\bfeta}} = \hX_{\bfeta} \frac{\partial}{\partial \hX_{\bfeta}}$.


\section{All-genus open mirror symmetry}\label{sect:AllGenus}
In this section, we review the Chekhov-Eynard-Orantin topological
recursion on the mirror curve which provides the higher-genus B-model. Moreover, we prove all-genus mirror symmetry which retrieves the all-genus open Gromov-Witten potentials from the topological recursion invariants on the mirror curve (Theorem \ref{thm:AllGenusMirror}).

\subsection{Differentials of the second kind}
The main reference of this subsection is \cite{Fay73}. On the compactified mirror curve $\Cbar_q$, let $A_1,\dots, A_\fg, B_1,\dots, B_\fg \in H_1(\Cbar_q;\bC)$ be the choice of A- and B-cycles in \cite[Section 5.9]{flz2020remodeling} which form a symplectic basis of $(H_{1}(\Cbar_q;\bC), \cap)$, where $\cap$ denotes the intersection pairing. Let $B$ be the \emph{fundamental bidifferential of the second kind} on $\Cbar_q$ normalized by the A-cycles $A_1,\dots, A_\fg$. It is also referred to as the Bergman kernel in \cite{EO07,EO15}.

Let $\bsi \in I_\Si$. Following \cite{Eynard11, EO15}, let
$
    \zeta_{\bsi} =\sqrt{\hx-\check{u}^{\bsi}}
$
be a local holomorphic coordinate near the critical point $p_{\bsi} \in \Crit$. Around $p_{\bsi}$, we have
$$
  \hx = \check{u}^{\bsi} +\zeta_\bsi^2, \qquad 
  \hy = \check{v}^{\bsi} +\sum_{k \in \bZ_{\ge 1}} h^{\bsi}_k \zeta_\bsi^k
$$
where
$$
  h_1^{\bsi} = \sqrt{ \frac{2}{\frac{d^2\hat{x}}{d\hat{y}^2}(p_{\bsi}) } }. 
$$
For $k \in \bZ_{\ge 0}$, define
$$
    \theta_{\bsi}^k(p)\coloneqq -(2k-1)!! 2^{-k}\Res_{p'\to p_{\bsi}} B(p,p')\zeta_\bsi^{-2k-1}
$$
which is a differential of second kind on $\Cbar_q$. Specifically, it has a single pole of order $2k+2$ at $p_{\bsi}$. For $k \in \bZ_{\ge 1}$, define
$$
  \hxi_{\bsi}^k \coloneqq (-1)^k \left(\frac{d}{d\hx} \right)^{k-1} \frac{\theta_{\bsi}^0}{d\hx}
$$
which is a meromorphic function on $\Cbar_q$ with poles only at $\Crit$. Moreover, define
$$
  \htheta_{\bsi}^0 \coloneqq \theta_{\bsi}^0, \qquad \qquad \htheta_{\bsi}^k \coloneqq d \hxi_{\bsi}^k, \quad k \in \bZ_{\ge 1}.
$$
Note that the definitions of $\theta_{\bsi}^k$, $\hxi_{\bsi}^k$, $\htheta_{\bsi}^k$ do not depend on the choice of the reference flag $\bff_0$.

Let
$$
    \theta_{\bsi}(z) \coloneqq \sum_{k \in \bZ_{\ge 0}} \theta_{\bsi} z^k, \qquad \htheta_{\bsi}(z) \coloneqq \sum_{k \in \bZ_{\ge 0}} \htheta_{\bsi} z^k.
$$
By \cite[Proposition 6.6]{flz2020remodeling}, we have
$$
    \theta_{\bsi}(z) = \sum_{\bsi' \in I_\Sigma} \chR_{\bsi'}^{\spa \bsi}(z) \htheta_{\bsi'}(z)
$$
where $(\chR_{\bsi'}^{\spa \bsi})$ is the \emph{B-model $R$-matrix}. It is related to the A-model $R$-matrix in Section \ref{sect:AmodelRmatrix} by
$$
    R_{\bsi'}^{\spa \bsi}(z) = \chR_{\bsi'}^{\spa \bsi}(-z)
$$
\cite[Theorem 7.1]{flz2020remodeling} and we have
$$
    \theta_{\bsi}(z) = \sum_{\bsi' \in I_\Sigma} R_{\bsi'}^{\spa \bsi}(-z) \htheta_{\bsi'}(z).
$$

\subsection{Topological recursion}\label{sect:TR}
Let $\omega_{g,n}$ be defined recursively by the Chekhov-Eynard-Orantin topological recursion \cite{EO07} as follows:
$$
    \omega_{0,1}= \Phi = \hy d\hx,\qquad  \omega_{0,2}=B(p_1,p_2),
$$
and when $2g-2+n>0$,
\begin{align*}
\omega_{g,n}(p_1,\dots, p_n) = & \sum_{\bsi\in I_\Si}\Res_{p \to p_\bsi}
\frac{\int_{\xi = p}^{\bar{p}} B(p_n,\xi)}{2(\Phi(p)-\Phi(\bar{p}))}
\Bigl( \omega_{g-1,n+1}(p,\bar{p},p_1,\dots, p_{n-1}) \\
& + \sum_{g_1+g_2=g}
\sum_{ \substack{ I\cup J=\{1,..., n-1\} \\ I\cap J =\emptyset } }' \omega_{g_1,|I|+1} (p,p_I)\omega_{g_2,|J|+1}(\bar{p},p_J) \Bigr)
\end{align*}
where in the summation $\sum_{g_1, g_2} \sum_{I, J}'$, any term with $(g_1 = 0, I = \emptyset)$ or $(g_2 = 0, J = \emptyset)$ is excluded. We note that, for $(g,n) = (0,2)$, the fundamental bidifferential $\omega_{0,2}$ only has a double pole of order 2 along the diagonal, with biresidue $1$. For $2g-2+n > 0$, the only poles of $\omega_{g,n}$ (in any of the variables) are at $\Crit$.

\subsection{B-model graph sum}

We now state the graph sum formula of \cite{DOSS14} for $\omega_{g,n}$. See Section \ref{sect:AmodelDescendantGraph} for the definition of labeled graphs and the A-model graph sum. We introduce the following notation:
\begin{itemize}
\item  For $\bsi\in I_{\Sigma}$, we define
$$
  \check{h}^{\bsi}_{k} \coloneqq \frac{(2k-1)!!}{2^{k-1}}h^\bsi_{2k-1}. 
$$

\item For any $\bsi,\bsi'\in I_\Sigma$, we expand
$$
B(p_1,p_2) =\left( \frac{\delta_{\bsi,\bsi'}}{ (\zeta_\bsi-\zeta_{\bsi'})^2}
+ \sum_{k,l\in \bZ_{\geq 0}} B^{\bsi,\bsi'}_{k,l} \zeta_\bsi^k \zeta_{\bsi'}^l \right) d\zeta_\bsi d\zeta_{\bsi'}
$$
near $p_1=p_{\bsi}$ and $p_2=p_{\bsi'}$, and define
$$
  \check{B}^{\bsi,\bsi'}_{k,l} \coloneqq \frac{(2k-1)!! (2l-1)!!}{2^{k+l+1}} B^{\bsi,\bsi'}_{2k,2l}.
$$


\end{itemize}

Given a labeled graph $\vGa \in \bGa_{g,n}(\cX)$ with
$L^o(\Ga)=\{l_1,\dots,l_n\}$, 
we define its \emph{B-model weight} to be
\begin{align*}
w_B^{\bu}(\vGa) =& (-1)^{g(\vGa)-1}\prod_{v\in V(\Gamma)} \left(\frac{h^{\bsi(v)}_1}{\sqrt{-2}}\right)^{2-2g-\val(v)} \left\langle \prod_{h\in H(v)} \tau_{k(h)} \right\rangle_{g(v)}
\prod_{e\in E(\Gamma)} \check{B}^{\bsi(v_1(e)),\bsi(v_2(e))}_{k(h_1(e)),k(h_2(e))}  \\
& \cdot \prod_{l\in \cL^1(\Gamma)}(\check{\cL}^1)^{\bsi(l)}_{k(l)}
\prod_{i=1}^n (\check{\cL}^\bu)^{\bsi(l_i)}_{k(l_i)}(l_i)
\end{align*}
where
\begin{itemize}
\item (dilaton leaf)
$$
(\check{\cL}^1)^{\bsi}_k \coloneqq \frac{-1}{\sqrt{-2}}\check{h}^{\bsi}_k;
$$
\item (ordinary leaf)
$$
    (\check{\cL}^\bu)^{\bsi}_k(l_i) \coloneqq  \frac{1}{\sqrt{-2}} \theta_{\bsi}^k(p_i) = [z^k] \sum_{\brho \in I_\Sigma}  \frac{\htheta_{\brho}(z)}{\sqrt{-2}} \chR_{\brho}^{\spa \bsi}(z).
$$
\end{itemize}
Then, we have the following graph sum formula.


\begin{theorem}[{\cite[Theorem 3.7]{DOSS14}}] \label{thm:DOSS}
For $2g-2+n>0$, we have
$$
    \omega_{g,n} = \sum_{\vGa \in \bGa_{g,n}(\cX)}\frac{w_B^{\bu}(\vGa)}{|\Aut(\vGa)|}.
$$
\end{theorem}

We state the identification of A- and B-model graph sums provided by \cite{flz2020remodeling}. For $i=1, \dots, n$ and $\brho \in I_\Sigma$, introduce formal variables
$$
    \tilde{\bu}_i^\brho(z) = \sum_{k \in \bZ_{\ge 0}}(\tilde{u}_i)^\brho_k z^k \coloneqq \sum_{\bsi'\in I_\Si}
    \left(\bar{\bu}_i^{\bsi'}(z)
    S^{\widehat{\underline{\brho}} }_{\spa \bsi'}(z)\right)_+.
$$
Then, via the expression \eqref{eqn:AmodelOrdinaryLeaf}, under the relation $\frac{1}{\sqrt{-2}}\htheta_{\brho}^k(p_i) = -(\tilde{u}_i)^\brho_k$ and the closed mirror map, we have the identification of ordinary leaves
$$
    (\check{\cL}^\bu)^{\bsi}_k(l_i) = - (\cL^\bu)^{\bsi}_k(l_i).
$$

\begin{theorem}[{\cite[Theorem 7.2]{flz2020remodeling}}]\label{thm:GraphSumId}
Let $2g - 2 + n > 0$. For any $\vGa\in \bGa_{g,n}(\cX)$, we have
$$
    w^{\text{$\bu$}}_B(\vGa) \big|_{ \frac{1}{\sqrt{-2}}\htheta_{\brho}^k(p_i) = -(\tilde{u}_i)^\brho_k}
    = (-1)^{g(\vGa)-1+n} w^{\text{$\bu$}}_A(\vGa) \big|_{\substack{\su_1 = u_1 \\ \su_2 = u_2}}
$$
under the closed mirror map $\btau = \btau(q)$.
\end{theorem}

\subsection{B-model open potentials}\label{sect:BmodelOpen}
For $\bfeta \in J_\Sigma$, in the complex $\hX_{\bfeta}$-plane, let $A_{\bfeta}$ denote a small punctured disk centered at the origin when $\cL_\tau$ is outer, or a small annulus around the origin when $\cL_\tau$ is inner. In either case, there is an isomorphism
$$
    \rho_{\bfeta}\colon  A_{\bfeta} \longrightarrow U_{\bfeta} \subset C_q.
$$
We note that the constructions below, which treats the outer and inner cases simultaneously, recovers the outer case considered in \cite{flz2020remodeling} after the origin is added to $A_{\bfeta}$ and the puncture is added to $U_{\bfeta}$.

Let $g \in \bZ_{\ge 0}$ and $n \in \bZ_{> 0}$. Let $\bfeta_1, \dots, \bfeta_n \in J_\Sigma$. For $i = 1, \dots, n$, let $\hX_i$ denote a copy of the variable $\hX_{\bfeta_i}$ and $A_i$ denote a copy of $A_{\bfeta_i}$ with the isomorphism 
$$
    \rho_i\colon  A_i \longrightarrow U_{\bfeta_i}.
$$
The variable $\hX_i$ will be related to the A-model variable $\tX_i$ via the open mirror map $\tX_i = \tX_i(q, \hX_i)$ given by $\tX_{\bfeta_i} = \tX_{\bfeta_i}(q, \hX_{\bfeta_i})$.

\subsubsection{Disk potential}\label{sect:BmodelDisk}
For $(g, n) = (0, 1)$, let $\bff_1$ be the reference flag for the cone $\tau_1$, and consider
\begin{equation}\label{eqn:DiskResFree}
        - \rho_1^* \omega_{0,1} - \Res_{\hX_1 = 0} \left(-\rho_1^* \omega_{0,1} \right) \frac{d\hX_1}{\hX_1}  
        = \rho_1^*(-\log Y_{\bff_1}) \frac{d\hX_1}{\hX_1} - \Res_{\hX_1 = 0} \left( \rho_1^*(-\log Y_{\bff_1}) \frac{d\hX_1}{\hX_1} \right) \frac{d\hX_1}{\hX_1}
\end{equation}
which is a holomorphic 1-form on $A_1$ with no residue at $\hX_1 = 0$. Here, we may choose any branch of $\log Y_{\bff_1}$, such as in \eqref{eqn:LogY}, and the choice does not affect \eqref{eqn:DiskResFree} since the constant ambiguity is annihilated during the subtraction of residue. The difference between $\omega_{0,1}$ and $\hy_{\bff_1}d\hx_{\bff_1}$ is also annihilated during the subtraction of residue. Let 
$$
    \chF_{0,1}(q; \hX_1)
$$
be the holomorphic function on $A_1$ defined by the following two properties:
\begin{itemize}
    \item $\chF_{0,1}$ is a primitive of \eqref{eqn:DiskResFree}, i.e.
    $$
        d \chF_{0,1} = \rho_1^*(-\log Y_{\bff_1}) \frac{d\hX_1}{\hX_1} - \Res_{\hX_1 = 0} \left( \rho_1^*(-\log Y_{\bff_1}) \frac{d\hX_1}{\hX_1} \right) \frac{d\hX_1}{\hX_1};
    $$

    \item in the Laurent series expansion of $\chF_{0,1}$ in $\hX_1$, there are no constant terms.
    
\end{itemize}

\subsubsection{Annulus potential}\label{sect:BmodelAnn}
For $(g, n) = (0, 2)$, consider
\begin{equation}\label{eqn:AnnResFree}
    (\rho_1 \times \rho_2)^* \omega_{0,2} - \delta_{\bfeta_1, \bfeta_2} \frac{d\hX_1 d\hX_2}{(\hX_1 - \hX_2)^2}
\end{equation}
which is a holomorphic 2-form on $A_1 \times A_2$ with no residue at $\hX_1 = 0$ or $\hX_2 = 0$. Note that the diagonal double pole of $(\rho_1 \times \rho_2)^* \omega_{0,2}$ only appears when $\bfeta_1 = \bfeta_2$. Let 
$$
    \chF_{0,2}(q; \hX_1, \hX_2)
$$
be the holomorphic function on $A_1 \times A_2$ defined by the following two properties:
\begin{itemize}
    \item $\chF_{0,2}$ is a primitive of \eqref{eqn:AnnResFree}, i.e.
    $$
        \frac{\partial^2 \chF_{0,2} }{\partial \hX_1 \partial \hX_2} d\hX_1 d\hX_2  = (\rho_1 \times \rho_2)^* \omega_{0,2} - \delta_{\bfeta_1, \bfeta_2} \frac{d\hX_1 d\hX_2}{(\hX_1 - \hX_2)^2};
    $$

    \item in the Laurent series expansion of $\chF_{0,2}$ in $\hX_1, \hX_2$, the coefficient of the monomial $\hX_1^{d_1} \hX_2^{d_2}$ is zero whenever $d_1 = 0$ or $d_2 = 0$.
    
\end{itemize}

\subsubsection{Stable cases}
For $2g-2+n > 0$, $(\rho_1 \times \cdots \times \rho_n)^* \omega_{g,n}$ is a holomorphic $n$-form on $A_1 \times \cdots \times A_n$ with no residue at any $\hX_{\bfeta_i} = 0$. Let 
$$
    \chF_{g,n}(q; \hX_1, \dots, \hX_n)
$$
be the holomorphic function on $A_1 \times \cdots \times A_n$ defined by the following two properties:
\begin{itemize}
    \item $\chF_{g,n}$ is a primitive of $(-1)^n (\rho_1 \times \cdots \times \rho_n)^* \omega_{g,n}$, i.e.
    $$
        \frac{\partial^n \chF_{g,n} }{\partial \hX_1 \cdots \partial \hX_n} d\hX_1 \cdots d\hX_n  = (-1)^n (\rho_1 \times \cdots \times \rho_n)^* \omega_{g,n};
    $$

    \item in the Laurent series expansion of $\chF_{g,n}$ in $\hX_1, \dots, \hX_n$, the coefficient of the monomial $\hX_1^{d_1} \cdots \hX_n^{d_n}$ is zero whenever $d_i = 0$ for some $i$.
    
\end{itemize}

\subsection{B-model open graph sum}
For the rest of Section \ref{sect:AllGenus}, we take the specialization $u=1$, i.e. $u_1 = \sfa$, $u_2 = \sfb$, similar to Section \ref{sect:AmodelOpenGraph}.

Let $\bsi \in I_\Si$. For $k \in \bZ_{\ge 0}$ and $i = 1, \dots, n$, $\rho_i^* \theta_{\bsi}^k$ is a holomorphic 1-form on $A_i$ with no residue at $\hX_i = 0$. Let $\xi_{i, \bsi}^k$ be the holomorphic function on $A_i$ such that
$$
    d\xi_{i, \bsi}^k = \rho_i^* \theta_{\bsi}^k.
$$
We set $\hxi_{i, \bsi}^0 = \xi_{i, \bsi}^0$ and $\hxi_{i, \bsi}^k = \rho_i^* \hxi_{i, \bsi}^k$ for $k \in \bZ_{\ge 1}$, so that
$$
    d\hxi_{i, \bsi}^k = \rho_i^* \htheta_{\bsi}^k.
$$
Let
$$
    \xi_{i}^{\bsi}(z, \hX_i) \coloneqq \sum_{k \in \bZ_{\ge 0}} z^k \xi_{i, \bsi}^k(\hX_i), \qquad \hxi_{i}^{\bsi}(z, \hX_i) \coloneqq \sum_{k \in \bZ_{\ge 0}} z^k \hxi_{i, \bsi}^k(\hX_i) 
$$
so that $d\xi_{i}^{\bsi} = \rho_i^* \theta_{\bsi}$ and $d\hxi_{i}^{\bsi} = \rho_i^* \htheta_{\bsi}$. 

The \emph{B-model open leaf} is defined by 
$$
    (\check{\cL}^O)^{\bsi}_k(l_i) \coloneqq  - \frac{1}{\sqrt{-2}} \xi_{i, \bsi}^k = - [z^k] \sum_{\brho \in I_\Sigma} \frac{\hxi_{i}^{\brho}(z, \hX_i)}{\sqrt{-2}} \chR_{\brho}^{\spa \bsi}(z).
$$
Then, the descendant graph sum formula in Theorem \ref{thm:DOSS} imply the following open graph sum formula.

\begin{theorem}\label{thm:BmodelOpenGraph}
For $2g - 2 + n > 0$, we have
$$
    \chF_{g,n}(q; \hX_1, \dots, \hX_n) = \sum_{\vGa \in \bGa_{g,n}(\cX)}\frac{w_B^O(\vGa)}{|\Aut(\vGa)|}.
$$
where $w_B^O(\vGa)$ is defined by substituting the ordinary leaves $\check{\cL}^\bu$ in $w_B^{\bu}(\vGa)$ with the open leaves $\check{\cL}^O$.
\end{theorem}


\subsection{Disk mirror symmetry}
We now discuss the unstable cases of open mirror symmetry.
For $(g,n)=(0,1)$, in the terminology above, Theorem \ref{thm:DiskExpansion} is rephrased as follows.

\begin{proposition}\label{prop:DiskMirror}
For any $\bfeta_1 \in J_\Sigma$, we have
$$
    \chF_{0,1}(q; \hX_1) = F_{0,1}(\btau; \tX_1) 
$$
under the open-closed mirror map $\btau = \btau(q)$, $\tX_1 = \tX_1(q, \hX_1)$.
\end{proposition}


Proposition \ref{prop:DiskMirror} allows us to make the following identification of the A- and B-model open leaves.

\begin{lemma}\label{lem:OpenLeaf}
For any $i = 1, \dots, n$ and $\brho \in I_\Sigma$, we have
\begin{equation}\label{eqn:KeyLeafEqn}
    \frac{\hxi_{i}^{\brho}(z, \hX_i)}{\sqrt{-2}} = \sum_{\bsi' \in I_\Sigma} \left(\txi_{i}^{\bsi'}(z, \tX_i)
    S^{\widehat{\underline{\brho}} }_{\spa \bsi'}(z)\right)_+ \bigg|_{\substack{\su_1 = \sfa \\ \su_2 = \sfb}}
\end{equation}
under the open-closed mirror map $\btau = \btau(q)$, $\tX_i = \tX_i(q, \hX_i)$. In particular, for any $\bsi \in I_\Sigma$ and $k \in \bZ_{\ge 0}$, we have the identification of open leaves
$$
    (\check{\cL}^O)^{\bsi}_k(l_i) = - (\cL^O)^{\bsi}_k(l_i).
$$
\end{lemma}

Equation \eqref{eqn:KeyLeafEqn} is the analog of \cite[Equation (7.5)]{flz2020remodeling}. The identification of open leaves follows from \eqref{eqn:KeyLeafEqn} by multiplying the (A- or B-model) $R$-matrix.

\begin{proof}
By Theorem \ref{thm:DiskExpansion} applied to $\bfeta = \bfeta_i$, $\bff = \bff_i$ and using \eqref{eqn:AmodelDiskCanonical} to rewrite $F_{0,1}$, we have
$$
    \hX_i \frac{\partial}{\partial \hX_i} \rho_i^*(- \log Y_{\bff}) = \left(\tX_i \frac{\partial}{\partial \tX_i} \right)^2 [z^{-2}] \sum_{\bsi' \in I_\Sigma}
    \txi_i^{\bsi'}(z, \tX_i) \left( \text{$\one$}, \cS(\phi_{\bsi'}) \right)_{\cX,\bT'}  \bigg|_{\bT_{\sff}}.
$$
For the right-hand side, for $\bsi' \in I_\Sigma$, $\left( \one, \cS(\phi_{\bsi'}) \right)_{\cX,\bT'}$ takes form $\sum_{a \in \bZ_{\ge 0}} S_a z^{-a-2}$, and thus
$$
    [z^{-2}] \sum_{\bsi' \in I_\Sigma}
    \txi_i^{\bsi'}(z, \tX_i) \left( \one, \cS(\phi_{\bsi'}) \right)_{\cX,\bT'} = \sum_{a \in \bZ_{\ge 0}} ([z^a] \txi_i^{\bsi'}(z, \tX_i)) S_a.
$$
This together with \eqref{eqn:AmodelXiDerivative} implies that for any $k \in \bZ_{\ge -2}$, we have
\begin{align*}
    \left( \tX_i \frac{\partial}{\partial \tX_i} \right)^{k+2} [z^{-2}] \sum_{\bsi' \in I_\Sigma}
    \txi_i^{\bsi'}(z, \tX_i) \left( \text{$\one$}, \cS(\phi_{\bsi'}) \right)_{\cX,\bT'}  \bigg|_{\substack{\su_1 = \sfa \\ \su_2 = \sfb}}
    & = \left( \su_1^{\bff} \right)^{k+2} [z^k] \sum_{\bsi' \in I_\Sigma}
    \txi_i^{\bsi'}(z, \tX_i) \left( \text{$\one$}, \cS(\phi_{\bsi'}) \right)_{\cX,\bT'}  \bigg|_{\substack{\su_1 = \sfa \\ \su_2 = \sfb}} \\
    & = \sfa_{\bff}^{k+2} [z^k] \sum_{\bsi' \in I_\Sigma}
    \txi_i^{\bsi'}(z, \tX_i) \left( \text{$\one$}, \cS(\phi_{\bsi'}) \right)_{\cX,\bT'}  \bigg|_{\substack{\su_1 = \sfa \\ \su_2 = \sfb}}. 
\end{align*}
For the left-hand side, since $\hX_{\bff} \frac{\partial}{\partial \hX_{\bff}} = -\sfa_{\bff} \frac{d}{d\hx_{\bff}} =  -\sfa_{\bff} \frac{d}{d\hx}$, we have for any $k \in \bZ_{\ge -2}$ that
$$
    \left( \hX_i \frac{\partial}{\partial \hX_i} \right)^{k+1} \rho_i^*(-\log Y_{\bff}) =  \sfa_{\bff}^{k+2} \rho_i^* \left((-1)^{k+1} \left(\frac{d}{d\hx} \right)^{k+1} \hy_{\bff} \right).
$$
Here in the case $k = -2$, the notation $\left( \hX_i \frac{\partial}{\partial \hX_i} \right)^{-1}$ and $\left(\frac{d}{d\hx_{\bff}} \right)^{-1}$ stands for integration which is defined after removing the residues of the integrand, similar to the definition of $\chF_{0,1}$ in Section \ref{sect:BmodelDisk}. Therefore we have
$$
    \rho_i^* \left((-1)^{k+1} \left(\frac{d}{d\hx} \right)^{k+1} \hy_{\bff} \right) = [z^k] \sum_{\bsi' \in I_\Sigma}
    \txi_i^{\bsi'}(z, \tX_i) \left( \text{$\one$}, \cS(\phi_{\bsi'}) \right)_{\cX,\bT'}  \bigg|_{\substack{\su_1 = \sfa \\ \su_2 = \sfb}}
$$
and thus
\begin{equation}\label{eqn:OpenLeafDerivative}
    \sum_{k \in \bZ_{\ge -2}} z^k \rho_i^* \left((-1)^{k+1} \left(\frac{d}{d\hx} \right)^{k+1} \hy_{\bff} \right) = \sum_{\bsi' \in I_\Sigma}
    \txi_i^{\bsi'}(z, \tX_i) \left( \text{$\one$}, \cS(\phi_{\bsi'}) \right)_{\cX,\bT'}  \bigg|_{\substack{\su_1 = \sfa \\ \su_2 = \sfb}}.
\end{equation}

Let $\brho \in I_\Sigma$. As in the derivation of \cite[Equation (7.4)]{flz2020remodeling} (applied to the flag $\bff$ instead of the reference flag $\bff_0$), and using the sign
$$
    h_1^{\brho}(q)= \sqrt{ \frac{-2}{\Delta^\brho(\btau(q))} }.
$$
as in \cite[Section 7.1]{flz2020remodeling}, if 
$$
    \hat{\phi}_{\brho}(\btau(q)) = \sum_{i = 1}^{\text{$\fg$}} A_{\brho}^i(q) H_{a_i} \star_{\btau(q)} H_{b_i} - \sum_{a = 1}^{\text{$\fp$}} B_{\brho}^a(q) H_a + C_{\brho}(q) \one,
$$
then
$$
    - \frac{\theta_{\brho}^0}{\sqrt{-2}} = \sum_{i = 1}^{\text{$\fg$}} A_{\brho}^i(q) \frac{\partial^2 (\hy_{\bff} d\hx_{\bff})}{\partial\tau_{a_i}\partial\tau_{b_i}} + \sum_{a = 1}^{\fp} B_{\brho}^a(q) d\left(\frac{\frac{\partial (\hy_{\bff} d\hx_{\bff}) }{\partial\tau_a}}{d\hx_{\bff}}\right) + C_{\brho}(q) d\left(\frac{d\hy_{\bff}}{d\hx_{\bff}}\right).
$$
Thus for any $k \in \bZ_{\ge 0}$, we have
$$
    \sum_{i = 1}^{\text{$\fg$}} A_{\brho}^i(q) \frac{\partial^2}{\partial\tau_{a_i}\partial\tau_{b_i}} \left(\frac{d}{d\hx} \right)^{k-1} \hy_{\bff} + \sum_{a = 1}^{\fp} B_{\brho}^a(q) \frac{\partial}{\partial\tau_a} \left(\frac{d}{d\hx} \right)^{k} \hy_{\bff} + C_{\brho}(q) \left(\frac{d}{d\hx} \right)^{k+1} \hy_{\bff} 
    = - \frac{1}{\sqrt{-2}} \left(\frac{d}{d\hx} \right)^{k-1} \frac{\theta_{\brho}^0}{d\hx}
    = \frac{(-1)^k \hxi_{\brho}^k}{\sqrt{-2}}.
$$
Equation \eqref{eqn:KeyLeafEqn} then follows from applying the differential operator
$$
    z^2 \sum_{i = 1}^{\text{$\fg$}}  A_{\brho}^i(q) \frac{\partial^2 }{\partial\tau_{a_i}\partial\tau_{b_i}} - z \sum_{a = 1}^{\text{$\fp$}} B_{\brho}^a(q) \frac{\partial}{\partial\tau_a} + C_{\brho}(q)
$$
to the two sides of \eqref{eqn:OpenLeafDerivative} and applying $(-)_+$.
\end{proof}

\subsection{Annulus mirror symmetry}

We now turn to the case $(g,n)=(0,2)$.

\begin{proposition}\label{prop:AnnMirror}
For any $\bfeta_1, \bfeta_2 \in J_\Sigma$, we have
$$
    \chF_{0,2}(q; \hX_1, \hX_2) = - F_{0,2}(\btau; \tX_1, \tX_2) 
$$
under the open-closed mirror map $\btau = \btau(q)$, $\tX_i = \tX_i(q, \hX_i)$.
\end{proposition}
    
We start with the following lemma.

\begin{lemma}\label{lem:AnnMirrorDer}
For any $\bfeta_1, \bfeta_2 \in J_\Sigma$, we have
$$
    \left( \frac{1}{\sfa_{\bff_1} }\hX_1 \frac{\partial}{\partial \hX_1} + \frac{1}{\sfa_{\bff_2} }\hX_2 \frac{\partial}{\partial \hX_2} \right)\chF_{0,2}(q; \hX_1, \hX_2) 
    = - \left( \frac{1}{\sfa_{\bff_1} }\tX_1 \frac{\partial}{\partial \tX_1} + \frac{1}{\sfa_{\bff_2} }\tX_2 \frac{\partial}{\partial \tX_2} \right) F_{0,2}(\btau; \tX_1, \tX_2).
$$
\end{lemma}

\begin{proof}
We make the following computation, following the proof of \cite[Proposition 7.4]{flz2020remodeling}:
\begin{align*}
    & \left( \frac{1}{\sfa_{\bff_1} }\hX_1 \frac{\partial}{\partial \hX_1} + \frac{1}{\sfa_{\bff_2} }\hX_2 \frac{\partial}{\partial \hX_2} \right) \chF_{0,2}(q; \hX_1, \hX_2) \\
    & = \frac{1}{2} \sum_{\brho \in I_\Sigma} \hxi_{1, \brho}^0(\hX_1) \hxi_{2, \brho}^0(\hX_2)  \\
    & = - [z_1^0z_2^0] \sum_{\brho, \bsi_1, \bsi_2 \in I_\Sigma} 
    S^{\widehat{\underline{\brho}} }_{\spa \bsi_1}(z_1) S^{\widehat{\underline{\brho}} }_{\spa \bsi_2}(z_2) 
    \txi_{1}^{\bsi_1}(z_1, \tX_1) \txi_{2}^{\bsi_2}(z_2, \tX_2)
    \bigg|_{\substack{\su_1 = \sfa \\ \su_2 = \sfb}} \\
    & = - [z_1^0z_2^0] (z_1 + z_2) \sum_{\bsi_1, \bsi_2 \in I_\Sigma} 
    \double{ \frac{\phi_{\bsi_1}}{z_1 - \hat{\psi}}, \frac{\phi_{\bsi_2}}{z_2 - \hat{\psi}} }_{0,2}^{\cX,\bT'}
    \txi_{1}^{\bsi_1}(z_1, \tX_1) \txi_{2}^{\bsi_2}(z_2, \tX_2)
    \bigg|_{\substack{\su_1 = \sfa \\ \su_2 = \sfb}}  \\
    & = - [z_1^{-1}z_2^{-1}] \left( \frac{1}{\sfa_{\bff_1} }\tX_1 \frac{\partial}{\partial \tX_1} + \frac{1}{\sfa_{\bff_2} }\tX_2 \frac{\partial}{\partial \tX_2} \right) \sum_{\bsi_1, \bsi_2 \in I_\Sigma} 
    \double{ \frac{\phi_{\bsi_1}}{z_1 - \hat{\psi}}, \frac{\phi_{\bsi_2}}{z_2 - \hat{\psi}} }_{0,2}^{\cX,\bT'}
    \txi_{1}^{\bsi_1}(z_1, \tX_1) \txi_{2}^{\bsi_2}(z_2, \tX_2)
    \bigg|_{\substack{\su_1 = \sfa \\ \su_2 = \sfb}} \\
    & = - \left( \frac{1}{\sfa_{\bff_1} }\tX_1 \frac{\partial}{\partial \tX_1} + \frac{1}{\sfa_{\bff_2} }\tX_2 \frac{\partial}{\partial \tX_2} \right) F_{0,2}(\btau; \tX_1, \tX_2)
\end{align*}
where the second equality follows from \eqref{eqn:KeyLeafEqn} and the last equality uses the expression \eqref{eqn:AmodelAnnCanonical}.
\end{proof}

\begin{proof}[Proof of Proposition \ref{prop:AnnMirror}]
When $\bfeta_1$, $\bfeta_2$ are both outer and $\sgn(\bff_1) = \sgn(\bff_2)$, the proposition is implied by Lemma \ref{lem:AnnMirrorDer}. To handle the other cases, we first use the \emph{special geometry} property \cite{EO07} of $\omega_{0,2}$ to reduce the proposition to a statement in the limit $q \to 0$, or equivalently, $\tQ,\tau''\to 0$. Recall from \eqref{eqn:btauHBasis} that
$$
    \btau - \btau_0 = \sum_{a = 1}^{\text{$\fp$}} \btau_a \Hbar_a^{\bT'}
$$
where for $a = \fp'+1, \dots,\fp$, we set $\Hbar_a^{\bT'} = \one_{b_{3+a}}$. Then, for $a = 1, \dots, \fp$, there exits a cycle $B_a \in H_1(\Cbar_q, \Dbar_q^\infty; \bC)$, where $\Dbar_q^\infty$ denotes the set of punctures on $\Cbar_q$, such that
$$
    \frac{\partial \omega_{0,2}}{\partial \tau_a}(p_1, p_2) = \int_{p_{3} \in B_a} \omega_{0,3}(p_1, p_2, p_3).
$$
Then, $\frac{\partial \chF_{0,2}}{\partial \tau_a}(q; \hX_1, \hX_2)$ is obtained by the local primitive of $(\rho_1 \times \rho_2)^* \int_{p_{3} \in B_a} \omega_{0,3}(p_1, p_2, p_3)$ as in Section \ref{sect:BmodelAnn}, and the descendant graph sum formula for $\omega_{0,3}$ in Theorem \ref{thm:DOSS} gives a graph sum formula for $\frac{\partial \chF_{0,2}}{\partial \tau_a}$ where for the third factor we use the \emph{primary leaf} which is the integration
$$
    \int_{p_{3} \in B_a}  \frac{1}{\sqrt{-2}} \theta_{\bsi}^k(p_3)
$$
on the ordinary leaf. On the other hand, we have by \eqref{eqn:AmodelAnnCanonical} that
$$
    \frac{\partial F_{0,2}}{\partial \tau_a}(\btau; \tX_1, \tX_2) = [z_1^{-1}z_2^{-1}] \sum_{\bsi_1, \bsi_2 \in I_\Sigma}
    \double{\frac{\phi_{\bsi_1}}{z_1 - \hat{\psi}}, \frac{\phi_{\bsi_2}}{z_2 - \hat{\psi}}, H_a}_{0,3}^{\cX,\bT'}
    \txi_{1}^{\bsi_1}(z_1, \tX_1) \txi_{2}^{\bsi_2}(z_2, \tX_2)  \bigg|_{\bT_{\sff}}.
$$
The graph sum formula for the genus-$0$, $3$-pointed descendant potential in Theorem \ref{thm:Zong} gives a graph sum formula for $\frac{\partial F_{0,2}}{\partial \tau_a}$ where the third ordinary leaf is replaced by the \emph{primary leaf}
$$
    [z^0] \sum_{\brho \in I_\Si} H_a^{\brho} R(-z)_{\brho}^{\spa \bsi} 
$$
with $H_a^{\brho}$ defined by $H_a = \sum_{\brho \in I_\Si} H_a^{\brho} \hat{\phi}_{\brho}(\tau)$. The primary leaves are identified up to a sign in the proof of \cite[Theorem 7.10]{flz2020remodeling}. Together with the identification of descendant weights in Theorem \ref{thm:GraphSumId} and the identification of open leaves in Lemma \ref{lem:OpenLeaf}, we have
$$
    \frac{\partial \omega_{0,2}}{\partial \tau_a}(p_1, p_2) = - \frac{\partial F_{0,2}}{\partial \tau_a}(\btau; \tX_1, \tX_2) 
$$
for all $a$. This reduces the proposition to
\begin{equation}\label{eqn:AnnMirrorLT}
    \chF_{0,2}(0; \hX_1, \hX_2) = - F_{0,2}(0; \tX_1, \tX_2) 
\end{equation}
where the ``$0$'' on the right-hand side denotes the limit $\tQ,\tau''\to 0$.

Observe that the two sides of \eqref{eqn:AnnMirrorLT} are nonzero only if the 2-cones $\tau_1$, $\tau_2$ associated to the open phases $\bfeta_1$, $\bfeta_2$ are contained in a common 3-cone $\sigma$. Such $\sigma$ is unique unless $\bfeta_1 = \bfeta_2$ is inner, in which case we combine the argument below for the two contributing choices of $\sigma$. Given $\sigma$, $F_{0,2}(0; \tX_1, \tX_2)$ (resp. $\chF_{0,2}(0; \hX_1, \hX_2)$) is the A-model (resp. B-model) annulus potential of the affine toric Calabi-Yau 3-orbifold $\cX_\sigma$ determined by $\sigma$ in the limit $q \to 0$. Here, the open phases $\bfeta_1$, $\bfeta_2$ are viewed as \emph{outer} phases in $\cX_\sigma$. We may also suppose that the preferred flag $\bff_i$, $i = 1,2$, is the flag $(\tau_i, \sigma)$ in $\cX_\sigma$; in the event that $\bfeta_i$ is inner in $\cX$ and $\bff_i$ is not $(\tau_i, \sigma)$, we may instead use the other flag associated to $\tau_i$ for $\bff_i$ and the corresponding A-/B-model open moduli parameters for $\tX_i$, $\hX_i$ respectively, with the definition of $F_{0,2}$, $\chF_{0,2}$ adjusted accordingly.

In view of the above interpretation in the affine situation, we have \emph{power series} expansions
$$
    F_{0,2}(0; \tX_1, \tX_2) = \sum_{\mu_1, \mu_2 \in \bZ_{>0}} N_{\mu_1, \mu_2} \tX_1^{\mu_1} \tX_2^{\mu_2}, \qquad
    \chF_{0,2}(0; \hX_1, \hX_2) = \sum_{\mu_1, \mu_2 \in \bZ_{>0}} \chN_{\mu_1, \mu_2} \hX_1^{\mu_1} \hX_2^{\mu_2}.
$$
With the convention for the flags $\bff_1$, $\bff_2$ adjusted if necessary, Lemma \ref{lem:AnnMirrorDer} implies that 
\begin{equation}\label{eqn:AnnMirrorLTDer}
     \left( \frac{1}{\sfa_{\bff_1} }\hX_1 \frac{\partial}{\partial \hX_1} + \frac{1}{\sfa_{\bff_2} }\hX_2 \frac{\partial}{\partial \hX_2} \right) \chF_{0,2}(0; \hX_1, \hX_2) 
     = - \left( \frac{1}{\sfa_{\bff_1} }\tX_1 \frac{\partial}{\partial \tX_1} + \frac{1}{\sfa_{\bff_2} }\tX_2 \frac{\partial}{\partial \tX_2} \right) F_{0,2}(0; \tX_1, \tX_2).
\end{equation}
This again implies \eqref{eqn:AnnMirrorLT} if $\sgn(\bff_1) = \sgn(\bff_2)$. It remains to consider the case $\sgn(\bff_1) = - \sgn(\bff_2)$. Here, the differential operators in \eqref{eqn:AnnMirrorLTDer} annihilate terms with 
\begin{equation}\label{eqn:ExceptionalF}
    \frac{\mu_1}{\sfa_{\bff_1}} + \frac{\mu_2}{\sfa_{\bff_2}} = 0,
\end{equation}
while any other $N_{\mu_1, \mu_2}$ is identified with $-\chN_{\mu_1, \mu_2}$.

We now fix $\mu_1, \mu_2 \in \bZ_{>0}$ and consider the dependence of the coefficients $N_{\mu_1, \mu_2}$, $\chN_{\mu_1, \mu_2}$ on the framing $\sff$, which induces the framings $\sff_1$, $\sff_2$ on the two flags. By the virtual localization computation in Section \ref{sect:OpenGWInv}, the annulus invariant $N_{\mu_1, \mu_2}(\sff)$ is a polynomial in $\sff_1$ and $\sff_2$, and thus a rational function in $\sff$. On the other hand, by Lemma \ref{lem:AnnRational} below, $\chN_{\mu_1, \mu_2}(\sff)$ is also a polynomial in $\sff_1$ and $\sff_2$, and thus a rational function in $\sff$. The condition \eqref{eqn:ExceptionalF} is met only when the integers $\sfa_{\bff_1}$, $\sfa_{\bff_2}$ determined by $\sff$ satisfy 
$$
    \sfa_{\bff_2} \sw_{\bff_1, 1} - \sfa_{\bff_1} \sw_{\bff_2, 1} \big|_{\bT_{\sff}} = 0,
$$
which is a non-generic condition. Therefore, $N_{\mu_1, \mu_2}(\sff)$ and $-\chN_{\mu_1, \mu_2}(\sff)$ are equal for all but finitely many $\sff \in \bQ$, and are thus equal for all $\sff$. This implies the desired equality \eqref{eqn:AnnMirrorLT}.
\end{proof}

\begin{lemma}\label{lem:AnnRational}
For any $\bfeta_1, \bfeta_2 \in J_\Sigma$, in the series expansion 
$$
    \chF_{0,2}(0; \hX_1, \hX_2) = \sum_{\mu_1, \mu_2} \chN_{\mu_1, \mu_2} \hX_1^{\mu_1} \hX_2^{\mu_2},
$$
each $\chN_{\mu_1, \mu_2}$ is a polynomial in $\sff_1$ and $\sff_2$.
\end{lemma}

\begin{proof}
Note that $\omega_{0,2}$ is independent of the framing $\sff$. On $U_{\bfeta_1} \times U_{\bfeta_2}$, consider the \emph{unframed} annulus potential
$$
    \chF^{\text{un}}_{0,2}(0; X_{\bff_1}, X_{\bff_2}) = \sum_{\mu_1, \mu_2} \chN^{\text{un}}_{\mu_1, \mu_2} X_{\bff_1}^{\mu_1} X_{\bff_2}^{\mu_2}
$$
defined by the local primitive of $\omega_{0,2} - \delta_{\bfeta_1, \bfeta_2} \frac{dX_{\bff_1} dX_{\bff_2}}{(X_{\bff_1} - X_{\bff_2})^2}$ in the unframed coordinates $X_{\bff_1}$, $X_{\bff_2}$, where each $\chN^{\text{un}}_{\mu_1, \mu_2}$ is a constant independent of $\sff$. Then, we have
$$
    \chF_{0,2}(0; \hX_1, \hX_2) = \sum_{\mu_1, \mu_2} \chN^{\text{un}}_{\mu_1, \mu_2} 
    \left(\xi_{\sfa_{\bff_1}}^{-1}(Y_{\bff_1}^{-\sff_1}) \right)^{\mu_1 - 1} 
    \left(\xi_{\sfa_{\bff_2}}^{-1}(Y_{\bff_2}^{-\sff_2}) \right)^{\mu_2 - 1}
    \hX_1^{\mu_1} \hX_2^{\mu_2}.
$$
For $i = 1, 2$, as computed in Section \ref{sect:LocalExpansion}, $Y_{\bff_i}^{-\sff_i}$ can be expanded as a series in $\hX_i$ whose coefficients are polynomials in $\sff_i$. The lemma then follows.
\end{proof}

\subsection{All-genus open mirror symmetry}
Finally, we prove the main theorem on all-genus open mirror symmetry.

\begin{theorem}\label{thm:AllGenusMirror}
For any $g \in \bZ_{\ge 0}$, $n \in \bZ_{\ge 1}$, and $\bfeta_1, \dots, \bfeta_n \in J_\Sigma$, we have
$$
    \chF_{g,n}(q; \hX_1, \dots, \hX_n) = (-1)^{g-1+n} F_{g,n}(\btau; \tX_1, \dots, \tX_n) 
$$
under the open-closed mirror map $\btau = \btau(q)$, $\tX_i = \tX_i(q, \hX_i)$.
\end{theorem}

\begin{proof}
The unstable cases $(g, n) = (0,1)$ and $(0,2)$ are Propositions \ref{prop:DiskMirror} and \ref{prop:AnnMirror} respectively. In the stable case $2g - 2 + n > 0$, the identification of descendant weights in Theorem \ref{thm:GraphSumId} and the identification of open leaves in Lemma \ref{lem:OpenLeaf} implies that for any $\vGa\in \bGa_{g,n}(\cX)$, we have the identification of open weights
$$
    w^O_B(\vGa) = (-1)^{g(\vGa)-1+n} w^O_A(\vGa).
$$
This implies the theorem by the graph sum formulas for $F_{g,n}$ and $\chF_{g,n}$ in Theorems \ref{thm:AmodelOpenGraph} and \ref{thm:BmodelOpenGraph} respectively.
\end{proof}


\printbibliography

@article {BCY12,
    AUTHOR = {Bryan, Jim and Cadman, Charles and Young, Ben},
     TITLE = {The orbifold topological vertex},
   JOURNAL = {Adv. Math.},
  FJOURNAL = {Advances in Mathematics},
    VOLUME = {229},
      YEAR = {2012},
    NUMBER = {1},
     PAGES = {531--595},
      ISSN = {0001-8708,1090-2082},
   MRCLASS = {14N35 (14D23 14J32)},
  MRNUMBER = {2854183},
       DOI = {10.1016/j.aim.2011.09.008},
       URL = {https://doi.org/10.1016/j.aim.2011.09.008},
}

@article {MOOP11,
    AUTHOR = {Maulik, D. and Oblomkov, A. and Okounkov, A. and
              Pandharipande, R.},
     TITLE = {Gromov-{W}itten/{D}onaldson-{T}homas correspondence for toric
              3-folds},
   JOURNAL = {Invent. Math.},
  FJOURNAL = {Inventiones Mathematicae},
    VOLUME = {186},
      YEAR = {2011},
    NUMBER = {2},
     PAGES = {435--479},
      ISSN = {0020-9910,1432-1297},
   MRCLASS = {14N35 (14M25)},
  MRNUMBER = {2845622},
MRREVIEWER = {Hsian-Hua\ Tseng},
       DOI = {10.1007/s00222-011-0322-y},
       URL = {https://doi.org/10.1007/s00222-011-0322-y},
}

@incollection{Viro01,
	author = {Viro, Oleg},
	booktitle = {European {C}ongress of {M}athematics, {V}ol. {I} ({B}arcelona, 2000)},
	mrclass = {14P25},
	mrnumber = {1905317},
	mrreviewer = {Sergey M. Finashin},
	pages = {135--146},
	publisher = {Birkh\"{a}user, Basel},
	series = {Progr. Math.},
	title = {Dequantization of real algebraic geometry on logarithmic paper},
	volume = {201},
	year = {2001}}

@article{IV96,
	author = {Itenberg, Ilia and Viro, Oleg},
	fjournal = {The Mathematical Intelligencer},
	issn = {0343-6993},
	journal = {Math. Intelligencer},
	mrclass = {14P25},
	mrnumber = {1413249},
	mrreviewer = {Eugenii Shustin},
	number = {4},
	pages = {19--28},
	title = {Patchworking algebraic curves disproves the {R}agsdale conjecture},
	volume = {18},
	year = {1996}}

@article{GT13,
	author = {Gholampour, Amin and Tseng, Hsian-Hua},
	fjournal = {Michigan Mathematical Journal},
	issn = {0026-2285},
	journal = {Michigan Math. J.},
	mrclass = {14N35 (53D45)},
	mrnumber = {3160540},
	mrreviewer = {Qile Chen},
	number = {4},
	pages = {753--768},
	title = {On computations of genus 0 two-point descendant {G}romov-{W}itten invariants},
	volume = {62},
	year = {2013}}

@article{flz2020remodeling,
	author = {Fang, Bohan and Liu, Chiu-Chu Melissa and Zong, Zhengyu},
	fjournal = {Journal of the American Mathematical Society},
	issn = {0894-0347},
	journal = {J. Amer. Math. Soc.},
	mrclass = {14N35 (14J33)},
	mrnumber = {4066474},
	number = {1},
	pages = {135--222},
	title = {On the remodeling conjecture for toric {C}alabi-{Y}au 3-orbifolds},
	volume = {33},
	year = {2020}}

@article{flz2020affine,
	author = {Fang, Bohan and Liu, Chiu-Chu Melissa and Zong, Zhengyu},
	fjournal = {Algebraic Geometry},
	issn = {2313-1691},
	journal = {Algebr. Geom.},
	mrclass = {14N35 (14J33)},
	mrnumber = {4061327},
	number = {2},
	pages = {192--239},
	title = {All-genus open-closed mirror symmetry for affine toric {C}alabi-{Y}au 3-orbifolds},
	volume = {7},
	year = {2020}}

@article{Iritani09,
	author = {Iritani, Hiroshi},
	fjournal = {Advances in Mathematics},
	issn = {0001-8708},
	journal = {Adv. Math.},
	mrclass = {53D37 (14J33 14N35 53D45)},
	mrnumber = {2553377},
	mrreviewer = {Hsian-Hua Tseng},
	number = {3},
	pages = {1016--1079},
	title = {An integral structure in quantum cohomology and mirror symmetry for toric orbifolds},
	volume = {222},
	year = {2009}}

@article{Eynard11,
	archiveprefix = {arXiv},
	author = {B. Eynard},
	date-modified = {2025-04-17 17:27:18 +0800},
	eprint = {1104.0176},
	% primaryclass = {math-ph},
	title = {Intersection numbers of spectral curves},
	year = {2011},
	bdsk-url-1 = {https://arxiv.org/abs/1104.0176}}

@phdthesis{Zong15,
	author = {Zong, Zhengyu},
	school = {Columbia University},	
	title = {Equivariant {G}romov-{W}itten {T}heory of {GKM} {O}rbifolds},
	year = {2015}
}

@article{Zhu15,
	author = {Zhu, Shengmao},
	fjournal = {Mathematical Research Letters},
	issn = {1073-2780},
	journal = {Math. Res. Lett.},
	mrclass = {81T20 (14H81)},
	mrnumber = {3342249},
	number = {2},
	pages = {633--643},
	title = {On a proof of the {B}ouchard-{S}ulkowski conjecture},
	volume = {22},
	year = {2015}}

@article{Teleman12,
	author = {Teleman, Constantin},
	fjournal = {Inventiones Mathematicae},
	issn = {0020-9910},
	journal = {Invent. Math.},
	mrclass = {57R56 (18D10 53D45)},
	mrnumber = {2917177},
	mrreviewer = {Julia Bergner},
	number = {3},
	pages = {525--588},
	title = {The structure of 2{D} semi-simple field theories},
	volume = {188},
	year = {2012}}

@incollection{Ruan06,
	author = {Ruan, Yongbin},
	booktitle = {Gromov-{W}itten theory of spin curves and orbifolds},
	mrclass = {14N35 (14C05)},
	mrnumber = {2234886},
	mrreviewer = {Gianni Ciolli},
	pages = {117--126},
	publisher = {Amer. Math. Soc., Providence, RI},
	series = {Contemp. Math.},
	title = {The cohomology ring of crepant resolutions of orbifolds},
	volume = {403},
	year = {2006}}

@incollection{Ruan02,
	author = {Ruan, Yongbin},
	booktitle = {Symposium in {H}onor of {C}. {H}. {C}lemens ({S}alt {L}ake {C}ity, {UT}, 2000)},
	mrclass = {32S45 (14N35 19M05 32J81 81T30)},
	mrnumber = {1941583},
	mrreviewer = {Anatoly Libgober},
	pages = {187--233},
	publisher = {Amer. Math. Soc., Providence, RI},
	series = {Contemp. Math.},
	title = {Stringy geometry and topology of orbifolds},
	volume = {312},
	year = {2002}}

@article{RZ15,
	author = {Ross, Dustin and Zong, Zhengyu},
	fjournal = {Advances in Mathematics},
	issn = {0001-8708},
	journal = {Adv. Math.},
	mrclass = {14N35},
	mrnumber = {3406532},
	mrreviewer = {Babak Haghighat},
	pages = {1448--1486},
	title = {Cyclic {H}odge integrals and loop {S}chur functions},
	volume = {285},
	year = {2015}}

@article{RZ13,
	author = {Ross, Dustin and Zong, Zhengyu},
	fjournal = {Geometry \& Topology},
	issn = {1465-3060},
	journal = {Geom. Topol.},
	mrclass = {14N35 (05E05 53D45)},
	mrnumber = {3190303},
	mrreviewer = {Jian Xun Hu},
	number = {5},
	pages = {2935--2976},
	title = {The gerby {G}opakumar-{M}ari\~{n}o-{V}afa formula},
	volume = {17},
	year = {2013}}

@article{Mikhalkin00,
	author = {Mikhalkin, G.},
	fjournal = {Annals of Mathematics. Second Series},
	issn = {0003-486X},
	journal = {Ann. of Math. (2)},
	mrclass = {14P05},
	mrnumber = {1745011},
	mrreviewer = {A. Tognoli},
	number = {1},
	pages = {309--326},
	title = {Real algebraic curves, the moment map and amoebas},
	volume = {151},
	year = {2000}}

@article{MNOP06b,
	author = {Maulik, D. and Nekrasov, N. and Okounkov, A. and Pandharipande, R.},
	fjournal = {Compositio Mathematica},
	issn = {0010-437X},
	journal = {Compos. Math.},
	mrclass = {14N35 (14C05)},
	mrnumber = {2264665},
	mrreviewer = {Hsian-Hua Tseng},
	number = {5},
	pages = {1286--1304},
	title = {Gromov-{W}itten theory and {D}onaldson-{T}homas theory. {II}},
	volume = {142},
	year = {2006}}

@article{MNOP06,
	author = {Maulik, D. and Nekrasov, N. and Okounkov, A. and Pandharipande, R.},
	fjournal = {Compositio Mathematica},
	issn = {0010-437X},
	journal = {Compos. Math.},
	mrclass = {14N35 (14J32)},
	mrnumber = {2264664},
	mrreviewer = {Hsian-Hua Tseng},
	number = {5},
	pages = {1263--1285},
	title = {Gromov-{W}itten theory and {D}onaldson-{T}homas theory. {I}},
	volume = {142},
	year = {2006}}

@article{Marino08,
	author = {Mari\~{n}o, Marcos},
	fjournal = {Journal of High Energy Physics. A SISSA Journal},
	issn = {1126-6708},
	journal = {J. High Energy Phys.},
	mrclass = {81T45 (14N35 81T30)},
	mrnumber = {2391060},
	mrreviewer = {Vincent Bouchard},
	number = {3},
	pages = {060, 34},
	title = {Open string amplitudes and large order behavior in topological string theory},
	year = {2008}}

@article{Givental01,
	author = {Givental, Alexander B.},
	fjournal = {International Mathematics Research Notices},
	issn = {1073-7928},
	journal = {Internat. Math. Res. Notices},
	mrclass = {53D45 (14N35)},
	mrnumber = {1866444},
	mrreviewer = {Gilberto Bini},
	number = {23},
	pages = {1265--1286},
	title = {Semisimple {F}robenius structures at higher genus},
	year = {2001}}

@incollection{Givental97,
	author = {Givental, Alexander},
	booktitle = {Integrable systems and algebraic geometry ({K}obe/{K}yoto, 1997)},
	mrclass = {14N35 (14M25 32S05 33C90 53D45)},
	mrnumber = {1672116},
	mrreviewer = {Andreas Gathmann},
	pages = {107--155},
	publisher = {World Sci. Publ., River Edge, NJ},
	title = {Elliptic {G}romov-{W}itten invariants and the generalized mirror conjecture},
	year = {1998}}

@book{Fay73,
	author = {Fay, John D.},
	mrclass = {30A48 (14H15)},
	mrnumber = {335789},
	mrreviewer = {H. M. Farkas},
	pages = {iv+137},
	publisher = {Springer-Verlag, Berlin-New York},
	series = {Lecture Notes in Mathematics, Vol. 352},
	title = {Theta functions on {R}iemann surfaces},
	year = {1973}}

@article{EO07,
	author = {Eynard, B. and Orantin, N.},
	fjournal = {Communications in Number Theory and Physics},
	issn = {1931-4523},
	journal = {Commun. Number Theory Phys.},
	mrclass = {14H15 (14N35 32A27 37K10 37K20 81T45)},
	mrnumber = {2346575},
	mrreviewer = {Vincent Bouchard},
	number = {2},
	pages = {347--452},
	title = {Invariants of algebraic curves and topological expansion},
	volume = {1},
	year = {2007}}

@article{EO15,
	author = {Eynard, B. and Orantin, N.},
	fjournal = {Communications in Mathematical Physics},
	issn = {0010-3616},
	journal = {Comm. Math. Phys.},
	mrclass = {14N35 (14J33 53D45 81R05 81R10)},
	mrnumber = {3339157},
	mrreviewer = {Sergiy Koshkin},
	number = {2},
	pages = {483--567},
	title = {Computation of open {G}romov-{W}itten invariants for toric {C}alabi-{Y}au 3-folds by topological recursion, a proof of the {BKMP} conjecture},
	volume = {337},
	year = {2015}}

@article{DOSS14,
	author = {Dunin-Barkowski, P. and Orantin, N. and Shadrin, S. and Spitz, L.},
	fjournal = {Communications in Mathematical Physics},
	issn = {0010-3616},
	journal = {Comm. Math. Phys.},
	mrclass = {81T45 (14N35 53D45)},
	mrnumber = {3199996},
	mrreviewer = {Wan Keng Cheong},
	number = {2},
	pages = {669--700},
	title = {Identification of the {G}ivental formula with the spectral curve topological recursion procedure},
	volume = {328},
	year = {2014}}

@book{CK99,
	author = {Cox, David A. and Katz, Sheldon},
	isbn = {0-8218-1059-6},
	mrclass = {14J32 (14-02 14M25 14N10 14N35 32G81 32J81 32Q25)},
	mrnumber = {1677117},
	mrreviewer = {Andreas Gathmann},
	pages = {xxii+469},
	publisher = {American Mathematical Society, Providence, RI},
	series = {Mathematical Surveys and Monographs},
	title = {Mirror symmetry and algebraic geometry},
	volume = {68},
	year = {1999}}

@incollection{Chen18,
	author = {Chen, Lin},
	booktitle = {Topological recursion and its influence in analysis, geometry, and topology},
	mrclass = {14N35 (14H10)},
	mrnumber = {3840131},
	mrreviewer = {Sergiy Koshkin},
	pages = {83--102},
	publisher = {Amer. Math. Soc., Providence, RI},
	series = {Proc. Sympos. Pure Math.},
	title = {Bouchard-{K}lemm-{M}arino-{P}asquetti conjecture for {$\Bbb C^3$}},
	volume = {100},
	year = {2018}}

@article{BKMP10,
	author = {Bouchard, Vincent and Klemm, Albrecht and Mari\~{n}o, Marcos and Pasquetti, Sara},
	fjournal = {Communications in Mathematical Physics},
	issn = {0010-3616},
	journal = {Comm. Math. Phys.},
	mrclass = {81T45 (11F23 14N35 81T30)},
	mrnumber = {2628817},
	mrreviewer = {Brad Safnuk},
	number = {3},
	pages = {589--623},
	title = {Topological open strings on orbifolds},
	volume = {296},
	year = {2010}}

@article{BKMP09,
	author = {Bouchard, Vincent and Klemm, Albrecht and Mari\~{n}o, Marcos and Pasquetti, Sara},
	fjournal = {Communications in Mathematical Physics},
	issn = {0010-3616},
	journal = {Comm. Math. Phys.},
	mrclass = {81T45 (32G81 32Q25 81T30)},
	mrnumber = {2480744},
	mrreviewer = {Johannes Walcher},
	number = {1},
	pages = {117--178},
	title = {Remodeling the {B}-model},
	volume = {287},
	year = {2009}}

@article{BCMS13,
	author = {Bouchard, Vincent and Catuneanu, Andrei and Marchal, Olivier and Su\l kowski, Piotr},
	fjournal = {Letters in Mathematical Physics},
	issn = {0377-9017},
	journal = {Lett. Math. Phys.},
	mrclass = {14N35 (14J33 14J81)},
	mrnumber = {3004817},
	mrreviewer = {Siu-Cheong Lau},
	number = {1},
	pages = {59--77},
	title = {The remodeling conjecture and the {F}aber-{P}andharipande formula},
	volume = {103},
	year = {2013}}

@article{AKV02,
	author = {Aganagic, Mina and Klemm, Albrecht and Vafa, Cumrun},
	fjournal = {Zeitschrift f\"{u}r Naturforschung. A. Journal of Physical Sciences},
	issn = {0932-0784},
	journal = {Z. Naturforsch. A},
	mrclass = {81T30 (14J81 32Q25 81T60)},
	mrnumber = {1906661},
	mrreviewer = {Chiu-Chu Melissa Liu},
	number = {1-2},
	pages = {1--28},
	title = {Disk instantons, mirror symmetry and the duality web},
	volume = {57},
	year = {2002}}

@article{AKMV05,
	author = {Aganagic, Mina and Klemm, Albrecht and Mari\~{n}o, Marcos and Vafa, Cumrun},
	fjournal = {Communications in Mathematical Physics},
	issn = {0010-3616},
	journal = {Comm. Math. Phys.},
	mrclass = {81T45 (14N35 81T30)},
	mrnumber = {2117633},
	mrreviewer = {Chiu-Chu Melissa Liu},
	number = {2},
	pages = {425--478},
	title = {The topological vertex},
	volume = {254},
	year = {2005}}

@article{HV00,
	author = {Kentaro Hori and Cumrun Vafa},
	archivePrefix = "arXiv",
	eprint = {hep-th/0002222},
	journal = {HUTP-00/A005},
	title = {Mirror Symmetry},
	year = {2000}}

@article{GP99,
	author = {Graber, T. and Pandharipande, R.},
	fjournal = {Inventiones Mathematicae},
	issn = {0020-9910},
	journal = {Invent. Math.},
	number = {2},
	pages = {487--518},
	title = {Localization of virtual classes},
	volume = {135},
	year = {1999}}

@article{Ross14,
	author = {Ross, Dustin},
	fjournal = {Transactions of the American Mathematical Society},
	issn = {0002-9947},
	journal = {Trans. Amer. Math. Soc.},
	number = {3},
	pages = {1587--1620},
	title = {Localization and gluing of orbifold amplitudes: the {G}romov-{W}itten orbifold vertex},
	volume = {366},
	year = {2014}}

@article{LLLZ09,
	author = {Li, Jun and Liu, Chiu-Chu Melissa and Liu, Kefeng and Zhou, Jian},
	fjournal = {Geometry \& Topology},
	issn = {1465-3060},
	journal = {Geom. Topol.},
	number = {1},
	pages = {527--621},
	title = {A mathematical theory of the topological vertex},
	volume = {13},
	year = {2009}}

@article{Jiang08,
	author = {Jiang, Yunfeng},
	fjournal = {Illinois Journal of Mathematics},
	issn = {0019-2082},
	journal = {Illinois J. Math.},
	mrclass = {14F40 (14A20)},
	mrnumber = {2524648},
	number = {2},
	pages = {493--514},
	title = {The orbifold cohomology ring of simplicial toric stack bundles},
	volume = {52},
	year = {2008}}

@incollection{Givental98,
	author = {Givental, Alexander},
	booktitle = {Topological field theory, primitive forms and related topics ({K}yoto, 1996)},
	mrclass = {14M25 (14J32 14N35)},
	mrnumber = {1653024},
	mrreviewer = {Charles F. Doran},
	pages = {141--175},
	publisher = {Birkh\"{a}user Boston, Boston, MA},
	series = {Progr. Math.},
	title = {A mirror theorem for toric complete intersections},
	volume = {160},
	year = {1998}}

@article{FMN10,
	author = {Fantechi, Barbara and Mann, Etienne and Nironi, Fabio},
	fjournal = {Journal f\"{u}r die Reine und Angewandte Mathematik. [Crelle's Journal]},
	issn = {0075-4102},
	journal = {J. Reine Angew. Math.},
	mrclass = {14M25 (14D23)},
	mrnumber = {2774310},
	mrreviewer = {Hsian-Hua Tseng},
	pages = {201--244},
	title = {Smooth toric {D}eligne-{M}umford stacks},
	volume = {648},
	year = {2010}}

@article{FLT22,
	author = {Fang, Bohan and Liu, Chiu-Chu Melissa and Tseng, Hsian-Hua},
	fjournal = {Forum of Mathematics. Sigma},
	journal = {Forum Math. Sigma},
	mrclass = {14N35 (14J33 14M25 53D37 53D45)},
	mrnumber = {4458543},
	mrreviewer = {William Liu},
	pages = {Paper No. e58, 56},
	title = {Open-closed {G}romov-{W}itten invariants of 3-dimensional {C}alabi-{Y}au smooth toric {DM} stacks},
	volume = {10},
	year = {2022}}

@article{FL13,
	author = {Fang, Bohan and Liu, Chiu-Chu Melissa},
	fjournal = {Communications in Mathematical Physics},
	issn = {0010-3616},
	journal = {Comm. Math. Phys.},
	mrclass = {14N35 (14J33 53D12 53D37 53D45 81T30)},
	mrnumber = {3085667},
	mrreviewer = {Bhupendra Nath Tiwari},
	number = {1},
	pages = {285--328},
	title = {Open {G}romov-{W}itten invariants of toric {C}alabi-{Y}au 3-folds},
	volume = {323},
	year = {2013}}

@article{CCIT15,
	author = {Coates, Tom and Corti, Alessio and Iritani, Hiroshi and Tseng, Hsian-Hua},
	fjournal = {Compositio Mathematica},
	issn = {0010-437X},
	journal = {Compos. Math.},
	mrclass = {14N35 (14A20 14J33 14M25 53D45)},
	mrnumber = {3414388},
	mrreviewer = {Ruifang Song},
	number = {10},
	pages = {1878--1912},
	title = {A mirror theorem for toric stacks},
	volume = {151},
	year = {2015}}

@article{CCK15,
	author = {Cheong, Daewoong and Ciocan-Fontanine, Ionu\c{t} and Kim, Bumsig},
	fjournal = {Mathematische Annalen},
	issn = {0025-5831},
	journal = {Math. Ann.},
	mrclass = {14D20 (14D23 14N35)},
	mrnumber = {3412343},
	mrreviewer = {Alfonso Zamora},
	number = {3-4},
	pages = {777--816},
	title = {Orbifold quasimap theory},
	volume = {363},
	year = {2015}}

@incollection{CR02,
	author = {Chen, Weimin and Ruan, Yongbin},
	booktitle = {Orbifolds in mathematics and physics ({M}adison, {WI}, 2001)},
	mrclass = {53D45 (14N35)},
	mrnumber = {1950941},
	mrreviewer = {Ignasi Mundet-Riera},
	pages = {25--85},
	publisher = {Amer. Math. Soc., Providence, RI},
	series = {Contemp. Math.},
	title = {Orbifold {G}romov-{W}itten theory},
	volume = {310},
	year = {2002}}

@article{BCS05,
	author = {Borisov, Lev A. and Chen, Linda and Smith, Gregory G.},
	fjournal = {Journal of the American Mathematical Society},
	issn = {0894-0347},
	journal = {J. Amer. Math. Soc.},
	mrclass = {14N35 (14C15 14M25)},
	mrnumber = {2114820},
	mrreviewer = {Domenico Fiorenza},
	number = {1},
	pages = {193--215},
	title = {The orbifold {C}how ring of toric {D}eligne-{M}umford stacks},
	volume = {18},
	year = {2005}}

@article{AV00,
	author = {Mina Aganagic and Cumrun Vafa},
	archivePrefix = "arXiv",
	eprint = {hep-th/0012041},
	title = {Mirror Symmetry, {$D$}-Branes and Counting Holomorphic Discs},
	year = {2000}}

@unpublished{FLYZ-crepant,
	author = {Fang, Bohan and Liu, Chiu-Chu Melissa and Yu, Song and Zong, Zhengyu},
	note = {In preparation},
	title = {Topological recursion, crepant transformation conjecture, and holomorphic anomaly equations}}

@article{CoatesIritaniTseng-crepant,
	author = {Coates, Tom and Iritani, Hiroshi and Tseng, Hsian-Hua},
	fjournal = {Geometry \& Topology},
	issn = {1465-3060,1364-0380},
	journal = {Geom. Topol.},
	mrclass = {53D45 (14J33 14N35 53D37)},
	mrnumber = {2529944},
	mrreviewer = {Yunfeng\ Jiang},
	number = {5},
	pages = {2675--2744},
	title = {Wall-crossings in toric {G}romov-{W}itten theory. {I}. {C}repant examples},
	volume = {13},
	year = {2009}}

@article{CoatesRuan13-crepant,
	author = {Coates, Tom and Ruan, Yongbin},
	fjournal = {Universit\'e{} de Grenoble. Annales de l'Institut Fourier},
	issn = {0373-0956,1777-5310},
	journal = {Ann. Inst. Fourier (Grenoble)},
	mrclass = {53D45 (14N35 83E30)},
	mrnumber = {3112518},
	mrreviewer = {Hans-Bert\ Rademacher},
	number = {2},
	pages = {431--478},
	title = {Quantum cohomology and crepant resolutions: a conjecture},
	volume = {63},
	year = {2013}}

@incollection{BryanGraber-crepant,
	author = {Bryan, Jim and Graber, Tom},
	booktitle = {Algebraic geometry---{S}eattle 2005. {P}art 1},
	isbn = {978-0-8218-4702-2},
	mrclass = {14N35 (14E15)},
	mrnumber = {2483931},
	mrreviewer = {Hsian-Hua\ Tseng},
	pages = {23--42},
	publisher = {Amer. Math. Soc., Providence, RI},
	series = {Proc. Sympos. Pure Math.},
	title = {The crepant resolution conjecture},
	volume = {80, Part 1},
	year = {2009}}

@article{Yu25,
author = {Song Yu},
title = {{The Open Crepant Transformation Conjecture for toric Calabi-Yau 3-orbifolds}},
volume = {130},
journal = {Journal of Differential Geometry},
number = {1},
publisher = {Lehigh University},
pages = {27 -- 70},
year = {2025},
doi = {10.4310/jdg/1747062764},
URL = {https://doi.org/10.4310/jdg/1747062764}
}

@article{Zhou09b,
	author = {Jian Zhou},
	archivePrefix = "arXiv",
	eprint = {0911.2343},
	title = {Local Mirror Symmetry for the Topological Vertex},
	url = {https://arxiv.org/pdf/0911.2343.pdf},
	year = {2009}}

@article{Zhou09a,
	author = {Jian Zhou},
	archivePrefix = "arXiv",
	eprint = {0910.4320},
	title = {Local Mirror Symmetry for One-Legged Topological Vertex},
	url = {https://arxiv.org/pdf/0910.4320.pdf},
	year = {2009}}

@article{Yu24,
	author = {Yu, Song},
	doi = {10.1007/s00220-024-05077-5},
	fjournal = {Communications in Mathematical Physics},
	issn = {0010-3616},
	journal = {Comm. Math. Phys.},
	mrclass = {14N35},
	mrnumber = {4790517},
	number = {9},
	pages = {Paper No. 219, 34},
	title = {Open/closed {BPS} correspondence and integrality},
	url = {https://doi.org/10.1007/s00220-024-05077-5},
	volume = {405},
	year = {2024}}

@article{FLYZ25,
	author = {Fang, Bohan and Liu, Chiu-Chu Melissa and Yu, Song and Zong, Zhengyu},
	archivePrefix = "arXiv",
	eprint = {2504.15696},
	title = {Remodeling Conjecture with descendants},
	year = {2025}}

\end{document}